\documentclass[a4paper,reqno]{amsart}
\usepackage[english]{babel}
\usepackage{amsmath,amsthm,amssymb,amsfonts}

\usepackage{enumitem}

\usepackage{hyperref}
\usepackage{mathtools}
\usepackage[utf8]{inputenc}
\hypersetup{
 colorlinks,
 linkcolor={blue!90!black},
 citecolor={red!80!black},
 urlcolor={blue!50!black}
}
\usepackage{tikz}
\usepackage{tikz-cd}
\usepackage{graphicx}
\usepackage[all,cmtip]{xy}
\usepackage{mleftright}
\usepackage{booktabs}
\usepackage{cite}
\usepackage{amsfonts}

\newtheorem{theorem}{\sc \textbf{Theorem}}[section]  
\newtheorem{proposition}[theorem]{\sc \textbf{Proposition}}   
\newtheorem{corollary}[theorem]{\sc \textbf{Corollary}}        
\newtheorem{lemma}[theorem]{\sc \textbf{Lemma}}

\renewcommand{\approx}{ \asymp}

\theoremstyle{remark}
\newtheorem{definition}[theorem]{\sc \textbf{Definition}}

\newtheorem{remark}[theorem]{\sc \textbf{Remark}}

\newcommand{\Ls}{\mathcal{L}}

\def\bX{{\mathbf X}}

\numberwithin{equation}{section}

\newcommand{\N}{\mathbb{N}}
\newcommand{\Z}{\mathbb{Z}}
\newcommand{\R}{\mathbb{R}}
\newcommand{\C}{\mathbb{C}}

\def\e{\mathrm{e}}

\def\supp{\operatorname{supp}}

\newcommand{\dd}{d}
\newcommand{\dm}{\mathbf{d}}
\newcommand{\nm}{\mathbf{n}}

\newcommand{\dell}{\dot{\ell}}

\newcommand{\dB}{\dot{B}}
\newcommand{\dF}{\dot{F}}
\newcommand{\dX}{\dot{X}}
\newcommand{\dL}{\dot{L}}

\newcommand{\x}{\underline{x}}
\newcommand{\y}{\underline{y}}
\newcommand{\X}{\mathbf{X}}

\newcommand{\Ss}{\mathcal{S}}
\newcommand{\Ps}{\mathcal{P}}

\makeatletter
\@namedef{subjclassname@1991}{Mathematics Subject Classification}
\usepackage{enumerate}
\graphicspath{{images/}}
\begin{document}
\author[Tommaso Bruno]{Tommaso Bruno}
\address{Department of Mathematics: Analysis, Logic and Discrete Mathematics, Ghent University,
Krijgslaan 281, 9000 Ghent, Belgium}
\email{tommaso.bruno@ugent.be}
\keywords{Dirichlet space; Besov spaces; Triebel Lizorkin spaces; homogeneous spaces; algebra properties; Nilpotent Lie groups; Grushin operators.}
\subjclass{Primary 46E35, 58J35, 43A85; Secondary 46E36, 46F10, 22E25}
\thanks{The author acknowledges
  support by the Research Foundation – Flanders (FWO) through the
  postdoctoral grant 12ZW120N}
\title[Homogeneous algebras via heat kernel estimates]{Homogeneous algebras via heat kernel estimates}
\begin{abstract}
We study homogeneous Besov and Triebel--Lizorkin spaces defined on doubling metric measure spaces in terms of a self-adjoint operator whose heat kernel satisfies Gaussian estimates together with its derivatives. When the measure space is a smooth manifold and such operator is a sum of squares of smooth vector fields, we prove that their intersection with $L^\infty$ is an algebra for pointwise multiplication. Our results apply  to nilpotent Lie groups and Grushin settings.
\end{abstract}
\maketitle

\section{Introduction}
The continuity of pointwise multiplication within spaces of functions, i.e.\ their algebra property, has been proven to be a fundamental tool in the study of nonlinear PDEs. Homogeneous Besov and Triebel–Lizorkin spaces are  among the spaces of most interest; but though their properties have been widely investigated on $\R^d$ and their study has been extended to a variety of settings~\cite{GKKP1, GKKP2, WHHY, BBD, LYY}, their algebra properties have not yet been fully understood in non-Euclidean contexts. Some results have been obtained for Sobolev and Besov spaces on doubling Riemannian manifolds and Lie groups of polynomial growth~\cite{BBR, CRTN, GS}, and for Sobolev spaces on doubling metric measure spaces endowed with a \textit{carré du champ}~\cite{BF,BCF}. The main purpose of this paper is to study Besov and Triebel--Lizorkin spaces on doubling metric measure spaces of infinite measure, endowed with a self-adjoint operator whose heat kernel satisfies Gaussian estimates together with its derivatives. When the measure space is a smooth manifold and such operator is a sum of squares of smooth vector fields, we prove that their intersection with $L^\infty$ is an algebra. Noncompact nilpotent Lie groups and the Grushin setting belong to this class.   

Rather than functions, the elements of homogeneous spaces are distributions modulo polynomials. These do not require a significantly different approach if the norms are defined by means of Littlewood--Paley decompositions, which make use of multipliers supported away from the origin. Properties such as frame and atomic decompositions, among others, have been obtained in this way in a remarkably general setting~\cite{GKKP1, GKKP2, BBD, BD, LYY}. Nevertheless,~\cite{BF,BCF} and recent results for inhomogeneous spaces on Lie groups~\cite{Fe, BPTV, BPV1} reveal that a powerful tool to obtain algebra properties are paraproducts expressed in terms of noncompactly supported multipliers, in particular heat semigroups of some underlying self-adjoint operator. Under this point of view, the possibility of expressing the norm of these spaces in terms of such semigroups is of fundamental importance. However, the image of a distribution modulo polynomials through a multiplier containing the origin in its support is not well defined as a distribution, as it might depend on the representative of the class. Though a characterization of homogeneous Triebel--Lizorkin and Besov norms in terms of noncompactly supported multipliers has been known on $\R^d$ for long time~\cite{Triebel}, to the best of our knowledge this issue has remained quite unexplored in other contexts. In the Euclidean setting itself the picture has not yet been completely clarified; see the recent paper~\cite{MPS} for the case of  Sobolev spaces. We shall elaborate on this later on.

A possible approach to circumvent this problem is to define homogeneous spaces as the closure of suitable test functions with respect to homogeneous norms. This is the case,  e.g., of~\cite{BCF, BF, CRTN}. Nevertheless, as shown in~\cite{MPS}, the drawback of such a definition is that it might give rise to spaces of distributions modulo polynomials whose degree depends on the order of regularity of the space. For this reason we maintain the classical approach, and following a theory which has been recently developed in~\cite{KP} we work with distributions and their equivalence classes modulo polynomials.

We shall first consider doubling metric measure spaces of infinite measure, endowed with a self-adjoint operator $\Ls$ whose heat kernel satisfies Gaussian estimates together with its time derivatives. A quite general framework where these estimates are available is that of strictly local regular Dirichlet spaces of Harnack type. We shall define homogeneous Besov and Triebel--Lizorkin spaces, $\dB^{p,q}_\alpha$ and $\dF^{p,q}_\alpha$ respectively, by means of the heat semigroup of $\Ls$, and prove equivalent characterizations of their norms and embeddings. In order to obtain algebra properties for all regularities, it seems appropriate to restrict to differentiable structures: our framework will then be a smooth manifold, and $\Ls$ will be a sum of squares of smooth vector fields whose heat kernel satisfies Gaussian estimates together with its space derivatives. In this setting, we prove that $L^\infty \cap \dB^{p,q}_\alpha$ and $L^\infty \cap \dF^{p,q}_\alpha$, suitably interpreted, are algebras for pointwise multiplication whenever $\alpha> 0$, $q\in [1,\infty]$, and $p\in [1,\infty]$ or $p\in (1,\infty)$ respectively. We shall also  show interpolation properties, representation formulae, and discuss the equivalence of our definition with those which make use of Littlewood--Paley decompositions. The two main examples that we shall consider here are nilpotent Lie groups and Grushin settings. In the case of the Grushin operator, we shall prove estimates for the derivatives of its heat kernel which appear to be new and of independent interest. Our results insert in between the aforementioned~\cite{BCF, BF},~\cite{KP, GKKP1}, and recent results of Peloso, Vallarino and the present author~\cite{BPV1, BPV2,BPV3} where a theory of inhomogeneous function spaces on Lie groups was developed. We shall exploit and combine many strengths of the three theories by remaining essentially self-contained.

Let us finally mention that the range of integrability indices $p,q$ for $\dB^{p,q}_\alpha$ and $\dF^{p,q}_\alpha$ will always lie in $ [1,\infty]$. The study of the case $p,q\in (0,1)$, which seems to require substantially different techniques, is left to future work. Results in the spirit of this paper, for a limited range of indices and regularities but under weaker geometric assumptions, are also being object of investigation~\cite{BPRV}.

\smallskip

The structure of the paper is as follows. In the remainder of the introduction, we describe  the setting in detail and introduce some convenient notation. In Section~\ref{sec:distrL} we introduce distributions, polynomials and prove some auxiliary results about the heat semigroup of $\Ls$. In Section~\ref{BFspaces} we define Besov and Triebel--Lizorkin spaces and obtain equivalent characterizations of their norms. We prove embeddings in Section~\ref{sec:embed} and algebra properties in Section~\ref{sec:AP}. Section~\ref{sec:SS} is devoted to Calder\'on-type representation formulae and interpolation properties.  In Section~\ref{sec:inho} we discuss a parallel theory of inhomogeneous spaces, and the final Section~\ref{sec:example} shows how nilpotent Lie groups and Grushin settings are particular instances of the framework of the paper.

\subsection{Setting of the paper }
We denote by $(M, d,\mu )$ a second countable, locally compact, connected metric measure space of infinite volume,  whose measure $\mu$ is Radon, positive, noncollapsing and doubling:  in particular, there exist constants $c_0,\dm>0$ such that
\begin{equation*}
\inf_{x\in M} \mu(B(x,1))\geq  c_0^{-1}, \qquad \mu(B(x,R r))\leq c_0 R^\dm \mu(B(x,r))
\end{equation*}
for all $x\in M$, $r>0$ and $R\geq 1$. Here $B(x,r)$ stands for the ball centred at $x$ with radius $r$ with respect to the distance $d$. The constant $\dm$ might be called the dimension of $M$ though, as we shall see later on, when $M$ is a manifold $\dm$ might not be its topological dimension. In very limited circumstances, we shall also assume a stronger noncollapsing condition on $\mu$, but this will be discussed in due course.

Let $L^2(\mu)$ be the space of real-valued $L^2$ functions on $M$. We assume the existence of a non-negative self-adjoint operator $\Ls$ on $L^2(\mu)$, mapping real-valued to real-valued functions, which generates a positivity preserving, contractive and symmetric semigroup $\e^{-t\Ls} = T_t$ on $L^2(\mu)$ such that $T_t 1=1$ for all $t>0$. Its integral (heat) kernel $H_t(\cdot ,\cdot )$ is assumed to be H\"older continuous, satisfy Gaussian two-sided bounds, and with time derivatives satisfying Gaussian upper bounds. More precisely, we assume that there exist $C,b_0,b_1>0$ such that, for $x,y\in M$ and $t>0$,
\begin{equation}\label{Htestimate}
C^{-1} \frac{\e^{-b_0 \, d^2(x,y)/t}}{(\mu(B(x,\sqrt t)) \mu(B(y,\sqrt t)))^{1/2}} \leq H_t(x,y) \leq C \frac{\e^{-b_1 \, d^2(x,y)/t}}{(\mu(B(x,\sqrt t)) \mu(B(y,\sqrt t)))^{1/2}},
\end{equation}
that there exist $a, b_2>0$ such that, for $x,y,y'\in M$ and $t>0$ with $d(y,y')\leq \sqrt{t}$,
\[
| H_t(x,y) - H_t(x,y') |
\leq C \,   \left( \frac{d(y,y')}{\sqrt{t}}\right)^a \frac{\e^{-{b_2 \, d^2(x,y)}/t}}{(\mu(B(x,\sqrt t))\mu(B(y,\sqrt t)))^{1/2}},
\]
and that for every $k\in \{0,1,\dots\}$ there exist positive constants $C = C_{k}$ and $\delta=\delta_{k}$ such that, for $x,y\in M$ and $t>0$,
\begin{equation}\label{dtHtestimate}
|\partial_t^{k} H_t(x,y)| \leq C t^{-k} \frac{\e^{-\delta \, d^2(x,y)/t}}{(\mu(B(x,\sqrt t))\mu(B(y,\sqrt t)))^{1/2}}. 
\end{equation}
As already mentioned, the above estimates are realized e.g.\ by the heat kernel on strictly local regular Dirichlet spaces of Harnack type with a complete intrinsic metric~\cite[Section 2.3.2]{GSC}. In many circumstances, the H\"older continuity assumption is redundant, as~\cite[Theorem 5.11 and Corollary 7.6]{GT} show. Observe moreover that our framework satisfies the assumptions of~\cite{KP, GKKP1}, and that by~\cite[Theorem 1.4.1]{Davies} the semigroup $T_t$ can be extended to a contraction semigroup on $L^p$, $1\leq p\leq \infty$, strongly continuous when $1\leq p<\infty$.

\begin{remark}
Due to the Gaussian term, and since there is $C>0$ such that
\[
\mu(B(y,\sqrt{t})) \leq C (1+d(x,y)/\sqrt{t})^\dm \mu(B(x,\sqrt{t}))
\]
for all $x,y\in M$ and $t>0$ because of the doubling condition, up to changing the constants involved one might replace $(\mu(B(x,\sqrt t)) \mu(B(y,\sqrt t)))^{1/2}$ by $\mu(B(x,\sqrt{t}))$ in~\eqref{Htestimate} and~\eqref{dtHtestimate}. 
\end{remark}

\begin{remark}\label{noncollaps}
Since $M$ is connected, the reverse doubling property of $\mu$ holds~\cite[Proposition 2.2]{CKP}. This implies the  existence of some $\dm^*$, $0<\dm^*\leq \dm$, and $c>0$ such that $\mu(B(x,Rr))\geq c R^{\dm^*} \mu(B(x,r))$ for all $x\in M$, $r>0$ and $R\geq 1$; together with the noncollapsing condition of $\mu$, this implies  $\mu(B(x,r))\geq c\, c_0^{-1} r^{\dm^*} $ for all $r\geq 1$. The noncollapsing and the doubling conditions, instead, lead to $\mu(B(x,r))\geq c_0^{-2} r^{\dm} $ for all $r\leq  1$.
\end{remark}

\subsection{Notation} \label{Sub:not}
For $p\in [1,\infty]$, we shall denote by $L^p(\mu)$, or simply by $L^p$, the usual Lebesgue spaces endowed with their usual norm $\|\cdot \|_p$. For $p\in [1,\infty]$, $q\in [1,\infty)$ and a measurable function $F\colon  (0,\infty) \times M \to \R$, we write
\[
\|F\|_{L^p(L^q_+)} =\bigg\| \left( \int_0^\infty \! |F(t, \cdot )|^q \, \frac{\dd t}{t} \right)^{\! 1/q}\bigg\|_p, \qquad \|F\|_{L^q_+(L^p)} = \bigg( \int_0^\infty \!\left\| F(t, \cdot )\right\|_p^q  \, \frac{\dd t}{t} \bigg)^{\! 1/q},
\]
and define the spaces $L^p(L^q_+)$ and $L^q_+(L^p)$ accordingly.
When $q=\infty$, we shall mean $L^\infty_+ = L^\infty((0,\infty))$. We shall also use their discrete counterparts
\[
\|F\|_{L^p(\dell^q)} =\bigg\| \bigg( \sum_{j\in \Z}  |F(2^j, \cdot )|^q \bigg)^{1/q}\bigg\|_p, \qquad \|F\|_{\dell^q(L^p)} = \bigg( \sum_{j\in\Z} \| F(2^j, \cdot )\|_p^q \bigg)^{1/q},
\]
and analogously $L^p(\dell^\infty)$ and $\dell^{\infty}(L^p)$. We shall often display explicitly the dependence on the $t$ variable in the norms: we shall equivalently write $\|F\|_{L^p(L^q_+)}$ and $\|F(t,\cdot)\|_{L^p(L^q_+)}$, and similarly in the other cases. To denote either $L^p(L^q_+)$ or $L^q_+(L^p)$ when no distinction is needed, we shall write $X^{p,q}_+$. We shall always consider real-valued functions.

For two positive quantities $A$ and $B$, we write $A\lesssim B$ to indicate that there exists a constant $c>0$ such that $A\leq c \,B$. If $A\lesssim B$ and $B\lesssim A$, we write $A\approx B$.  We denote by $C<\infty$, or $c > 0$, a constant that may vary from place to place but is independent of significant quantities. If $\beta\geq 0$, we denote its integer part by $[\beta]$.

\section{Distributions, polynomials and the heat semigroup}\label{sec:distrL}

We recall here the theory of distributions for $\Ls$ settled in~\cite{KP}, to which we refer the reader for all the details, and prove some further results. We fix once and for all a reference point $x_0\in M$, we write $d(x)= d(x,x_0)$ for all $x\in M$, and define $\Ss(\Ls)$, or simply $\Ss$, as
\begin{equation}\label{defS}
\Ss= \left\{ \phi\in L^2\colon \Ls^n \phi \in L^2, \; p_{n}(\phi)<\infty \; \; \forall n\in \N\right\}
\end{equation}
where
\[
p_{n}(\phi) = \big\| ( 1+ d(\cdot ))^n \max_{0\leq k \leq n}|\Ls^k\phi| \big\|_{\infty} , \quad n\geq 0.
\]
It is to observe that $\Ss$ is a Fréchet space. We denote by $\Ss'$ its topological dual with the weak topology. We shall refer to the elements of $\Ss'$ as distributions, and write the action of $f\in \Ss'$ on $\phi \in \Ss$ as $\langle f, \phi\rangle$. Any operator $T$ which can be defined on and preserves $\Ss$ may also be defined by duality on $\Ss'$, as usual, as $\langle Tf, \phi\rangle=  \langle f,T\phi\rangle$. We recall that if $\varphi \in \Ss(\R)$ is real valued and even, then $\varphi (\sqrt{\Ls})$ is continuous on $\Ss$ and $\Ss'$ by~\cite[Theorem 3.4]{KP}, see also~\cite[Theorem 2.2]{GKKP1}; and that by~\cite[Proposition 3.2]{GKKP1}, if $f\in \Ss'$, then $\varphi (\sqrt{\Ls}) f$ is a continuous and slowly growing function, where this means that there exists $h\in \N$ such that
\[
|\varphi (\sqrt{L})f(x)| \lesssim (1 + d(x))^h, \qquad x\in M.
\]
\begin{remark}\label{remark:continuity}
We observe for future convenience that in particular $\Ls^h T_t f = T_t \Ls^h f$ is continuous and slowly growing for all $t>0$, $h\in \N$ and $f\in \Ss'$. Observe also that $\Ls T_t f= -\partial_t T_t f$. 
\end{remark}

We define a (generalized) polynomial of degree $k\in \N$ as a distribution $\rho\in \Ss'$ such that $\Ls^k \rho=0$, and we write $\rho \in \Ps_k$. The space of all polynomials will be denoted by $\Ps$. We define an equivalence relation on $\Ss'$ by saying that $f\sim g$ whenever $f-g\in \Ps$, and denote by $\Ss'/\Ps$ the set of all equivalence classes in $\Ss'$. The equivalence class in $\Ss'/\Ps$ of $f\in \Ss'$ will be denoted by $[f]$. In particular, the action of $T_t$ on $\Ss'$ will be very important in the following.

\begin{lemma}\label{lem:poly}
Suppose $f\in \Ss'$. Then $T_t f \to f$ in $\Ss'$ when $t\to 0^+$.
\end{lemma}
The proof of the above lemma is implicitly contained in the proof of~\cite[Proposition 3.8]{GKKP1}, so we omit the details. In particular, we observe that if $f\in \Ss'$ and $T_t f =0$ for all $t>0$, then $f=0$.
\begin{lemma}\label{lemma:TLnorm}
Let $k\geq 1$ be an integer. If $\rho\in \Ps_{k}$, then
\[
T_t \rho = \sum_{j=0}^{k-1}\frac{ (-1)^j}{j!} (\Ls^j \rho)\,  t^j \quad \mbox{ in } \:  \Ss'.
\]
Moreover, $\Ls^j\rho$ is a continuous slowly growing function for all $j$, and the equality above holds as continuous functions.
\end{lemma}

\begin{proof}
We prove the statement by induction on $k$.  Suppose $k=1$, hence $\Ls \rho =0$. For $\phi \in \Ss$,  consider the function $u(t) = \langle T_t \rho, \phi\rangle$. Then $u'(t) = -\langle T_t\Ls \rho,\phi \rangle =0$. Thus $u(t)$ is constant, whence
\[
u(t) = \langle T_t \rho, \phi\rangle = \lim_{t\to 0^+} \,  \langle T_t \rho, \phi\rangle = \langle\rho, \phi\rangle,
\]
the last equality by Lemma~\ref{lem:poly}. Since $\phi$ is arbitrary, we obtain $ T_t \rho =  \rho$. 

Assume now that $k\geq 2$, that the statement holds for $k-1$ and suppose $\rho \in \Ps_k$. Since $\Ls \rho\in \Ps_{k-1}$, for $\phi \in \Ss$ by the inductive assumption
\[
-\frac{d}{dt} \langle T_t \rho, \phi \rangle=  \langle T_t \Ls \rho, \phi \rangle =  \sum_{j=0}^{k-2}\frac{ (-1)^j}{j!} \langle \Ls^{j+1}\rho, \phi\rangle t^j.
\]
Therefore, $ \langle T_t \rho,\phi\rangle = c +\sum_{j=1}^{k-1}\frac{ (-1)^j}{j!} \langle \Ls^j \rho, \phi\rangle t^j $ for some $c\in \R$. By taking the limit for $t\to 0$, one gets $c = \langle \rho, \phi \rangle$ and the first part of the statement follows by the arbitrariness of $\phi$. 

Let us now consider $\rho \in \Ps_k$ and conclude the proof. Since $\Ls^{k-1} \rho =T_1 \Ls^{k-1} \rho$ in $\Ss'$ and $T_1 \Ls^{k-1} \rho$ is continuous and slowly growing by Remark~\ref{remark:continuity}, so is $\Ls^{k-1}\rho$. Now observe that $\Ls^{k-2}\rho = T_1 \Ls^{k-2}\rho + \Ls^{k-1}\rho$, hence $\Ls^{k-2}\rho$ is also continuous and slowly growing. By iteration, the conclusion follows.
\end{proof}

\begin{lemma}\label{lemma:pol}
For all $\rho \in \Ps\setminus \{0\}$ there exist $k\geq 1$, $t_0>0$, a relatively compact open set $U\subset M$ and $c>0$ such that $|\mathbf{1}_{\overline{U}} T_t \rho| \geq c t^{k-1}$ for all $t\geq t_0$. In particular, if $t^{\beta} T_t \rho \in X^{p,q}_+$ for some $\beta>0$ and $p,q\in [1, \infty]$, then $\rho=0$.
\end{lemma} 

\begin{proof}
Let $k\geq 1$ be the maximal integer for which there is $x\in M$ such that $\Ls^{k-1}\rho(x) \neq 0$. Since $\Ls^{k-1}\rho$  is continuous, there exists a relatively compact open $U$ and $\epsilon>0$ such that $|\Ls^{k-1}\rho(x)| \geq \epsilon$ for all $x\in \overline{U}$. Since $\Ls^j\rho$ is continuous by Lemma~\ref{lemma:TLnorm}, it is bounded on $\overline{U}$ for all $j=0,\dots, k-2$. The first statement then follows by Lemma~\ref{lemma:TLnorm}. To conclude, notice that if $\rho \neq 0$ and $k$ is as above, then $\|t^{\beta} T_t \rho\|_{X^{p,q}_+}  \gtrsim \| \mathbf{1}_{[t_0,\infty)}(t) \mathbf{1}_{\overline{U}} t^{\beta+k-1} \|_{X^{p,q}_+} =\infty$.
\end{proof}
\subsection{Heat kernel estimates and beyond}\label{sec:consheat}
In this section we show some direct consequences of the heat kernel estimates which will be of use all over the paper. We write $b=b_0/b_1$, where $b_1$ and $b_0$ are those of~\eqref{Htestimate}.

\begin{lemma}\label{pointwiseestheat1}
The following hold.
\begin{itemize}
\item[\emph{(1)}] For every $0<c_1<c_2$, there exists $C>0$ such that $|T_t f | \leq C  \, T_{bc_2 s}|f|$  for all $s >0$, $t\in [c_1 s,c_2 s]$ and measurable functions $f\in \Ss'$.
\item[\emph{(2)}] For all $h\in \N$ there are $a_h,C>0$ such that  $|\Ls^h T_{t} f |\leq C t^{-h} T_{a_h t}|f|$ for all $t>0$ and measurable functions $f\in\Ss'$.
\item[\emph{(3)}] If $h\in\N$ and $p\in [1,\infty]$, then $ \|\Ls^h T_t g\|_p \lesssim  t^{-h} \|g\|_p$ for all $t>0$ and $g\in L^p$.
\end{itemize}
\end{lemma}
\begin{proof}
Statement (1) is a consequence of the positivity of $H_t$, which implies $|T_t f | \leq  T_t |f|$, and its two-sided estimates~\eqref{Htestimate} together with the doubling property of $\mu$. Indeed, these imply
\[
H_t(x,y)  \lesssim  H_{b c_2 s}(x,y)\qquad \forall\, x,y\in M,\; \forall\, t\in [c_1 s,c_2 s].
\]
Statement (2) follows similarly by using~\eqref{dtHtestimate},~\eqref{Htestimate} and the fact that $\mu$ is doubling. Finally, (3) is a consequence of (2) and the $L^p$-boundedness of $T_t$.
\end{proof}

\begin{lemma}\label{lemmaseminorm}
Let $n,h\geq 0$ be integers. Then $p_n(\Ls^h T_t \phi ) \lesssim  t^{-h}(1+\sqrt{t})^n p_{n}(\phi)$ for all $\phi \in\Ss$ and $t>0$.
\end{lemma}

\begin{proof}
Recall that
\[
p_n(\Ls^h T_t\phi )   =\sup_{x\in M} \max_{0 \leq k \leq n}  (1+ d(x))^n  | \Ls^h T_{t}  \Ls^k \phi |(x).
\]
Since $(1+d(x))^n \leq (1+d(x,y))^n (1+d(y))^n$, and by Lemma~\ref{pointwiseestheat1}~(1) and~(2), for $k=0,\dots, n$
\begin{align*}
(1+d(x))^n & |\Ls^h T_{t} \Ls^k\phi(x)| \\
& \leq t^{-h }[ \sup_{y\in M} (1+d(y))^n |\Ls^k \phi(y)| ] \int_M (1+d(x,y))^n H_{ct} (x,y)\, \dd \mu(y)\\
& \lesssim t^{-h } p_{n}(\phi) \|  (1+d(x,\cdot ))^n H_{ct}(x,\cdot)\|_1.
\end{align*}
We shall then estimate $\|  (1+d(x,\cdot ))^n H_t(x,\cdot)\|_1$, or rather $\|  (1+d(x,\cdot ))^n H_t(x,\cdot)\|_r$, for a general $1\leq r<\infty$ as this will be of use later on. By~\eqref{Htestimate}, one has
\begin{align*}
&\mu(B(x,\sqrt{t}))^{r}  \|  (1+d(x,\cdot ))^n H_t(x,\cdot)\|_r^r\\
&\lesssim \int_{M}  (1+d(x,y))^{rn}\e^{-cr \frac{d(x,y)^2}{t}}\, \dd \mu(y) \\
&  \lesssim    \int_{B(x,\sqrt{t})} (1+\sqrt{t})^{nr} \, \dd \mu(y) + \sum_{k\geq 0}  (1+ 2^{k}\sqrt{t})^{rn} \!\int_{B(x, 2^{k+1}\sqrt{t} )\setminus B(x, 2^{k}\sqrt{t} )} \!\e^{-cr2^{2k}}\,\dd \mu(y)\\
& \lesssim  \mu(B(x,\sqrt{t})) (1+\sqrt{t})^{nr} + \sum_{k\geq 0}    (1+ 2^{k}\sqrt{t})^{rn}\mu(B(x, 2^{k+1}\sqrt{t}) ) \e^{-cr2^{2k}}.
\end{align*}
Hence
\begin{align*}
\|  (1+d(x,\cdot ))^n H_t(x,\cdot)\|_r^r
&\lesssim  \sum_{k\geq 0}  \e^{-cr2^{2k}} (1+ 2^{k}\sqrt{t})^{rn}\frac{\mu(B(x,2^{k+1}\sqrt{t}))}{\mu(B(x,\sqrt{t}))^{r}}\\
& \lesssim \mu(B(x,\sqrt{t}))^{1-r} \sum_{k\geq 0} \e^{-cr2^{2k}} (1+ 2^{k}\sqrt{t})^{rn} 2^{k\dm},
\end{align*}
the second inequality by the doubling property of $\mu$. Observe finally that if $t\leq 1$ then
\[
1\lesssim \mu(B(x,1))= \mu(B(x, t^{-1/2}t^{1/2}))\lesssim t^{-\dm/2}\mu(B(x,\sqrt{t})),
\] 
while if $t\geq 1$
\[
\mu(B(x,\sqrt{t})) \geq \mu(B(x,1)) \gtrsim 1.
\]
Thus, for $1\leq r<\infty$,
\begin{equation}\label{normHt}
\|  (1+d(x,\cdot ))^n H_t(x,\cdot)\|_r \lesssim t^{-\dm/(2r')} \mathbf{1}_{(0,1)}(t) + t^{n/2}\mathbf{1}_{[1,\infty)}(t).
\end{equation}
In particular, 
\begin{equation}\label{norm1}
\|  (1+d(x,\cdot ))^n H_t(x,\cdot)\|_1 \lesssim (1+\sqrt{t})^n,
\end{equation}
which concludes the proof.
\end{proof}

\begin{proposition}\label{contCRTN}
Suppose $p\in(1,\infty)$, $q\in [1, \infty]$ and $0<c_1<c_2$. If $(t_j)_{j\in\Z}$ is a sequence of measurable functions such that $c_1 2^j \leq t_j(x)\leq c_2 2^{j}$ for every $j\in \Z$ and $x\in M$, then
\[
\| T_{t_j} f_j\|_{L^p(\dell^q)} \lesssim \| f_j\|_{L^p(\dell^q)}
\]
for all measurable functions $(f_j)$ in $\mathcal S'$.
\end{proposition}

\begin{proof}
 Consider the operator
\[
T f(x) = \sup\nolimits_{j\in \Z} |T_{t_j(x)}f(x)|, \qquad x\in M,
\]
which is bounded on $L^p$, since $T f(x) \leq \sup_{t>0} |T_tf(x)| \eqqcolon T^* f(x)$ and $T^*$ is bounded on $L^p$ by the maximal theorem of~\cite[p.\ 73]{Stein}. Moreover, $|T f|\leq T  |f|$. By applying~\cite[Corollary 1.23, p.\ 482]{GCRDF} to $T$ and observing that $|T_{t_j} f_j |\leq T |f_j|$, the conclusion follows when $1<p\leq q$.

Assume now that $p>q$. By Lemma~\ref{pointwiseestheat1}~(1), for every function $g\geq 0$
\[
T_{t_j(x)} g(x) \lesssim T_{b c_2 2^{j} } g(x).
\]
Moreover, since $T_t1=1$ and by Jensen's inequality, $|T_{t_j(x)}f(x)|^q \leq T_{t_j(x)}|f(x)|^q$. By the symmetry of $T_t$ on $L^2$, for all measurable functions $w$,
\begin{align*}
\Big| \int_M |T_{t_j(x)}f(x)|^q w(x) \, \dd \mu(x)\Big|
&\leq \int_M T_{t_j(x)}|f(x)|^q |w|(x) \, \dd \mu(x) \\
&\lesssim \int_M T_{b c_2 2^{j}}|f(x)|^q |w|(x) \, \dd \mu(x) \\
& = \int_M |f(x)|^q T_{b c_2 2^{j}} |w|(x) \, \dd \mu(x)\\
& \leq \int_M |f(x)|^q T^* |w|(x) \, \dd \mu(x).
\end{align*}
Thus 
\[
\Big| \int_M \Big( \sum_{j\in \Z} |T_{t_j(x)}f_j(x)|^q \Big) w(x) \, \dd \mu(x)\Big|  \lesssim \int_M  \Big( \sum_{j\in \Z} |f_j(x)|^q\Big) T^* |w|(x) \, \dd \mu(x),
\]
so that if $r$ is such that $\frac{1}{r}= 1- \frac{q}{p}$, 
\begin{align*}
\| T_{t_j} f_j\|_{L^p(\dell^q)}  & = \Big\|  \sum_{j\in \Z} |T_{t_j}f_j |^q \Big\|^{1/q}_{p/q} \\
&= \sup_{\|w\|_r=1} \Big|\int_M \Big( \sum_{j\in \Z} |T_{t_j(x)}f_j(x)|^q \Big) w(x) \, \dd \mu(x)  \Big|^{1/q}\\
& \lesssim \Big\|  \Big( \sum_{j\in \Z} |f_j|^q\Big)\Big\|_{p/q}^{1/q}  \sup_{\|w\|_r=1} \|T^*|w|\|_r \lesssim \| f_j\|_{L^p(\dell^q)} ,
\end{align*}
where we used that $T^*$ is bounded on $L^r$, as $1<r<\infty$.
\end{proof}

We now prove a proposition which may be thought of as an integral analogue of Proposition~\ref{contCRTN}. We state it in a quite general, though involved form, in order to apply it in several different circumstances. For notational convenience, we write $T_0 f =f $.

\begin{proposition}\label{CRTNintegrale}
Suppose $p\in (1,\infty)$, $q\in [1,\infty]$, $c>0$ and let $F\colon [0,\infty) \times M \to \R$ be a measurable function. Assume that there exist two functions $F_1, F_2 \colon [0,\infty) \times M \to \R$, continuous in $[0,\infty)$ for each $x\in M$ and $\delta, \epsilon \geq 0$ such that 
\begin{itemize}
\item[\emph{(1)}] $T_{ct} |F(t,x)| \lesssim T_{\delta s} |F_1( s,x)|$ for all $s>0$, $ t\in [s, 2s] $ and $x\in M$, and
\item[\emph{(2)}]  $|F_1( t,x)| \lesssim T_{\epsilon s}|F_2(s,x)|$ for all $s>0$, $ t\in [s/2, s] $ and $x\in M$.
\end{itemize}
Then $\|T_{ct} |F(t,\cdot)| \|_{L^p(L^q_+)} \lesssim \|F_2(t,\cdot)\|_{L^p(L^q_+)}$.
\end{proposition}

\begin{proof}
First, suppose $q<\infty$. By (1) with $s=2^j$, $j\in \Z$,
\begin{align*}
  \int_0^\infty \left( T_{ct} |F(t,x)|\right)^q \, \frac{\dd t}{t} &  \lesssim  \sum_{j\in \Z}\int_{2^{j}}^{2^{j+1}} \left(   T_{\delta 2^j} |F_1( 2^{j},x)|\right)^q  \, \frac{\dd t}{t} \lesssim  \sum_{j\in\Z} \left(   T_{\delta 2^j} |F_1( 2^{j},x)|\right)^q.
\end{align*}
Thus, by Proposition~\ref{contCRTN}
\[
\|T_{ct} |F(t,\cdot)| \|_{L^p(L^q_+)}  \lesssim  \|T_{\delta 2^j} |F_1(2^j,\cdot)| \|_{L^p(\dell^q)} \lesssim  \|F_1( 2^j,\cdot) \|_{L^p(\dell^q)}.
\]
We now make use of  (2) to reconstruct a continuous $L^q_+$ norm. Observe indeed that, for $j\in \Z$, by the mean value theorem if $x\in M$ there exists $s_j(x)\in [2^{j}, 2^{j+1}]$ such that 
\[
\int_{2^{j}}^{2^{j+1}}   |   F_2(  t,x)| ^q \, \frac{\dd t}{t} \approx |   F_2( s_j(x),x)|^q.
\]
Then, by (2) with $s=s_j(x)$ and $t=2^j$, and by  Proposition~\ref{contCRTN},
\begin{align*}
 \|F_1( 2^j,\cdot) \|_{L^p(\dell^q)}
&\lesssim \| T_{\epsilon s_j }F_2( s_j(\cdot),\cdot)\|_{L^p(\dell^q)}  \\
& \lesssim \| F_2(s_j(\cdot),\cdot)\|_{L^p(\dell^q)} \\
& \approx   \Big\|  \Big(\sum_{j\in\Z}\int_{2^{j}}^{2^{j+1}}   |  F_2(t,\cdot)|^q \, \frac{\dd t}{t}  \Big)^{1/q}\Big\|_p = \|F_2( t,\cdot)\|_{L^p(L^q_+)},
\end{align*}
and the conclusion follows. If $q=\infty$ the statement is simpler. Arguing as above indeed, 
\begin{align*}
\sup_{t>0}   T_{ct} |F(t,x)| 
& = \sup_{j\in \Z} \sup_{t\in [2^{j}, 2^{j+1}]}   T_{ct} |F(t,x)| \lesssim \sup_{j\in\Z}     T_{ \delta 2^j }|F_1( 2^j, x)|,
\end{align*}
so that by Proposition~\ref{contCRTN}
\begin{align*}
\|T_{ct} |F(t,\cdot)| \|_{L^p(L^q_+)}
&\lesssim \| \sup_{j\in\Z}     T_{\delta 2^j }|F_1(  2^j, \cdot)| \|_p \\
&\lesssim  \| \sup_{j\in \Z}    |F_1( 2^j, \cdot)| \|_p \\
& \lesssim  \| \sup_{j\in \Z}    T_{\epsilon 2^j} |F_2( 2^j, \cdot)| \|_p  \lesssim \| \sup_{t>0   } |F_2( t,\cdot)|\|_p.\qedhere
\end{align*}
\end{proof}
The following corollary contains the instances of Proposition~\ref{CRTNintegrale} which we shall use most.
\begin{corollary}\label{bilinearCRTNint}
Suppose $p\in(1,\infty)$, $q\in [1, \infty]$, $\beta\in \R$ and $c>0$. Then for all $f\in \Ss'$ and $\gamma \geq 0$ (with $f$ measurable function if $\gamma =0$),
\[
\|t^\beta T_{c t} |T_{\gamma t} f |\|_{L^p(L^q_+)} \lesssim \|t^\beta T_{\gamma t} f \|_{L^p(L^q_+)}.
\]
Moreover, there exists $\gamma>0$ such that  for $f,g$ in $\Ss'$
\begin{equation}\label{Ttprod}
\|t^\beta T_{c t}(|T_t f|\cdot |T_t g |)\|_{L^p(L^q_+)} \lesssim \|t^\beta T_{\gamma t} (|T_t f|)\cdot T_{\gamma t} (|T_t g|)\|_{L^p(L^q_+)} .
\end{equation}
\end{corollary}
\begin{proof}
Pick $s>0$. By Lemma~\ref{pointwiseestheat1},  if $t\in [s,2s]$ then
\[
T_{c t} | T_{\gamma t} f | \leq    T_{c t} T_{\gamma (t - \frac{s}{2} ) }| T_{\gamma\frac{s}{2}  } f| \lesssim T_{c_1 s} | T_{\gamma \frac{s}{2} } f|,
\]
and analogously if $t\in [s/2,s]$ then
\[
|  T_{\gamma \frac{t}{2} } f| \lesssim T_{c_2\gamma s} |  T_{\gamma \frac{s}{8} } f|,
\]
for some constants $c_1,c_2>0$. The conclusion then follows by Proposition~\ref{CRTNintegrale} with $F(t,\cdot) = t^\beta  T_{\gamma t} f$, $F_1(t,\cdot) = t^\beta T_{\gamma t/2}f$ and $F_2(t,\cdot) = t^\beta T_{\gamma t/8}f$.  To prove~\eqref{Ttprod}, for $t\in [ s,2s]$ one shows as above that
\begin{align*}
 T_{ct}(|T_t f|\cdot |T_t g|) &\lesssim  T_{c_3s} (T_{t-\frac{s}{2}}|T_{\frac{s}{2}} f|\cdot T_{t-\frac{s}{2}}|T_{\frac{s}{2}}  g|) \lesssim T_{c_3 s} (T_{c_4s}|T_{\frac{s}{2}} f|\cdot T_{c_4s}|T_{\frac{s}{2}}g|)
\end{align*}
for some positive constants $c_3,c_4$, and similarly that if $t\in [s/2,s]$
 \[
 T_{c_4 t}|T_{ t/2 } f|\cdot T_{c_4 t}|T_{ t/2 }g| \lesssim T_{c_5 s}|T_{c_6 s} f|\cdot T_{c_5 s}|T_{c_6 s }g|,
 \]
for some $c_5,c_6>0$. Hence the conclusion follows by Proposition~\ref{CRTNintegrale} with $F(t,\cdot ) = t^\beta |T_t f|\cdot |T_t g|$, $F_1(t,\cdot ) = t^\beta  T_{c_4t}|T_{t/2} f|\cdot T_{c_4t}|T_{t/2}g|$ and $F_2(t,\cdot ) = t^\beta  T_{c_5t}|T_{c_6 t} f|\cdot T_{c_5 t}|T_{c_6 t}g|$. 
\end{proof}

To conclude this section, we note the following  consequences of Schur's Lemma: if $0<\gamma<\eta$ and $q\in [1,\infty)$, then  for every measurable function $v \colon (0,\infty) \to (0,\infty)$ one has
\begin{equation}\label{Schurcont}
\int_{0}^\infty \left( u^\gamma \int_0^\infty  \frac{1}{(t+u)^{\eta}} v(t)\, \frac{\dd t}{t} \right)^q \, \frac{\dd u}{u} \lesssim  \int_{0}^\infty \left( t^{\gamma - \eta} v(t)\right)^q \, \frac{\dd t}{t},
\end{equation}
and given a sequence $(v_n)_{n\in \Z} \subset [0,\infty)$,
\begin{equation}\label{Schurdiscr}
\sum_{j\in\Z} \Big( 2^{j\gamma} \sum_{n\in \Z} \frac{2^{n\eta}}{(2^j+2^n)^{2\eta}} v_n \Big)^q  \lesssim \sum_{j\in\Z} \Big( 2^{j\gamma} \sum_{n\in \Z} 2^{-\max \{n,j\}\eta} v_n \Big)^q \lesssim \sum_{n\in\Z} (2^{(\gamma-\eta)n} v_n)^q,
\end{equation}
whenever the right hand sides are finite. If $q=\infty$ the same results hold with the obvious modifications. We refer the reader to~\cite[Lemma 4.3]{BPV1} for details.

\section{$\dB$- and $\dF$- spaces}\label{BFspaces}
We begin with the definition of Besov ($\dB$-) and Triebel--Lizorkin ($\dF$-) norms and spaces.
\begin{definition}\label{defBTL}
Suppose $\alpha \geq 0$, $m=[\alpha/2]+1$ and $1\leq  p,q\leq \infty$. For $[f] \in \Ss'/\Ps$ and $q<\infty$, define
\begin{align*}
\|[f]\|_{\dB_\alpha^{p,q} } & =  \inf_{\rho \in \Ps} \bigg( \int_0^\infty (t^{m-\frac{\alpha}{2}} \|T_t \Ls^m (f + \rho)\|_p)^q \,\frac{\dd t}{t}\bigg)^{1/q},\\
\|[f]\|_{\dF_\alpha^{p,q} } & =  \inf_{\rho\in \Ps} \bigg\| \left( \int_0^\infty (t^{m-\frac{\alpha}{2}}  |T_t \Ls^m (f + \rho)|)^q \,\frac{\dd t}{t}\right)^{1/q}\bigg\|_p,
\end{align*}
while if $q=\infty$,
\begin{align*}
\|[f]\|_{\dB_\alpha^{p,\infty} } & =\inf_{\rho \in \Ps}   \Big( \sup_{t>0} \, t^{m-\frac{\alpha}{2}} \|T_t \Ls^m (f + \rho)\|_p\Big),\\
\|[f]\|_{\dF_\alpha^{p,\infty} } & = \inf_{\rho\in \Ps} \Big\|\sup_{t>0} \, t^{m-\frac{\alpha}{2}}  |T_t \Ls^m (f + \rho)|\Big\|_p.
\end{align*}
The corresponding Besov and Triebel--Lizorkin spaces are respectively
\[\dB_{\alpha}^{p,q}\coloneqq \{ [f]\in \Ss'/\Ps \colon \|[f]\|_{\dB_\alpha^{p,q} } <\infty\}, \qquad \dF_\alpha^{p,q} \coloneqq \{ [f]\in \Ss'/\Ps \colon \|[f]\|_{\dF_\alpha^{p,q} } <\infty\}.\]
\end{definition}

The definitions are clearly well posed. It is less obvious, but true, that $\|\cdot \|_{\dB_\alpha^{p,q} }$ and $\| \cdot \|_{\dF_\alpha^{p,q} }$ are indeed norms on $\Ss'/\Ps$. This is what we are going to show now. We leave further remarks, and comparisons with analogous spaces defined via Littlewood--Paley decompositions, to Subsection~\ref{sub:comparison} below.

Under the notation introduced in Subsection~\ref{Sub:not}, when $m= [\alpha/2]+1$
\begin{align*}
 \| [f]\|_{\dB^{p,q}_\alpha} &= \inf_{\rho \in \Ps} \| t^{m-\alpha/2} T_t \Ls^m (f+\rho)\|_{L^q_+(L^p)},\\
 \| [f]\|_{\dF^{p,q}_\alpha} &= \inf_{\rho \in \Ps} \| t^{m-\alpha/2} T_t\Ls^m (f+\rho)\|_{L^p(L^q_+)}.
\end{align*}
We shall also use $\dX^{p,q}_\alpha$ to denote either $\dF^{p,q}_\alpha$ or $\dB^{p,q}_\alpha$ when no distinction is needed.

\begin{proposition}\label{prop:minattained} 
Suppose $[f]\in \Ss'/\Ps$ is such that, for some $\alpha\geq 0$, $p,q \in [1,\infty]$, and an integer $m> \alpha/2$,
\[
\inf_{\rho \in \Ps} \| t^{m-\alpha/2} T_t\Ls^m  (f+\rho)\|_{X^{p,q}_+}<\infty.
\]
Then there exists  $\wp_f =\wp(f,\alpha,m, p,q)\in \Ps$ such that
\[
\inf_{\rho \in \Ps} \| t^{m-\alpha/2}  T_t \Ls^m(f+\rho)\|_{X^{p,q}_+}= \| t^{m-\alpha/2} T_t \Ls^m (f+\wp_f)\|_{X^{p,q}_+}.
\]
If $[f]=[g]$ and $\wp_f$ and $\wp_g $ are as above, then $\Ls^m  (f+\wp_f) =  \Ls^m (g+\wp_g)$. 

Moreover, provided $p\in (1,\infty)$ when $X^{p,q}_+ = L^p(L^q_+)$, for all integers $n>m$
\[
\| t^{n-\alpha/2} T_t \Ls^n (f+\wp)\|_{X^{p,q}_+}= \inf_{\rho \in \Ps} \| t^{n-\alpha/2}  T_t \Ls^n(f+\rho)\|_{X^{p,q}_+}<\infty.
\]
\end{proposition}

\begin{proof}
Pick $n\geq m$, and suppose additionally that $p\in (1,\infty)$ if $n>m$ and $X^{p,q}_+ = L^p(L^q_+)$. By Lemma~\ref{pointwiseestheat1} and Corollary~\ref{bilinearCRTNint}, for all $\rho \in \Ps$,
\begin{equation}
\begin{split}\label{Mm}
\| t^{n-\alpha/2} \Ls^{n} T_t (f+\rho)\|_{X^{p,q}_+ } 
& \lesssim \| t^{m-\alpha/2} T_{ct} |\Ls^m T_{t/2} (f+\rho)|\|_{X^{p,q}_+ }\\
& \lesssim \| t^{m-\alpha/2} \Ls^m T_{t/2} (f+\rho) \|_{X^{p,q}_+ }\\
& \lesssim \| t^{m-\alpha/2} \Ls^m T_{t} (f+\rho) \|_{X^{p,q}_+},
\end{split}
\end{equation}
whence $\inf_{\rho \in \Ps} \| t^{n-\alpha/2} T_t\Ls^n  (f+\rho)\|_{X^{p,q}_+}$ is also finite.

Denote $\beta=n- \alpha/2$, and let $\rho,\varrho \in \Ps$ be such that $t^\beta T_t \Ls^{n} (f+\rho), t^\beta T_t \Ls^{n} (f+\varrho) \in X^{p,q}_+$. Then $t^\beta T_t \Ls^n (\rho-\varrho) \in X^{p,q}_+$. Since $ \Ls^n (\rho-\varrho)  \in \Ps$, by Lemma~\ref{lemma:pol} one has $\Ls^{n} (\rho-\varrho) =0$, whence the existence of $\wp_{f,n}=\wp(f,\alpha,n, p,q) \in \Ps$ such that
\[
\inf_{\rho \in \Ps} \| t^{n-\alpha/2}  T_t \Ls^n(f+\rho)\|_{X^{p,q}_+}= \| t^{n-\alpha/2} T_t \Ls^n (f+\wp_{f,n})\|_{X^{p,q}_+},
\]
and such that, if $[f]=[g]$, then $\Ls^n  (f+\wp_{f,n}) =  \Ls^n (g+\wp_{g,n})$.

We finally show that one can choose $\wp_{f,n}= \wp_{f,m}$ for all $n> m$. Indeed, by~\eqref{Mm} $t^{n-\alpha/2} T_t \Ls^{n} (f+\wp_{f,m}) \in X^{p,q}_+$; since also $t^{n-\alpha/2} T_t \Ls^{n} (f+\wp_{f,n}) \in X^{p,q}_+$, one gets as above that $\Ls^{n} \wp_{f,m}  = \Ls^{n} \wp_{f,n} $, whence the conclusion.
\end{proof}

The notation of the above proposition will be maintained throughout the paper. We shall sometimes write $\wp$ or $\wp_f$, without stressing the dependence on other parameters.
\begin{proposition}
For $1\leq p,q\leq \infty$ and $\alpha \geq 0$, $\| \cdot \|_{\dB_\alpha^{p,q} }$ and $\| \cdot \|_{\dF_\alpha^{p,q} }$ are norms.
\end{proposition}
\begin{proof}
If $\|[f]\|_{\dX_\alpha^{p,q} }=0$ and $m=[\alpha/2]+1$, then by Proposition~\ref{prop:minattained} there exists $\wp \in \Ps$ such that $T_t \Ls^m (f + \wp)=0$ for all $t>0$. Lemma~\ref{lem:poly} implies then $\Ls^m (f + \wp)=0$. Thus $f\in \Ps$ and $[f]=0$ in $\Ss'/\Ps$.
\end{proof}

\subsection{Representation formulae}
In this subsection we obtain continuous Calder\'on-type reproducing formulae, which will play the role of the classical Littlewood--Paley representation formula (see also Section~\ref{sec:SS}, and in particular Subsection~\ref{sub:comparison} below) involving heat semigroups. In all its forms,~\eqref{LPDlocal}, \eqref{HomLnf} and~\eqref{HomLnfBF}, it has a fundamental importance in our theory.  

\begin{proposition}\label{thm:decomp-SP}
Let $m\geq 1$ be an integer and suppose $f\in \Ss'$. Then for all $\tau>0$
\begin{equation}\label{LPDlocal}
f= \frac{1}{(m-1)!} \int_0^{\tau} (t\Ls)^m T_t f \, \frac{\dd t}{t} +  \sum_{k=0}^{m-1} \frac{1}{k!} (\tau \Ls)^k T_\tau f \qquad  \mbox{ in }\; \Ss'.
\end{equation}
\end{proposition}

\begin{proof}
It is enough to prove the decomposition in  $\Ss$ for a function $\phi\in\Ss$, for the case when $f\in \Ss'$ follows by duality. We start from the identity, for $\tau, u>0$,
\begin{equation*}
1 = \frac{1}{(m-1)!} \int_0^{\tau} (t u)^{m} \e^{-tu} \, \frac{\dd t}{t} + \sum_{k=0}^{m-1} \frac{1}{k!} (\tau u)^k \e^{-\tau u},
\end{equation*}
from which by functional calculus  one gets, for $\phi \in \mathcal{S}$, 
\[
\phi = \frac{1}{(m-1)!} \int_0^\tau (t\Ls)^m T_t\phi \, \frac{\dd t}{t} + \sum_{k=0}^{m-1} \frac{1}{k!} (\tau \Ls)^k T_\tau \phi,
\]
where the integral converges in $L^2$. The statement follows if we prove that for all $n$
\[
 \lim_{\epsilon \to 0} p_n\bigg( \int_0^\epsilon (t\Ls)^{m} T_t \phi \, \frac{\dd t}{t} \bigg)  = 0.
\]
This, however, is an immediate consequence of Lemma~\ref{lemmaseminorm}, since 
\[
p_n\bigg( \int_0^\epsilon (t\Ls)^{m} T_t \phi \, \frac{\dd t}{t} \bigg) \leq \int_0^\epsilon  t^{m-1} p_n(T_t \Ls^{m}\phi) \, \dd t \lesssim p_{n+m}(\phi) \int_0^\epsilon  t^{m-1} (1+\sqrt{t})^n \, \dd t,
\]
where we used that $p_n(\Ls^m \phi) \leq p_{n+m}(\phi)$. This completes the proof.
\end{proof}
Proposition~\ref{thm:decomp-SP} does not address the problem of representing a distribution $f \in \Ss'$ without inhomogeneous remainder terms. As the Euclidean setting shows, however (see e.g.~\cite[Lemma 2.6]{Bownik}), this is a rather delicate matter and some additional property of $\Ss$ seems to be required. We postpone a thorough discussion to Section~\ref{sec:SS} below. For our current purposes, it is enough to show representation formulae for elements $\Ls^n f$, with $f\in\Ss'$ or in some $\dB$- or $\dF$- space.

As the reader will soon realise, the proofs for $\dB$-norms are considerably easier than those for $\dF$-norms, and follow the same steps. For this reason, from this point on, these will be detailed only for Triebel--Lizorkin spaces, and the adaptation to the Besov case will be left to the reader. 

\begin{proposition}\label{HomLPgenbis}
If $f\in \Ss'$, then there exists $h>0$ such that, for all integers $m\geq 1$ and $n>h/2$,
 \begin{equation}\label{HomLnf}
\Ls^n f = \frac{1}{(m-1)!} \int_0^{\infty} (t\Ls)^m T_t \Ls^n f \, \frac{\dd t}{t} \qquad \mbox{ in }\; \Ss'.
\end{equation}
If $f\in \Ss'$ is such that $[f]\in \dB^{p,q}_\alpha$ for some $\alpha \geq 0$ and $p,q\in [1,\infty]$, or $[f]\in \dF^{p,q}_\alpha$ for some $\alpha \geq 0$, $p\in (1,\infty)$ and $q\in [1,\infty]$, then for all integers $m\geq 1$ and $n >\alpha/2$
 \begin{equation}\label{HomLnfBF}
\Ls^n (f+\wp_f) = \frac{1}{(m-1)!} \int_0^{\infty} (t\Ls)^m T_t \Ls^n (f+\wp_f) \, \frac{\dd t}{t} \qquad \mbox{ in }\; \Ss',
\end{equation}
where $\wp_{f} = \wp(f,\alpha, [\alpha/2]+1,p,q)$ is that of Proposition~\ref{prop:minattained}.
\end{proposition}

\begin{proof}
We first prove~\eqref{HomLnf}. Let $m\geq 1$ be an integer, and $h\in\N$ be such that 
\[
|\langle f,\phi\rangle|\leq C p_h(\phi) \qquad \forall\, \phi\in \Ss.
\]
Then observe that, for $\tau>0$, $k=0,\dots, m-1$ and $n>h/2$, by Lemma~\ref{lemmaseminorm}
\[
|\langle (\tau \Ls)^k \Ls^n T_\tau f, \phi\rangle|  \lesssim p_h( (\tau\Ls)^k \Ls^n T_\tau \phi)\lesssim \tau^{-n} (1+\tau)^{h/2} p_{n}(\phi),
\]
which implies  $(\tau \Ls)^k \Ls^n T_\tau f  \to 0 $ in $\Ss'$ when $\tau \to \infty$. The conclusion follows by Proposition~\ref{thm:decomp-SP}.

The proof of~\eqref{HomLnfBF} requires more work, and we split it in two steps. We suppose $[f]\in \dF^{p,q}_\alpha$ for some $\alpha \geq 0$, $p\in (1,\infty)$ and $q\in [1,\infty]$. We write $\nu_0= [\alpha/2] +1$, and let $m\geq 1$ and $n>\alpha/2$ be integers.

As a first step, we prove that the integral in~\eqref{HomLnfBF} converges in $\Ss'$. Since 
\[
\int_0^1 (t\Ls)^{m} T_t \Ls^n (f+\wp_f) \, \frac{\dd t}{t}  = \Ls^n\int_0^1 (t\Ls)^{m} T_t (f+\wp_f) \, \frac{\dd t}{t} \in \Ss'
\]
by Proposition~\ref{thm:decomp-SP}, it will be enough to show that the integral from $1$ to $\tau$ converges in $\Ss'$ when $\tau \to \infty$; we shall actually prove that it converges in $L^p$.

 Pick $\wp_f =\wp(f,\alpha, \nu_0, p,q)\in \Ps$. By H\"older's inequality, for $q\in (1,\infty)$ and $1\leq \sigma \leq \tau$
\begin{equation}\label{conv1tau}
\begin{split}
\bigg|\int_\sigma^\tau (t \Ls)^m \Ls^n T_t (f+ \wp_{f}) \, \frac{\dd t}{t} \bigg|
& \leq  \int_\sigma^\tau t^{m} |\Ls^{n+m-\nu_0} T_{t/2}  \Ls^{\nu_0} T_{t/2} (f+ \wp_{f}) | \, \frac{\dd t}{t}  \\
& \lesssim \int_\sigma^\tau t^{\nu_0-\alpha/2} t^{\alpha/2 - n} T_{ct}|\Ls^{\nu_0} T_{t/2} (f+ \wp_{f}) |\, \frac{\dd t}{t}  \\
& \leq   \bigg( \int_\sigma^\tau (t^{\nu_0-\alpha/2} T_{ct}|\Ls^{\nu_0} T_{t/2} (f+ \wp_{f}) | )^q\, \frac{\dd t}{t} \bigg)^{1/q}  \\
& \hspace{3cm} \times  \bigg(\int_\sigma^\tau (t^{\alpha/2 - n})^{q'}\, \frac{\dd t}{t}\bigg)^{1/q'},
\end{split}
\end{equation}
whence by Corollary~\ref{bilinearCRTNint}
\[
\bigg\|  \int_\sigma^\tau (t\Ls)^m T_t \Ls^n (f+ \wp_{f}) \, \frac{\dd t}{t}  \bigg\|_p \lesssim \bigg(\int_\sigma^\tau (t^{\alpha/2 - \nu_0})^{q'}\, \frac{\dd t}{t}\bigg)^{1/q'}  \| [f]\|_{\dF^{p,q}_\alpha},
\]
which completes our first step, since the cases $q=1, \infty$ are analogous.  

By the first step and Proposition~\ref{thm:decomp-SP}, the limit
\[
f_0 = \lim_{\tau \to \infty} \sum_{k=0}^{m-1} \frac{1}{k!} (\tau \Ls)^k \Ls^n T_\tau (f+\wp_f)
\]
exists in $\Ss'$. The second step in the proof is then showing that $f_0=0$. 

By~\eqref{HomLnf} applied to  $f+\wp_f$, one deduces that $f_0 \in \Ps$. In other words,
\[
\Ls^n (f+\wp_f) + f_0 = \frac{1}{(m-1)!} \int_0^{\infty} (s\Ls)^m T_s \Ls^n (f+\wp_f) \, \frac{\dd s}{s} 
\]
in $\Ss'$, with $f_0\in \Ps$. Therefore,
\begin{equation}\label{beforesplitting}
t^{n-\alpha/2} T_t  \Ls^n (f+\wp_f)  + t^{n-\alpha/2} T_t  f_0=\frac{t^{n-\alpha/2}}{(m-1)!} \int_0^{\infty} (s\Ls)^m \Ls^{n} T_{s+t} (f+\wp_f) \, \frac{\dd s}{s}.
\end{equation}
We shall show that $ t^{n-\alpha/2} T_t  f_0 \in L^p(L^q_+)$; this will complete the proof by Lemma~\ref{lemma:pol}.

By Lemma~\ref{pointwiseestheat1} and Corollary~\ref{bilinearCRTNint},
\begin{align}\label{equivform2}
\| t^{n-\alpha/2} \Ls^{n} T_t (f+\wp_f)\|_{L^p(L^q_+)}  & \lesssim \| t^{[\alpha/2]+1-\alpha/2} T_{ct} |\Ls^{[\alpha/2]+1} T_{t/2} (f+\wp_f)|\|_{L^p(L^q_+)}\nonumber \\
&\lesssim \| t^{[\alpha/2]+1-\alpha/2} \Ls^{[\alpha/2]+1} T_{t/2} (f+\wp_f) \|_{L^p(L^q_+)}\nonumber  \\
&\approx \| [f]\|_{\dF^{p,q}_\alpha},
\end{align}
whence the first term in the left hand side of~\eqref{beforesplitting} belongs to $L^p(L^q_+)$. We shall show that
\begin{equation}\label{mm1}
\bigg\| t^{n-\frac{\alpha}{2}} \int_0^{\infty} (s\Ls)^{m} \Ls^{n} T_{s+t} (f+\wp) \, \frac{\dd s}{s}\bigg\|_{L^p(L^q_+)} \lesssim \| t^{n+1-\frac{\alpha}{2}}  \Ls^{n+1} T_{t}  (f+\wp)\|_{L^p(L^q_+)},
\end{equation}
where the right hand side is finite by~\eqref{equivform2} with $n+1$ in place of $n$.  These two facts imply by~\eqref{beforesplitting}  that $ t^{n-\alpha/2} T_t  f_0 \in L^p(L^q_+)$, whence $f_0 =0$ by Lemma~\ref{lemma:pol}.

We split the integral in~\eqref{mm1}. By Lemma~\ref{pointwiseestheat1}
\begin{align*}
\left| \int_0^t s^{m-1}\Ls^{m+n} T_{s+t}  (f+\wp) \, \dd s\right|
&\leq \int_0^t T_{c's +t/2} | \Ls^{n+1} T_{t/2}  (f+\wp)| \, \dd s \\
& \lesssim t  T_{c t} | \Ls^{n+1} T_{t/2}  (f+\wp)|   ,
\end{align*}
the last inequality since $s\in [0,t]$. Hence by Corollary~\ref{bilinearCRTNint}
\begin{align}\label{split1}
\bigg\| t^{n-\frac{\alpha}{2}} \int_0^{t} s^m \Ls^{m+n} T_{s+t} (f+\wp) \, \frac{\dd s}{s}\bigg\|_{L^p(L^q_+)} 
&\lesssim \big\| t^{n+1-\frac{\alpha}{2}}  T_{c t} | \Ls^{n+1} T_{t/2}  (f+\wp)| \big\|_{L^p(L^q_+)} \nonumber \\
& \lesssim \big\| t^{n+1-\frac{\alpha}{2}}  \Ls^{n+1} T_{t/2}  (f+\wp) \big\|_{L^p(L^q_+)}.
\end{align}
If instead $s\geq t$, then
\[
|s^{m-1}\Ls^{m+n} T_{s+t}  (f+\wp)| \lesssim T_{cs} |\Ls^{n+1} T_{\frac{s}{2}+t}  (f+\wp)| \leq T_{t+cs} |\Ls^{n+1} T_{s/2}  (f+\wp)|,
\]
so that
\[
\left| \int_t^{\infty}s^{m-1} \Ls^{m+n} T_{s+t} (f+\wp) \, \dd s\right|  \lesssim  T_{ t} \int_t^\infty T_{cs}| \Ls^{n+1} T_{s/2}  (f+\wp)|\, \dd s.
\]
By Corollary~\ref{bilinearCRTNint} and Proposition~\ref{CRTNintegrale} (with $F(t, \cdot) =t^{n-\frac{\alpha}{2}} \int_t^\infty T_{cs}| \Ls^{n+1} T_{s/2}  (f+\wp)|\, \dd s$)  then
\begin{multline*}
\bigg\| t^{n-\frac{\alpha}{2}} \int_t^{\infty} s^{m} \Ls^{m+n} T_{s+t} (f+\wp) \, \frac{\dd s}{s}\bigg\|_{L^p(L^q_+)} \\
\lesssim \bigg\| t^{n-\frac{\alpha}{2}}  \int_{t}^\infty T_{cs} | \Ls^{n+1} T_{s/2}  (f+\wp)|\, \dd s\bigg\|_{L^p(L^q_+)} .
\end{multline*}
To conclude, observe that if $q<\infty$ then
\[
\int_0^\infty \Big( t^{n-\frac{\alpha}{2}} \int_{t}^\infty T_{cs} |\Ls^{n+1} T_{s/2}(f+\wp)|\, \dd s \Big)^q \, \frac{\dd t}{t}  = \int_0^\infty \Big( \int_0^\infty  K(s,t)g(s)\, \frac{\dd s}{s} \Big)^q \, \frac{\dd t}{t}
\]
where $g(s) = s^{n+1-\frac{\alpha}{2}}T_{cs} |\Ls^{n+1} T_{s/2} (f+\wp) |$ and $K(s,t) =(t/s)^{n-\frac{\alpha}{2}}   \mathbf{1}_{\{s\geq t\}}$. Since
\[
\sup_{t>0} \int_0 ^\infty K(s,t)\, \frac{\dd s }{s} \lesssim 1, \qquad   \qquad \sup_{s>0}\int_0 ^\infty K(s,t)\, \frac{\dd t}{t} \lesssim 1,
\]
by Schur's Lemma and Corollary~\ref{bilinearCRTNint} we get
\begin{align*}
\bigg\| t^{n-\frac{\alpha}{2}}  \int_{t}^\infty T_{cs} | \Ls^{n+1} T_{s/2} (f+\wp)|\, \dd s\bigg\|_{L^p(L^q_+)}\!\!\!\!
&\lesssim \| t^{n+1-\frac{\alpha}{2}} T_{ct}  | \Ls^{n+1} T_{t/2}  (f+\wp)|\|_{L^p(L^q_+)} \nonumber \\
& \lesssim  \| t^{n+1-\frac{\alpha}{2}}  \Ls^{n+1} T_{t/2}  (f+\wp)\|_{L^p(L^q_+)}.
\end{align*}
This together with~\eqref{split1} shows~\eqref{mm1}. The case $q=\infty$ is analogous, and the proof is complete.
\end{proof}

\subsection{Equivalent norm characterizations}\label{sub:normchar} The results proven so far allow us to show some equivalent characterizations of $\dB$- and $\dF$-norms as follows.

\begin{theorem} \label{teo-equiv1}
Suppose $\alpha\geq 0$, $q\in [1,\infty]$ and let $m>\alpha/2$ be an integer.
\begin{itemize}
\item[\emph{(i)}] If $p \in [1,\infty]$, then $\|[f]\|_{\dB_\alpha^{p,q}}  \approx \inf_{\rho \in \Ps}\| t^{m-\alpha/2} \Ls^m T_t (f+\rho)\|_{L^q_+(L^p)}$.
\item[\emph{(ii)}] If $p\in (1,\infty)$, then  $\|[f]\|_{\dF_\alpha^{p,q}} \approx \inf_{\rho\in \Ps} \| t^{m-\alpha/2} \Ls^m T_t (f+\rho)\|_{L^p(L^q_+)}$.
\end{itemize}
\end{theorem}

\begin{proof} Suppose $[f] \in \dF^{p,q}_\alpha$, let $\wp_{f} = \wp(f,\alpha, [\alpha/2]+1,p,q)$ be that of Proposition~\ref{prop:minattained} and pick an integer $m>\alpha/2$. By~\eqref{equivform2}, 
\[
\inf_{\rho \in \Ps}\| t^{m-\alpha/2} \Ls^m T_t (f+\rho)\|_{L^q_+(L^p)} = \| t^{m-\alpha/2} \Ls^{m} T_t (f+\wp_f)\|_{L^p(L^q_+)} \lesssim \| [f]\|_{\dF^{p,q}_\alpha},
\]
while by combining~\eqref{HomLnfBF} with~\eqref{mm1}, we conclude
\begin{align*}
\| t^{m-\alpha/2} \Ls^m T_t (f + \wp_f)\|_{L^p(L^q_+)} & \lesssim  \| t^{m+1-\alpha/2} \Ls^{m+1} T_t (f + \wp_f)\|_{L^p(L^q_+)}\\
&=  \inf_{\rho\in \Ps}\| t^{m+1-\alpha/2} \Ls^{m+1} T_t (f+\rho)\|_{L^p(L^q_+)},
\end{align*}
which completes the proof.
\end{proof}

\begin{theorem} \label{teo-equiv2}
Suppose $\alpha\geq   0$, $q\in [1,\infty]$, and let $m>\alpha/2$ be an integer.
\begin{itemize}
\item[\emph{(i)}]  If $p\in [1,\infty]$, then
\[
  \|[f]\|_{{\dB}_\alpha^{p,q}} \approx \inf_{\rho\in \Ps} \|  2^{j(m-\frac{\alpha}{2})} \Ls^m T_{2^j} (f+\rho)\|_{\dell^q(L^p)} =\|  2^{j(m-\frac{\alpha}{2})} \Ls^m T_{2^j} (f+\wp_f)\|_{\dell^q(L^p)}.
  \]
\item[\emph{(ii)}]  If $p\in (1, \infty)$, then 
\[
\|[f]\|_{\dF_\alpha^{p,q}} \approx  \inf_{\rho\in \Ps}  \|  2^{j(m-\frac{\alpha}{2})} \Ls^m T_{2^j} (f+\rho)\|_{L^p(\dell^q)} = \|  2^{j(m-\frac{\alpha}{2})} \Ls^m T_{2^j} (f+\wp_f)\|_{L^p(\dell^q)} . 
\]
\end{itemize}
Here $\wp_f = \wp(f,\alpha,[\alpha/2]+1, p,q)$ is that of Proposition~\ref{prop:minattained}.
\end{theorem}

\begin{proof}
In order to get (ii), it is enough to prove that for all $g\in \Ss'$ and all integers $m>\alpha/2$ one has
\[
\|  t^{m-\alpha/2} \Ls^m T_{t} g\|_{L^p(L^q_+)}  \approx \|  2^{j(m-\alpha/2)} \Ls^m T_{2^j}g\|_{L^p(\dell^q)},
\]
whenever these quantities are finite. The conclusion then follows by Theorem~\ref{teo-equiv1} and Proposition~\ref{prop:minattained}. To prove this, just observe that if $s>0$ and $t\in [s, 2s]$ then
\[
t^{m-\alpha/2 }|T_t \Ls^m g| \lesssim s^{m-\alpha/2} T_{t-\frac{s}{2}}|T_{\frac{s}{2}} \Ls^m g| \lesssim s^{m-\alpha/2} T_{\delta s}|T_{\frac{s}{2}} \Ls^m g| 
\]
for some $\delta>0$, while if $t\in [s/2, s]$ 
\[
t^{m-\alpha/2 }|T_{t/2} \Ls^m g| \lesssim s^{m-\alpha/2} T_{\frac{t}{2}-\frac{s}{8}}|T_{\frac{s}{8}} \Ls^m g| \lesssim  s^{m-\alpha/2} T_{\epsilon s}|T_{\frac{s}{8}} \Ls^m g|,
\]
for some $\epsilon >0$. We applied Lemma~\ref{pointwiseestheat1}. In the proof of Proposition~\ref{CRTNintegrale}, the reader will then precisely find the inequalities
\[
\|  t^{m-\alpha/2} \Ls^m T_{t} g\|_{L^p(L^q_+)}  \lesssim  \|  2^{j(m-\alpha/2)} \Ls^m T_{2^j}g\|_{L^p(\dell^q)}\lesssim \|  t^{m-\alpha/2} \Ls^m T_{t} g\|_{L^p(L^q_+)},
\]
which complete the proof.
\end{proof}

\section{Embeddings}\label{sec:embed}
To begin with, we observe that $\dF^{p,p}_\alpha = \dB^{p,p}_\alpha$ by definition. Moreover, if $\dL^{p}_\alpha$ is the homogeneous Sobolev space defined by the norm
\[
\| [f]\|_{\dL^{p}_\alpha}  = \inf_{\rho \in\Ps} \| \Ls^{\alpha/2}(f+\rho)\|_{p}, \qquad \alpha \geq 0, \; 1\leq p\leq \infty,
\]
then $\dF^{p,2}_\alpha = \dL^{p}_\alpha$ for $\alpha > 0$ and $1<p<\infty$ by Littlewood--Paley--Stein theory, see~\cite{Meda, Stein}. Despite this, embeddings of the form $\dF^{p,q}_{\alpha+\epsilon} \subseteq \dF^{p,q}_{\alpha} $ and $\dB^{p,q}_{\alpha+\epsilon} \subseteq \dB^{p,q}_{\alpha} $ in general do not hold.

For two results below, we shall need that $\mu$ satisfies a stronger noncollapsing condition, more precisely that there exists a constant $\nm>0$ such that 
\begin{equation}\label{strongnoncollaps}
\inf_{x\in M} \mu(B(x,r))\gtrsim r^{\nm}, \qquad r>0.
\end{equation}
This condition is slightly stronger than the reverse doubling property of $\mu$; recall Remark~\ref{noncollaps}.

Observe that~\eqref{strongnoncollaps} and the upper estimate~\eqref{Htestimate} imply that the semigroup $T_t$ is $L^1$-$L^\infty$ bounded, i.e.\ ultracontractive, with norm controlled by $t^{-\nm/2}$. For the purposes of the present subsection, one might indeed replace~\eqref{strongnoncollaps} by this ultracontractivity property of $T_t$.

We begin by observing that, for $g\in L^p$,
\begin{equation}\label{eqcontract}
\| T_t g\|_{{p_1}} \lesssim t^{\frac{\nm}{2}( \frac{1}{p_1}-\frac{1}{p})} \|g\|_p , \qquad 1\leq p\leq p_1\leq \infty, \quad t>0.
\end{equation}
This follows by the $L^1$-$L^\infty$ ultracontractivity of $T_t$ and~\cite[Proposition II.2.2]{VSCC}. 
\begin{theorem}\label{teo_embeddings}
Suppose $\alpha\geq 0$ and $q,r\in [1,\infty]$. The following embeddings hold.
\begin{itemize}
\item[\emph{(1)}] If $r\geq q$ and $p\in [1,\infty]$ or $p\in (1,\infty)$, then  $\dB_\alpha^{p,q} \subseteq \dB_{\alpha}^{p,r}$ and $\dF_\alpha^{p,q} \subseteq \dF_{\alpha}^{p,r}$ respectively.
\item[\emph{(2)}] If $p\in (1,\infty)$, then $\dB^{p,\min(p,q)}_\alpha \subseteq \dF^{p,q}_\alpha\subseteq \dB_\alpha^{p,\max(p,q)}$.
\end{itemize}
Assume further that~\eqref{strongnoncollaps} holds and that $\alpha_0,\alpha_1\geq 0$.  
\begin{itemize}
\item[\emph{(3)}] If $1\leq p_0< p_1 \leq \infty$, $\alpha_0 \geq \alpha_1$ and $\frac{\nm}{p_1}- \alpha_1 = \frac{\nm}{p_0}-\alpha_0$, then  $\dB_{\alpha_0}^{p_0,q} \subseteq \dB_{\alpha_1}^{p_1,q}$.
\item[\emph{(4)}] If $1< p_0< p_1 < \infty$, $\alpha_0 \geq \alpha_1 $ and $\frac{\nm}{p_1}- \alpha_1 = \frac{\nm}{p_0}-\alpha_0$, then  $\dF_{\alpha_0}^{p_0,q} \subseteq \dF_{\alpha_1}^{p_1,r}$.
\end{itemize}
\end{theorem}
\begin{proof}
Statement (1) follows from Theorem~\ref{teo-equiv2} and the embeddings of the $\dell^q$ spaces.

To prove (2), pick $g\in \Ss'$ and an integer $m>\alpha/2$. Suppose first $p\geq q$. By the triangle inequality  in $L^{p/q}$,
\begin{align*}
\| 2^{j(m-\alpha/2)} \Ls^m T_{2^j} g\|_{L^p(\dell^q)} 
&= \Big\| \sum_{j\in\Z}(2^{j(m-\alpha/2)} |\Ls^m T_{2^j} g| )^q \Big\|_{{p/q}}^{1/q}\\
& \leq    \Big( \sum_{j\in\Z} \|2^{qj(m-\alpha/2)}|\Ls^m T_{2^j} g|^q \|_{{p/q}}\Big)^{1/q}\\
& = \| 2^{j(m-\alpha/2)} \Ls^m T_{2^j}  g\|_{\dell^q(L^p)}.
\end{align*}
Hence by Theorem~\ref{teo-equiv2}
\begin{align*}
\| [f]\|_{\dF^{p,q}_\alpha}  
&\approx  \inf_{\rho \in \Ps} \| 2^{j(m-\alpha/2)} \Ls^m T_{2^j} (f+\rho)\|_{L^p(\dell^q)}\\
&  \lesssim   \inf_{\rho \in \Ps} \| 2^{j(m-\alpha/2)} \Ls^m T_{2^j}  (f+\rho)\|_{\dell^q(L^p)} \approx \| [f]\|_{\dB^{p,q}_\alpha} .
\end{align*}
This proves that $\dB^{p,q}_\alpha \subseteq\dF^{p,q}_\alpha$, and the embedding $\dF^{p,q}_\alpha \subseteq \dF^{p,p}_\alpha= \dB^{p,p}_\alpha $ follows by (1). Similarly, if $p< q$ then
\begin{align*}
\| 2^{j(m-\alpha/2)} \Ls^m T_{2^j}  g\|_{\dell^q(L^p)}
 & = \left\| \int_M 2^{pj(m-\alpha/2)} |\Ls^m T_{2^j}  g|^p \, \dd \mu \right\|_{\dell^{q/p}}^{1/p} \\
& \leq  \left(\int_M \| 2^{pj(m-\alpha/2)} |\Ls^m T_{2^j}  g|^p\|_{\dell^{q/p}} \, \dd \mu \right)^{1/p}\\
& =  \| 2^{j(m-\alpha/2)} \Ls^m T_{2^j} g\|_{L^p(\dell^q)},
\end{align*}
hence $\|[ f]\|_{\dB^{p,q}_\alpha}  \lesssim \| [f] \|_{\dF^{p,q}_\alpha}$, which together with $\dB^{p,p}_\alpha = \dF^{p,p}_\alpha\subseteq \dF^{p,q}_\alpha $ coming from (1) concludes the proof of (2).

Let us now consider (3). Let $m>\alpha_0/2$ be an integer. Since $\frac{\nm}{p_1} - \frac{\nm}{p_0}= \alpha_1 -\alpha_0$, by~\eqref{eqcontract} for $g\in \Ss'$
\[
\| t^{m-\alpha_1/2} \Ls^m T_t g \|_{p_1}  \lesssim \| t^{m-\alpha_{0}/2} \Ls^m T_{t/2} g\|_{p_0}
\]
hence $\| [f]\|_{\dB^{p_1,q}_{\alpha_1}}  \lesssim \| [f]\|_{\dB^{p_0,q}_{\alpha_0}}$ by a change of variables and Theorem~\ref{teo-equiv1}.

We finally prove (4). Observe that by~(1) it is enough to prove that $\dF^{p_0,\infty}_{\alpha_0}\subseteq \dF^{p_1,1}_{\alpha_1}$.  Suppose $m_0 =[\alpha_0/2]+1>\alpha_1/2$ and to simplify the notation, define
\[
\Lambda_{0,j} = 2^{j (m_0-\frac{\alpha_0}{2})}\Ls^{m_0} T_{2^{j}}, \qquad \Lambda_{1,j} =2^{j (m_0-\frac{\alpha_1}{2})} \Ls^{m_0}T_{2^{j}}, \qquad j\in \Z.
\]
Thus, for $[f] \in \dF^{p_0,\infty}_{\alpha_0}$ and $\wp(f, \alpha_0,[\alpha_0/2]+1, p_0,\infty) = \wp_{0} \in \Ps$ as in Proposition~\ref{prop:minattained}, by Theorem~\ref{teo-equiv2}
\[
\| [f]\|_{\dF^{p_0,\infty}_{\alpha_0}} \approx \| \Lambda_{0,j} (f+\wp_0)\|_{L^{p_0}(\dell^\infty)}, \qquad \| [f]\|_{\dF^{p_1,1}_{\alpha_1}} \approx  \inf_{\rho \in \Ps}\| \Lambda_{1,j} (f+\rho)\|_{L^{p_1}(\dell^1)}.
\]
Without loss of generality, we may assume that $\| [f]\|_{\dF^{p_0,\infty}_{\alpha_0}}=1$. By~\eqref{eqcontract},
\begin{align*}
\| \Lambda_{1,j} (f+\wp_0)\|_{\infty} &
= 2^{j \frac{\alpha_0-\alpha_1}{2}}  \| 2^{j (m_0-\frac{\alpha_0}{2})} T_{2^{{j-1}}} T_{2^{{j-1}}}\Ls^{m_0}(f+\wp_0)\|_{\infty}\\
& \lesssim 2^{j \frac{\alpha_0-\alpha_1}{2} - j\frac{\nm}{2p_0}}\| \Lambda_{0,j-1} (f+\wp_0)\|_{{p_0}} \\
& = 2^{- j\frac{\nm}{2p_1}}\| \Lambda_{0, j-1} (f+\wp_0)\|_{{p_0}} \lesssim 2^{- j\frac{\nm}{2p_1}},
\end{align*}
so that, for every $k\in \Z$,
\begin{equation}\label{eq0K}
\sum\nolimits_{j\geq -k} |\Lambda_{1,j} (f+\wp_0)| \lesssim \sum\nolimits_{j\geq -k} 2^{- j\frac{\nm}{2p_1}} \leq C 2^{k\frac{\nm}{2p_1}},
\end{equation}
for some $C>0$. On the other hand, one has
\begin{equation}\label{eqK+1infty}
\sum_{j\leq -k-1}  |\Lambda_{1,j} (f+\wp_0)| =  \sum_{j\leq -k-1} 2^{j \frac{\alpha_0 -\alpha_1}{2}} |\Lambda_{0,j} (f+\wp_0)| \lesssim 2^{-k\frac{\alpha_0 -\alpha_1}{2}} \sup_{j\in \Z} |\Lambda_{0,j}(f+\wp_0)|.
\end{equation}
Observe now that
\begin{align*}
\| \Lambda_{1,j}(f+\wp_0)\|_{L^{p_1}(\dell^1)}^{p_1} 
& \lesssim\int_0^\infty t^{p_1-1} \mu \Big( \Big\{ \sum\nolimits_{j\in\Z}  |\Lambda_{1,j} (f+\wp_0)|>t \Big\} \Big) \, \dd t.
\end{align*}
By~\eqref{eq0K} and~\eqref{eqK+1infty}, if  $k=k(t)$ is the largest integer such that $ C 2^{k\frac{\nm}{2p_1}}<\frac{t}{2}$,
then
\begin{align*}
\Big\{  \sum\nolimits_{j\in\Z}  |\Lambda_{1,j}  (f+\wp_0)|>t \Big\} 
&\subset \Big\{\sum\nolimits_{j\leq -k-1} |\Lambda_{1,j}  (f+\wp_0)|>\frac{t}{2}\Big\}\\
& \subset \Big\{\sup\nolimits_{j\in \N} |\Lambda_{0,j} (f+\wp_0)|>C' t 2^{k\frac{\alpha_0 -\alpha_1}{2}}\Big\},
\end{align*}
where we observe that $t 2^{k\frac{\alpha_0 -\alpha_1}{2}} \approx t^{\frac{p_1}{p_0}}$. Then,
\begin{align*}
\int_0^\infty  t^{p_1-1} & \mu \Big( \Big\{ \sum_{j\in\Z}  |\Lambda_{1,j}  (f+\wp_0)|>t \Big\} \Big) \, \dd t\\
& \lesssim \int_0^\infty  t^{p_1-1} \mu \Big( \Big\{ \sup_{j\in \N} |\Lambda_{0,j}  (f+\wp_0)|> t^{\frac{p_1}{p_0}}  \Big\} \Big) \, \dd t\\
& \lesssim \int_0^\infty  s^{p_0-1}  \mu \Big( \Big\{  \sup_{j\in \N} |\Lambda_{0,j}  (f+\wp_0)|> s\  \Big\} \Big) \, \dd s \\
& \lesssim \| \Lambda_{0,j}  (f+\wp_0) \|_{L^{p_0}(\dell^\infty)}^{p_0}  \lesssim 1,
\end{align*}
hence $\| [f]\|_{\dF^{p_1,1}_{\alpha_1}}  \lesssim \| [f]\|_{\dF^{p_0,\infty}_{\alpha_0}}  $ and the proof of (4) is complete.
\end{proof}

\section{Besov and Triebel--Lizorkin algebras}\label{sec:AP}
Within the setting above, in this section we further assume that $M$ is a smooth manifold endowed with a family $\mathfrak X = \{X_1, \dots, X_\kappa \}$  of smooth vector fields, and that $\Ls$ is their sum of squares
\[
\Ls = -\sum_{j=1}^\kappa  X_j^2.
\]
We assume that its heat kernel $H_t(x,y)$ is smooth jointly in $t>0$ and $x,y\in M$, and that all its derivatives with respect to ${\mathfrak X}$ satisfy Gaussian estimates: if  $\mathcal{I}= \{1,\dots, \kappa\}$, then for every $h,k\in \N$ there exist positive constants $C = C_{h,k}$ and $c=c_{h,k}$ such that
\begin{equation}\label{XHtestimate}
|(X_{J})_{y} (X_{I})_{x}  H_t(x,y)| \leq C t^{-\frac{h+k}{2}} \frac{\e^{-c \, d^2(x,y)/t}}{(\mu(B(x,\sqrt t))\mu(B(y,\sqrt t)))^{1/2}}, \qquad  I \in \mathcal{I}^h,\, J\in \mathcal{I}^k.
\end{equation}
If $h\in\N$ and $I\in \mathcal{I}^h$, then $X_I$ stands for $X_{I_1}\cdots X_{I_h}$; we shall also write $|I|=h$. If $h=0$, $X_I$ will be understood as the identity operator. Note that the estimates~\eqref{XHtestimate} are stronger than~\eqref{dtHtestimate}, and imply the following lemma.

\begin{lemma}\label{pointwiseestheat2}
For all $h\in \N$ there is $a_h>0$ such that for all  measurable functions $g\in\Ss'$
\begin{equation*}
|X_{J}T_{t} g |\lesssim t^{-h/2} T_{a_h t}|g|\qquad   t>0, \, J\in \mathcal{I}^h.
\end{equation*}
\end{lemma}
\begin{proof}
This is just an application of~\eqref{XHtestimate}, the heat kernel estimates~\eqref{Htestimate} and the doubling property of $\mu$.
\end{proof}

Let us define now the space 
\[
\tilde{\Ss} = \{ \phi\in L^2 \colon X_J \phi \in L^2, \; \tilde p_{n}(\phi) <\infty \; \forall n\in \N, \, \forall J\in\mathcal{I}^n \},
\]
endowed with the family of seminorms
\[
\tilde p_{n}(\phi) = \| (1+ d(\cdot ))^n \max_{0\leq |J| \leq n}|X_J\phi|\|_\infty, \quad n\in \N.
\]
In the next result we show that this space coincides with $\Ss$.

\begin{lemma} \label{lemma:Sderivate}
The spaces $\tilde \Ss$ and $\Ss$ coincide as Fréchet spaces. In particular, $X_J$ is continuous on $\Ss$ and $\Ss'$ for all $k\in\N$ and  $J\in \mathcal{I}^k$.
\end{lemma}
\begin{proof}

The continuous inclusion $\tilde \Ss \subseteq \Ss$ is immediate from $p_n(\phi) \leq \tilde p_{2n}(\phi)$, so we have only to consider the opposite inclusion. Suppose $\beta>0$. We claim that if $\beta>|J|/2$, then for all $n\in \N$ and $\psi\in \Ss$
\begin{equation}\label{claimS}
\| (1+d(\cdot))^n X_J (I+\Ls)^{-\beta} \psi\|_\infty  \lesssim \| (1+d(\cdot))^n \psi \|_\infty.
\end{equation}
Assuming the claim, if $\phi\in \Ss$ then
\begin{align*}
\tilde p_n(\phi)
& \leq \max_{0\leq |J|\leq n} \|(1+d(\cdot))^{n} X_J \phi\|_{\infty}\\
& \lesssim \|(1+d(\cdot))^{n} (I + \Ls)^{[n/2]+1} \phi\|_{\infty } \lesssim p_{n}(\phi),
\end{align*}
and this proves that $\Ss \subseteq \tilde \Ss$ (observe that since $(1+d(\cdot))^{-\dm -1} \in L^1(\mu)$, one has $X_J \phi \in L^2$ for all $J$). Therefore, it remains to prove the claim~\eqref{claimS}. 

Pick $\psi \in \Ss$. Since
\[
(I+\Ls)^{-\beta} \psi =\frac{1}{\Gamma(\beta)} \int_0^\infty t^\beta \e^{-t} T_t \psi\, \frac{\dd t}{t},
\]
by~\eqref{XHtestimate}
\begin{align*}
(1+d(x))^n &|X_J (I+\Ls)^{-\beta} \psi(x)|\\
& \leq (1+d(x))^n \int_0^\infty t^{\beta-1} \e^{-t}\int_{M} |(X_J)_x H_t(x,y)| |\psi(y)|\, \dd \mu(y) \, \dd t \\
& \lesssim \int_0^\infty t^{\beta -\frac{|J|}{2}-1} \e^{-t}  \int_{M} (1+d(x,y))^n H_t(x,y) (1+d(y))^n |\psi(y)|\, \dd \mu(y)\, \dd t,
\end{align*}
where we used that $(1+d(x))^n \leq (1+d(x,y))^n (1+d(y))^n$. Thus
\begin{align*}
(1+d(x))^n& |X_J (I+\Ls)^{-\beta} \psi(x)| \nonumber \\
& \lesssim \| (1+d(\cdot))^n \psi\|_\infty \int_0^\infty t^{\beta -\frac{|J|}{2}-1} \e^{-t}  \|  (1+d(x,\cdot ))^n H_t(x,\cdot)\|_1\, \dd t\\
&  \lesssim \| (1+d(\cdot))^n \psi\|_\infty \int_0^\infty t^{\beta -\frac{|J|}{2}-1} \e^{-t} (1+\sqrt{t})^{n}\, \dd t,
\end{align*}
the last inequality by~\eqref{norm1}. The proof is complete.
\end{proof}

By Lemma~\ref{lemma:Sderivate}, from now on we shall make no distinction between $\Ss$ and $\tilde{\Ss}$.

\begin{definition}
Suppose $1\leq p,q,r\leq \infty$ and $\alpha \geq 0$. We define
\begin{align*}
 L^r\cap \dB^{p,q}_\alpha  &= \{ f\in  L^r  \colon [f] \in \dB^{p,q}_\alpha \}, \qquad L^r \cap \dF^{p,q}_\alpha = \{f\in L^r\colon [f] \in \dF^{p,q}_\alpha \},
\end{align*}
endowed with the norms $\| f\|_ r + \|[f] \|_{\dB^{p,q}_\alpha}$ and $\| f\|_ r + \|[f] \|_{\dF^{p,q}_\alpha}$ respectively.
\end{definition}
The reader will forgive us for the slight abuse of notation, as $\dX^{p,q}_\alpha $ is a space of equivalent classes of distributions while $L^r$ is a space of functions. Such a choice is justified by the following lemma.

\begin{lemma}\label{lemmaintersec}
Suppose $p,q\in [1, \infty]$, $\alpha> 0$ and let $m>\alpha/2$ be an integer.
\begin{itemize}
\item[\emph{(i)}]  If $r\in [1,\infty]$ and $f\in L^r\cap \dB^{p,q}_\alpha$, then $\|[f]\|_{\dB_\alpha^{p,q} } \approx \|t^{m-\frac{\alpha}{2}} \Ls^m T_t f \|_{L^q_+(L^p)}$.
\item[\emph{(ii)}] If $r\in (1,\infty]$ and $f\in L^r \cap \dF^{p,q}_\alpha$, then $\|[f]\|_{\dF_\alpha^{p,q} } \approx \|t^{m-\frac{\alpha}{2}}  \Ls^m T_t f \|_{L^p(L^q_+)}$.
\end{itemize}
(In other words, under the assumptions above one can take $\wp_f=0$).
\end{lemma}
\begin{proof}
It is enough to show that, if $f\in L^r$ and $\| t^{m-\frac{\alpha}{2}}  \Ls^m T_t(f+\rho)\|_{X^{p,q}_+}<\infty$ for some $\rho \in \Ps$, then $\|t^{m-\frac{\alpha}{2}}  \Ls^m T_t (f+\rho)\|_{X^{p,q}_+} = \|t^{m-\frac{\alpha}{2}}  \Ls^m T_t f\|_{X^{p,q}_+}$. We shall prove that under the above conditions one has indeed $\Ls^m \rho =0$.

Assume that $\Ls^m \rho \neq 0$ and let $U$, $k$ and $t_0$ be as in Lemma~\ref{lemma:pol} for $\Ls^m \rho \in \Ps$. In other words, $\mathbf{1}_{U} |T_t\Ls^m \rho|\gtrsim \mathbf{1}_{U} t^{k-1}$ for $t\geq t_0$.

 Since $ t^{m-\frac{\alpha}{2}}  \Ls^m T_t (f+\rho) \in X^{p,q}_+$, in particular $\mathbf{1}_{U}\mathbf{1}_{[t_0,\infty)}(t) t^{m-\frac{\alpha}{2}}  \Ls^m T_t (f+\rho) \in X^{p,q}_+$. Moreover, by Lemma~\ref{pointwiseestheat1}
\[
\mathbf{1}_{[t_0,\infty)}(t) t^{m-\frac{\alpha}{2}}  |\Ls^m T_t f|\lesssim \mathbf{1}_{[t_0,\infty)}(t) t^{-\frac{\alpha}{2}} T_{ct} |f|.
\]
Since $f\in L^r$, this latter function belongs to $X^{r,q}_+$: by the $L^r$ boundedness of $T_t$ if $X^{r,q}_+ = L^q_+(L^r)$ or if $r=\infty$ and $X^{\infty,q}_+ = L^\infty(L^q_+)$, and by Proposition~\ref{CRTNintegrale} if $X^{r,q}_+ = L^r(L^q_+)$ and $1<r<\infty$. Then
\[
\mathbf{1}_{U}\mathbf{1}_{[t_0,\infty)}(t) t^{m-\frac{\alpha}{2}}  \Ls^m T_t f \in X^{r,q}_+, \qquad \mathbf{1}_{U}\mathbf{1}_{[t_0,\infty)}(t) t^{m-\frac{\alpha}{2}}  \Ls^m T_t (f+\rho) \in X^{p,q}_+,
\]
from which we conclude that $\mathbf{1}_{U}\mathbf{1}_{[t_0,\infty)}(t) t^{m-\frac{\alpha}{2}}  \Ls^m T_t \rho \in X^{\min(p,r),q}_+$. This contradicts the estimate $\mathbf{1}_{U}\mathbf{1}_{[t_0,\infty)}(t) |T_t\Ls^m \rho|\gtrsim \mathbf{1}_{U}\mathbf{1}_{[t_0,\infty)}(t)  t^{k-1}$.
\end{proof}

We prove now a representation formula in the same spirit as Proposition~\ref{HomLPgenbis}, but for functions in some $L^p$ space.

\begin{proposition}\label{prop:constant}
Suppose $p\in [1,\infty]$ and $f\in L^p$. Then for all integers $m,n\geq 1$ and $J\in \mathcal{I}^n$
\begin{equation}\label{LPD-constant-der}
X_J f = \frac{1}{(m-1)!}  \int_0^{\infty} t^m X_J \Ls^{m} T_t f \, \frac{\dd t}{t}  \qquad \mbox{in } \Ss'. 
\end{equation}
\end{proposition}

\begin{proof}
By~\eqref{LPDlocal}, one gets 
\[
X_J f = \frac{1}{(m-1)!}\int_0^\tau t^m  X_J\Ls^m T_t f \, \frac{\dd t}{t} + \sum_{k=0}^{m-1} \frac{1}{k!} \tau^k X_J \Ls^k T_\tau f, \qquad \tau>0.
\]
Since  $n>0$ and
\[
\|  \tau^k X_J \Ls^k T_\tau f\|_{p} \lesssim \tau^{-n/2} \|f\|_p
\]
for all $\tau>0$ and $k=0,\dots, m-1$ by Lemma~\ref{pointwiseestheat1}, one gets $\tau^k X_J \Ls^k T_\tau f \to 0$ in $\Ss'$ for $\tau \to \infty$, and~\eqref{LPD-constant-der} follows.
\end{proof}

\begin{remark} If $p \in (1,\infty)$ and $n=0$, then~\eqref{LPD-constant-der} holds in $L^p$; this follows by~\cite[Proposition 2.11]{BCF}, see also~\cite[Theorem 2.3]{BDY}. If $f\in L^\infty$ and $n=0$, one can obtain~\eqref{LPD-constant-der} in a weaker form. Namely, there exists $\rho \in \Ps$ and a sequence of $(\tau_k)_k >0$, with $\lim_{k\to\infty}\tau_k =\infty$, such that
\begin{equation}\label{LPD-constant}
f - \rho = \frac{1}{(m-1)!}  \lim_{k\to \infty}\int_0^{\tau_k} (t\Ls)^mT_t f \, \frac{\dd t}{t}  \qquad \mbox{in } \Ss'. 
\end{equation}
Indeed, observe that by~\eqref{LPDlocal}, for all $\tau>0$
\[
f -\sum_{k=0}^{m-1} \frac{1}{k!} \tau^k \Ls^k T_\tau f = \frac{1}{(m-1)!}\int_0^\tau (t\Ls)^mT_t f \, \frac{\dd t}{t}.
\]
Since   $\|  (\tau \Ls)^k T_\tau f\|_{\infty} \lesssim \|f\|_\infty$ for all $\tau>0$ and $k=0,\dots, m-1$ by Lemma~\ref{pointwiseestheat1}, the sequence
\[
\int_0^\tau (t\Ls)^mT_t f \, \frac{\dd t}{t}, \qquad \tau\in \N
\]
is uniformly bounded in $L^\infty$. By the Banach--Alaoglu Theorem, it admits a weak-* convergent subsequence, hence convergent in $\Ss'$, whose limit is in $L^\infty$. Hence, there is a sequence $(\tau_k)$ in $\N$ such that
\[
\lim_{k\to \infty} \int_0^{\tau_k} (t\Ls)^mT_t f \, \frac{\dd t}{t}
\] 
exists in $\Ss'$, and belongs to $L^\infty$. This implies that
\[
f_0 = \lim_{k\to \infty} \sum_{j=0}^{m-1} \frac{1}{j!} (\tau_k \Ls)^j T_{\tau_k}f 
\] 
exists in $\Ss'$, and by~\eqref{HomLnf} one gets $f_0\in \Ps$. Then~\eqref{LPD-constant} follows. Since $f$ and the right-hand side are in $L^\infty$, we also get $f_0 \in L^\infty$. If one further assumes that the only polynomials which are bounded are the constants, then $f_0$ reduces to a constant. We refer the reader to the next Section~\ref{sec:SS} for a more thorough discussion on homogeneous Calder\'on-type representation formulae.
\end{remark}

To prove the algebra properties, we shall also need the following result. For $f\in \Ss'$, $t>0$ and $m\in \N$ we write
\[
T^{(m),*}_{t} f =  \sup_{s\in [t,2t]}\max_{ |J|= m} |X_J T_s f| . 
\]
\begin{lemma}\label{teo-equiv3}
Suppose $\alpha>0$, $q\in [1,\infty]$ and $f\in L^r$ for some $r\in (1,\infty]$. Let $m >\alpha$ be an integer.
\begin{itemize}
\item[\emph{(i)}] If $p\in [1,\infty]$, then $ \|2^{j(m-\alpha)/2} T^{(m),*}_{{2^j}} f\|_{\dell^q(L^p)} \lesssim \|[f]\|_{\dB_\alpha^{p,q}} $.
\item[\emph{(ii)}] If  $p\in (1,\infty)$, then  $  \|2^{j(m-\alpha)/2} T^{(m),*}_{{2^j}} f \|_{L^p(\dell^q)} \lesssim \|[f]\|_{\dF_\alpha^{p,q}}$.
\end{itemize}
\end{lemma}
\begin{proof}
Let $m>\alpha$ be an integer, and pick $|J|=m$ and $t>0$. By Proposition~\ref{prop:constant}, we may write
\begin{align*}
X_J T_t f& = \frac{1}{(m-1)!}\sum_{n\in\Z} X_J T_t f_n , 
\qquad \quad f_n = \int_{2^{n}}^{2^{n+1}} s^m  \Ls^m T_s  f\, \frac{\dd s}{s}.
\end{align*}
Observe now that
\begin{align*}
T_t f_n  
&= T_{2^{n-1} + t} \int_{2^{n}}^{3\, 2^{n-1}} (s\Ls)^m T_{s-2^{n-1}} f \, \frac{\dd s}{s} + T_{2^{n} +t} \int_{3\,2^{n-1}}^{2^{n+1}} (s\Ls)^m T_{s-2^{n}} f \, \frac{\dd s}{s} \\
 &= T_{2^{n-1}+t} \int_{2^{n-1}}^{2^{n}} (s+2^{n-1})^{m-1}\Ls^m T_s f\, \dd s + T_{2^{n}+t} \int_{2^{n-1}}^{2^{n}} (s+2^{n})^{m-1} \Ls^m T_{s} f\,\dd s. 
\end{align*}
If  $t\in [2^{j}, 2^{j+1}]$, by Lemma~\ref{pointwiseestheat1} we then have, for some $c,c'>0$,
\begin{align*}
 |X_J T_t f_n |
& \lesssim (2^{n}+2^{j})^{-\frac{m}{2}} T_{c(2^{j}+2^{n})} \int_{2^{n-1}}^{2^{n}} s^m |\Ls^m T_s f | \, \frac{\dd s}{s}\\
& \lesssim 2^{-\frac{m}{2}\max \{j,n\}} T_{c 2^{j}} \Big( T_{c2^{n}} \int_{2^{n-1}}^{2^{n}} s^m |\Ls^m T_s f | \, \frac{\dd s}{s}\Big)\\ &\lesssim 2^{-\frac{m}{2}\max \{j,n\}} T_{c 2^{j}}  (2^{nm} T_{c' 2^{n}}|\Ls^m T_{2^{n-2}} f|).
 \end{align*}
Therefore, by Proposition~\ref{contCRTN} and~\eqref{Schurdiscr},
\begin{align*}
\|2^{j(m-\alpha)/2} & T^{(m),*}_{2^j} f \|_{L^p(\dell^q)} \\
  &\lesssim \Big\| 2^{j(m-\alpha)/2} T_{c 2^{j}}  \sum_{n\in\Z}  2^{ -\frac{m}{2} \max\{ j, n\}} (2^{nm} T_{c' 2^{n}}|\Ls^m T_{2^{n-2}} f|)\Big\|_{L^p(\dell^q)}\\
  & \lesssim \|  2^{n(m-\frac{\alpha}{2})} T_{c' 2^{n}}|\Ls^m T_{2^{n-2}} f|\|_{L^p(\dell^q)}.
\end{align*}
Proposition~\ref{contCRTN}, Lemma~\ref{lemmaintersec} and Theorem~\ref{teo-equiv2} complete the proof of (ii).
\end{proof}

\subsection{Paraproducts} Our proof of the algebra properties goes via a paraproduct decomposition, which we proceed to show. Suppose $\alpha> 0$ and let $p,q,p_1p_2,p_3,p_4\in [1,\infty]$ be such that $\frac{1}{p_1} + \frac{1}{p_2} = \frac{1}{p_3} + \frac{1}{p_4} = \frac{1}{p}$ and $\frac{1}{p_2}+ \frac{1}{p_3} \leq  1$. 

Assume that either  $f\in L^{p_3} \cap \dB^{p_1,q}_\alpha$ and $g \in  L^{p_2} \cap\dB^{p_4,q}_\alpha $, or $f\in L^{p_3} \cap \dF^{p_1,q}_\alpha $ and $g \in  L^{p_2}\cap \dF^{p_4,q}_\alpha $ with $p,p_1,p_4\in (1,\infty)$ and $p_2,p_3\in (1,\infty]$. If $\frac{1}{r} = \frac{1}{p_2}+\frac{1}{p_3}$, then in particular $fg \in L^r$. For all integers $m>\alpha/2$ and $n\geq 0$, and all $\tau>0$, the integrals
\[
\Pi_f^{(m,n,\tau)}(g) = \sum_{h,k=0}^{m-1} \frac{1}{(m-1)! h! k!} \int_0^\tau t^{h+m+k} \Ls^{h+n} T_{t} [\Ls^m T_t  f \cdot \Ls^k T_t g] \, \frac{\dd t}{t},
\]
its symmetric $\Pi_g^{(m,n,\tau)}(f) $, and
\[
\Pi^{(m,n,\tau)}(f,g) = \sum_{h,k=0}^{m-1} \frac{1}{(m-1)! h! k!} \int_0^\tau t^{m+h+k} \Ls^{m+n} T_t  [\Ls^h T_t  f \cdot \Ls^k T_t g] \, \frac{\dd t}{t},
\]
are well-defined elements of $\Ss'$. This follows by the conditions on $f$ and $g$. For example, if $h,k \in \{ 0,\dots, m-1\}$, then for all $\phi \in\Ss$
\[
\int_0^\tau t^{h+m+k}   |\langle \Ls^{h+n} T_{t} [\Ls^m T_t  f \cdot \Ls^k T_t g], \phi\rangle |\, \frac{\dd t}{t}  \lesssim \|[f]\|_{\dF^{p_1, q}_\alpha}\, \|g\|_{p_2}\|\Ls^n \phi\|_{p'},
\]
and the other terms are similar. Observe that if $n\geq 1$, all integrals above admit a limit in $\Ss'$ for $\tau \to \infty$. As for $\Pi_f^{(m,n,\infty)}(g)$, e.g., given integers $h,k$ and $\phi \in\Ss$,
\[
\int_\tau^\infty t^{h+m+k}   |\langle \Ls^{h+n} T_{t} [\Ls^m T_t  f \cdot \Ls^k T_t g], \phi\rangle | \, \frac{\dd t}{t}  \lesssim \| f\|_{p_3} \, \|g\|_{p_2} \, \|\phi\|_{r'} \int_\tau^\infty  t^{-n} \, \frac{\dd t}{t}.
\]
The other terms are analogous. In this case, we denote the limits by $\Pi_f^{(m,n,\infty)}(g)$, $\Pi_g^{(m,n,\infty)}(f)$ and $\Pi^{(m,n,\infty)}(f,g) $ respectively. We also define
\[
\mathcal{E}^{(m,n,\tau)}(f,g) = \sum_{h,k,\ell=0}^{m-1}  \frac{1}{h!k!\ell!}\tau^{h+k+\ell} \Ls^{h+n}T_\tau [( \Ls^k T_\tau f) \cdot (\Ls^\ell T_\tau g)].
\]
The following proposition concerns the aforementioned paraproduct decomposition.
\begin{proposition}\label{paraproduct}
Suppose $\alpha> 0$ and let $p,q,p_1p_2,p_3,p_4\in [1,\infty]$ be such that $\frac{1}{p_1} + \frac{1}{p_2} = \frac{1}{p_3} + \frac{1}{p_4} = \frac{1}{p}$ and $\frac{1}{p_2}+ \frac{1}{p_3} \leq  1$. If either
\begin{itemize}
\item[\emph{(i)}] $f\in L^{p_3} \cap \dB^{p_1,q}_\alpha$ and $g \in  L^{p_2} \cap\dB^{p_4,q}_\alpha $, or
\item[\emph{(ii)}] $p,p_1,p_4\in (1,\infty)$, $p_2,p_3\in (1,\infty]$, $f\in L^{p_3} \cap \dF^{p_1,q}_\alpha $ and $g \in  L^{p_2}\cap \dF^{p_4,q}_\alpha $,
\end{itemize}
then, for all integers $m>\alpha/2$,
\begin{equation}\label{para}
fg = \Pi_f^{(m,0,\tau)}(g) + \Pi_g^{(m,0,\tau)}(f) + \Pi^{(m,0,\tau)}(f,g) +\mathcal{E}^{(m,0,\tau)}(f,g) \qquad \mbox{in } \Ss',
\end{equation}
and, for all integers $m>\alpha/2$ and $n\geq 1$,
\begin{equation}\label{parainfty}
\Ls^n(fg) = \Pi_f^{(m,n,\infty)}(g) + \Pi_g^{(m,n,\infty)}(f) + \Pi^{(m,n,\infty)}(f,g)  \qquad \mbox{in } \Ss'.
\end{equation}
\end{proposition}

\begin{proof}
If $r\in [1,\infty]$ is such that $\frac{1}{r}= \frac{1}{p_2}+\frac{1}{p_3}$, then $fg \in L^r$ is a distribution. The other conditions on $f$ and $g$, together with Lemma~\ref{lemmaintersec}, not only make  all terms in~\eqref{para}  well-defined elements in $\Ss'$, but they allow to justify some switching in the order of integrations below. We leave the details to the interested reader. 

We first assume that~\eqref{para} holds, and we deduce~\eqref{parainfty}. By~\eqref{para}, for $n\geq 1$
\[
\Ls^n(fg) = \Pi_f^{(m,n,\tau)}(g) + \Pi_g^{(m,n,\tau)}(f) + \Pi^{(m,n,\tau)}(f,g) +\mathcal{E}^{(m,n,\tau)}(f,g) \qquad \mbox{in } \Ss'.
\]
Since all $\Pi$'s above admit a limit for $\tau \to \infty$ as explained above, it is enough to show that  $\mathcal{E}^{(m,n,\tau)}(f,g) \to 0$ in $\Ss'$ for $\tau \to \infty$. But this follows because, for integers $h,k,\ell$,
\begin{align*}
\| \tau^{h+k+\ell} \Ls^{h+n}T_\tau [( \Ls^k T_\tau f) \cdot (\Ls^\ell T_\tau g)]\|_r 
&\leq \tau^{-n+k+\ell} \|( \Ls^k T_\tau f) \cdot (\Ls^\ell T_\tau g)\|_r \\
& \leq  \tau^{-n}  \| f\|_{p_2}\|g\|_{p_3}.
\end{align*}
It remains to show~\eqref{para}. For notational convenience, we introduce the operators
\[
\Phi_t(\Ls) = -\sum_{k=0}^{m-1} \frac{1}{k!} (t\Ls)^k T_t,\qquad  \Psi_t(\Ls)  = \frac{1}{(m-1)!} (t\Ls)^mT_t.
\]
Observe that by Proposition~\ref{thm:decomp-SP},
\begin{equation}\label{eqsinf}
f = \int_0^{\tau} \Psi_a(\Ls)f \frac{\dd a}{a}  - \Phi_\tau(\Ls)f, \qquad \int_t^{\tau} \Psi_a(\Ls)f \frac{\dd a}{a}  = \Phi_\tau(\Ls)f - \Phi_t(\Ls)f,
\end{equation}
in $\Ss'$, and analogously for $g$ and $fg$.  Therefore
\begin{align*}
fg & = \int_{0}^\tau  \Psi_t(\Ls) \bigg\{ \bigg[ \int_{0}^{\tau} \Psi_u(\Ls) f \, \frac{\dd u}{u} - \Phi_\tau(\Ls)f \bigg] \cdot \bigg[  \int_{0}^{\tau} \Psi_v(\Ls) g \, \frac{\dd v}{v} - \Phi_\tau(\Ls)g \bigg] \bigg\}  \frac{\dd t}{t}\\
& \qquad  -  \Phi_\tau(\Ls)\bigg\{ \bigg[ \int_{0}^{\tau} \Psi_u(\Ls) f \, \frac{\dd u}{u} - \Phi_\tau(\Ls)f \bigg] \cdot \bigg[  \int_{0}^{\tau} \Psi_v(\Ls) g \, \frac{\dd v}{v} - \Phi_\tau(\Ls)g \bigg] \bigg\} \\
& = I_1(f,g) - I_2(f,g) - I_2(g,f) + I_3(f,g) - I_4(f,g) + I_5(f,g) + I_5(g,f) \\
& \qquad +\mathcal{E}^{(m,0,\tau)}(f,g),
\end{align*}
where
\begin{align*}
 I_1(f,g) & = \int_{0}^\tau  \Psi_t(\Ls) \bigg\{ \bigg[ \int_{0}^{\tau} \Psi_u(\Ls) f \, \frac{\dd u}{u}\bigg] \cdot \bigg[  \int_{0}^{\tau} \Psi_v(\Ls) g \, \frac{\dd v}{v} \bigg] \bigg\}  \frac{\dd t}{t} \\
  I_2(f,g) &=\int_0^\tau  \int_0^{\tau} \Psi_t(\Ls) ( \Psi_u(\Ls) f \cdot \Phi_\tau(\Ls)g)   \frac{\dd u\, \dd t }{u\, t}\\
    I_3(f,g) &=\int_0^\tau \Psi_\tau (\Ls) ( \Phi_\tau(\Ls) f \cdot \Phi_\tau(\Ls)g)   \frac{\dd t }{t}\\
     I_4(f,g) &=\int_0^\tau  \int_0^{\tau} \Phi_\tau(\Ls) ( \Psi_u(\Ls) f \cdot \Psi_v(\Ls)g)   \frac{\dd u\, \dd v }{u\, v}\\
      I_5(f,g) &=\int_0^\tau \Phi_\tau (\Ls) ( \Psi_t(\Ls) f \cdot \Phi_\tau(\Ls)g)   \frac{\dd t }{t}.
\end{align*}
By splitting the inner integrals in $I_1$ as the sum of the integrals on $[0,t]$ and on $[t,\tau]$, and exchanging the order of integration whenever allowed, one can see that
\begin{equation}\label{eqsimp}
\begin{split}
I_1(f,g) =\int_{0}^\tau  \Psi_t(\Ls) \bigg[  \int_t^{\tau} \Psi_u(\Ls) f\, \frac{\dd u}{u} \cdot \int_t^\tau \Psi_v(\Ls) g\, \frac{\dd v }{v}\bigg] &\frac{\dd t}{t}  \\&  + I_6(f,g) + I_6(g,f)
\end{split}
\end{equation}
where
\[
I_6(f,g) =  \int_{0}^\tau  \int_u^{\tau} \Psi_t(\Ls) \bigg[ \Psi_u(\Ls) f \cdot   \int_u^{\tau}  \Psi_v(\Ls) g \, \frac{\dd v}{v}\bigg] \,  \frac{\dd t}{t}  \frac{\dd u}{u}.
\]
The first term in the right hand side of~\eqref{eqsimp} equals, by means of the second equality in~\eqref{eqsinf} applied to both inner integrals,
\begin{align*}
\Pi^{(m,0,\tau)}(f,g) - I_7(f,g) - I_7(g,f) + I_8(f,g),
\end{align*}
where
\[
 I_7(f,g)= \int_{0}^\tau \Psi_t(\Ls) [ \Phi_t(\Ls)f \cdot \Phi_\tau(\Ls)g]\frac{\dd t}{t} , \quad  I_8(f,g)= \int_0^\tau \Psi_t(\Ls) [\Phi_\tau(\Ls) \cdot \Phi_\tau(g)]\, \frac{\dd t}{t}.
\]
By tedious but elementary similar computations one can show that 
\begin{multline*}
- I_7(f,g) - I_7(g,f) + I_8(f,g)- I_2(f,g) - I_2(g,f) + I_3(f,g) - I_4(f,g) \\
+ I_5(f,g) + I_5(g,f)  + I_6(f,g) + I_6(g,f) = \Pi_f^{(m,0,\tau)}(g) + \Pi_g^{(m,0,\tau)}(f) ,
\end{multline*}
which concludes the proof.
\end{proof}

We are now in a position to establish the algebra properties of $\dB$- and $\dF$- spaces.
\begin{theorem} \label{teo_algebra}
Suppose  $\alpha> 0$, $q\in [1,\infty]$, and let $p,p_1,p_2,p_3,p_4\in [1,\infty]$ be such that $\frac{1}{p_1} + \frac{1}{p_2} = \frac{1}{p_3} + \frac{1}{p_4} = \frac{1}{p}$ and $\frac{1}{p_2}+ \frac{1}{p_3} \leq  1$.
\begin{itemize}
\item[\emph{(i)}] If $f\in L^{p_3} \cap \dB^{p_1,q}_\alpha$ and $g \in  L^{p_2} \cap\dB^{p_4,q}_\alpha $, then
\begin{equation*} 
\|[fg]\|_{\dB^{p,q}_\alpha} \lesssim \|[f]\|_{\dB^{p_1,q}_\alpha}\|g\|_{{p_2}} + \|f\|_{{p_3}} \|[g]\|_{\dB^{p_4,q}_\alpha}.
\end{equation*}
\item[\emph{(ii)}] If $p,p_1,p_4\in (1,\infty)$,  $p_2,p_3\in (1,\infty]$, $f\in L^{p_3} \cap \dF^{p_1,q}_\alpha $ and $g \in  L^{p_2}\cap \dF^{p_4,q}_\alpha $, then
\begin{equation*}
\|[fg]\|_{\dF^{p,q}_\alpha} \lesssim \|[f]\|_{\dF^{p_1,q}_\alpha}\|g\|_{{p_2}} + \|f\|_{ {p_3}} \|[g]\|_{\dF^{p_4,q}_\alpha}.
\end{equation*}
\end{itemize}
In particular, for $\alpha> 0$, $q\in [1,\infty]$ and $p\in [1,\infty]$ or $p\in (1,\infty)$ respectively,  $L^\infty \cap \dB^{p,q}_\alpha$ and $L^\infty \cap \dF^{p,q}_\alpha$ are algebras under pointwise multiplication.
\end{theorem}
\begin{proof}
We prove only (ii), for the proof of (i) follows the same steps and is easier. Observe that $\|[f]\|_{\dF^{p,q}_\alpha}= \| t^{m-\frac{\alpha}{2}} \Ls^m T_t f\|_{L^p(L^q_+)}$ by Lemma~\ref{lemmaintersec}, and that the same holds for $g$ and $fg$.

Pick $m = [\alpha/2]+1$. By~\eqref{parainfty}, 
\begin{align*}
\|[fg]\|_{\dF^{p,q}_\alpha} 
& = \| t^{m-\alpha/2} T_t \Ls^m (fg)\|_{L^{p}(L^q_+)}\\
&  \lesssim \| t^{m-\alpha/2} T_t\Pi_f^{(m,m,\infty)}(g) \|_{L^{p}(L^q_+)} + \| t^{m-\alpha/2} T_t \Pi_g^{(m,m,\infty)}(f)\|_{L^{p}(L^q_+)} \\
& \hspace{5cm}  + \| t^{m-\alpha/2} T_t\Pi^{(m,m,\infty)}(f,g)\|_{L^{p}(L^q_+)} .
\end{align*}
We claim that (from now on we omit the superscripts)
\begin{equation}\label{eqalg1TL}
\| t^{m-\alpha/2} T_t\Pi_f(g) \|_{L^{p}(L^q_+)} \lesssim \|[f]\|_{\dF^{p_1,q}_\alpha}\|g\|_{{p_2}},
\end{equation}
which by symmetry implies also $\| t^{m-\alpha/2} T_t \Pi_g(f)\|_{L^{p}(L^q_+)} \lesssim   \|f\|_{{p_3}}   \|[g]\|_{\dF^{p_4,q}_\alpha}$, and that
\begin{equation}\label{eqalg2TL}
\| t^{m-\alpha/2} T_t\Pi(f,g)\|_{L^{p}(L^q_+)} \lesssim \|[f]\|_{\dF^{p_1,q}_\alpha}\|g\|_{{p_2}} + \|f\|_{{p_3}} \|[g]\|_{\dF^{p_4,q}_\alpha}.
\end{equation}
The theorem  follows by the claims.

\smallskip

We first prove~\eqref{eqalg1TL}. By definition,
\begin{multline*}
\| u^{m-\alpha/2} T_u\Pi_f(g) \|_{L^{p}(L^q_+)} \\
\lesssim \sum_{h,k=0}^{m-1} \bigg\| u^{m-\frac{\alpha}{2}} \int_0^\infty t^{h+m+k } |\Ls^{m+h} T_{u+t}  [\Ls^mT_t  f  \cdot \Ls^k T_t g]| \, \frac{\dd t}{t} \bigg\|_{L^p(L^q_+)}.
\end{multline*}
Suppose now $h,k \in \{0,\dots, m-1\}$ and $u>0$. By Lemma~\ref{pointwiseestheat1}, there exist $a_{h},a_{m}>0$ such that
\begin{align*}
t^{h} |\Ls^{m+h} T_{u+t}  [\Ls^mT_t  f  \cdot \Ls^k T_t g] |  
& \lesssim  T_{a_h t}  |\Ls^m T_{t/2 + u}  [\Ls^mT_t  f  \cdot \Ls^k T_t g]|\\
& \lesssim \left(\tfrac{t}{2}+u\right)^{-m} T_{a_h t}T_{a_{m} (t/2+u)} |\Ls^mT_t  f  \cdot \Ls^k T_t g|\\
& \lesssim (u+t)^{-m} T_{c(t+u)} |\Ls^mT_t  f  \cdot \Ls^k T_t g|,
\end{align*}
for some $c>0$. Therefore,
\begin{multline*}
 \int_0^\infty t^{h+m+k }|\Ls^{m+h} T_{u+t}  [\Ls^mT_t  f  \cdot \Ls^k T_t g]| \, \frac{\dd t}{t}  \\
\lesssim   T_{cu}\int_0^\infty t^{m+k } (u+t)^{-m} T_{ct} |\Ls^mT_t  f  \cdot \Ls^k T_t g| \, \frac{\dd t}{t}.
\end{multline*}
By Proposition~\ref{contCRTN} with $F(u,\cdot ) =u^{m-\frac{\alpha}{2}} \int_0^\infty t^{m+k } (u+t)^{-m} T_{ct} |\Ls^mT_t  f  \cdot \Ls^k T_t g| \, \frac{\dd t}{t} $, by~\eqref{Schurcont} and then by Corollary~\ref{bilinearCRTNint}
\begin{align*}
 \Big\| u^{m-\frac{\alpha}{2}} \int_0^\infty t^{h+m+k }& |\Ls^{m+h} T_{u+t}  [\Ls^mT_t  f  \cdot \Ls^k T_t g]| \, \frac{\dd t}{t} \Big\|_{L^p(L^q_+)}\\
&  \lesssim \Big\| u^{m-\frac{\alpha}{2}} \int_0^\infty t^{m+k } (u+t)^{-m} T_{ct} |\Ls^mT_t  f  \cdot \Ls^k T_t g| \, \frac{\dd t}{t}\Big\|_{L^p(L^q_+)}\\
&  \lesssim  \| t^{m+k- \frac{\alpha}{2}} T_{ct} |\Ls^m T_t f \cdot  \Ls^k T_t g|  \|_{L^p(L^q_+)}\\
&   \lesssim  \| t^{m+k- \frac{\alpha}{2}} T_{c't} |\Ls^m T_t f | \cdot T_{c't} |\Ls^k T_t g|  \|_{L^p(L^q_+)} \\
& \lesssim  \| t^{m- \frac{\alpha}{2}} T_{c't} |\Ls^m T_t f | \cdot T_{c'' t}|g|  \|_{L^p(L^q_+)}.
\end{align*}

By  H\"older's inequality,  the $L^{p_2}$-boundedness of the heat maximal operator (observe that $p_2>1$) and Corollary~\ref{bilinearCRTNint} we conclude
\begin{align*}
\|u^{m-\alpha/2} T_u\Pi_f(g)\|_{L^p(L^q_+)} &  \lesssim  \|\sup_{t>0} T_{c''t} |g| \|_{p_2} \| t^{m- \frac{\alpha}{2}} |\Ls^m T_t f |\|_{L^{p_1}(L^q_+)} \\
& \lesssim \|g\|_{p_2}\| [f]\|_{\dF^{p_1, q}_\alpha}.
\end{align*}
The proof of~\eqref{eqalg1TL} is thus complete.

We prove~\eqref{eqalg2TL}.  By definition,
\begin{multline*}
\| u^{m-\alpha/2} T_u\Pi(f,g)\|_{L^{p}(L^q_+)} \\
  \lesssim \sum_{h,k=0}^{m-1}  \bigg\| u^{m-\alpha/2}  \int_0^\infty t^{m+h+k} |\Ls^{2m} T_{t+u}  [\Ls^h T_t  f \cdot \Ls^k T_t g] | \, \frac{\dd t}{t} \bigg\|_{L^{p}(L^q_+)} .
\end{multline*}
 By Lemma~\ref{pointwiseestheat1} and the Leibniz rule, there is $a=a_m>0$ such that
\begin{align*}
|\Ls^{2m} T_{t+u} & [\Ls^h T_t  f \cdot \Ls^k T_t g] |
=  |\Ls^m T_{u+t} \Ls^m  [\Ls^h T_t  f \cdot \Ls^k T_t g]| \\
& \lesssim (t+u)^{-m} T_{a(t+u)} |\Ls^m[\Ls^h T_t  f \cdot \Ls^k T_t g]| \\
& \lesssim (t+u)^{-m} T_{a(t+u)}  \sum_{i=0}^{2m} \max_{|I|= i+2h} | Y_{I}T_t f |\max_{|J|= 2m+2k-i} |Z_{J} T_tg |,
\end{align*}
where $(Y_I, Z_J) = (\Ls^{h}, \Ls^{m+k})$ if $|I|=2h$ and $|J|=2k+2m$, $(Y_I, Z_J) = (\Ls^{m+h}, \Ls^{k})$ if $|I|=2h+2m$ and $|J|=2k$, and $(Y_I, Z_J) = (X_I, X_J)$ otherwise. Thus, if we define
\[
F_t^{i,h,k}(f,g)=\max_{|I|= i+2h}  | Y_{I}T_t f |\max_{|J|= 2m+2k-i}|Z_{J} T_tg |,
\]
by Proposition~\ref{CRTNintegrale} and~\eqref{Schurcont} we obtain as before
\begin{align*}
\| u^{m-\alpha/2} &\Ls^m  T_u\Pi(f,g)\|_{L^{p}(L^q_+)}  \\
&\lesssim \!\sum_{h,k=0}^{m-1}\sum_{i=0}^{2m} \left\|u^{m-\alpha/2} T_{au} \int_0^\infty t^{m+h+k}(t+u)^{-m} T_{at}   F_t^{i,h,k}(f,g)\, \frac{\dd t}{t} \right\|_{L^{p}(L^q_+)}  \\
& \lesssim\sum_{h,k=0}^{m-1}  \sum_{i=0}^{2m} \| t^{m+k+h-\frac{\alpha}{2}}  T_{at} F_t^{i,h,k}(f,g)\|_{L^p(L^q_+)}.
\end{align*}
We separate two cases, depending on the values of $i$. The cases $i=0$ or $i=2m$ are symmetric, so we only give details for $i=0$. By Corollary~\ref{bilinearCRTNint} and H\"older's inequality we obtain as above
\begin{align*}
\| t^{m+k+h-\frac{\alpha}{2}}  T_{at} F_{t}^{0,h,k}(f,g)\|_{L^p(L^q_+)} &= \| t^{m+k+h-\frac{\alpha}{2}}  T_{at} (| \Ls^h T_t f | \cdot  |\Ls^{m+k} T_t g|)\|_{L^p(L^q_+)}\\
& \lesssim \| t^{m+k+h-\frac{\alpha}{2}} T_{ct} |\Ls^h T_t f | \cdot  T_{ct}|\Ls^{m+k} T_t g|\|_{L^p(L^q_+)}\\
&   \leq  \|\sup_{t>0}   T_t |f| \|_{p_3}   \| t^{m+k-\alpha/2} T_{ct}|\Ls^{m+k} T_t g |\|_{L^{p_4}(L^q_+)}\\
& \lesssim \| f\|_{{p_3}} \| [g]\|_{\dF^{p_4,q}_\alpha},
\end{align*}
the last inequality by the $L^{p_3}$-boundedness of the heat maximal operator.

Assume now that $i\in \{1,\dots, 2m-1 \}$.  Since for $t\in [2^j,2^{j+1}]$
\[
 F_t^{i,h,k} (f,g)\leq  T_{2^j}^{(i+2h), *}f  \cdot T_{2^j}^{(2m+2k-i), *}g,
\]
one has, by Proposition~\ref{contCRTN},
\begin{align*}
  \|  t^{m+h+k-\frac{\alpha}{2}}  T_{at} & F_t^{i,h,k}(f,g)\|_{L^p(L^q_+)}  \\
  &\lesssim  \|  2^{j(m+h+k-\frac{\alpha}{2})}  T_{c{2^j}} ( T_{2^j}^{(i+2h), *}f \cdot T_{2^j}^{(2m+2k-i), *}g) \|_{L^p(\dell^q)}\\
  &\lesssim  \|  2^{j(m+h+k-\frac{\alpha}{2})}  T_{2^j}^{(i+2h), *}f \cdot T_{2^j}^{(2m+2k-i), *}g\|_{L^p(\dell^q)}.
\end{align*}
We now pick $\theta = \frac{i+2h}{2(m+h+k)}$ and apply H\"older's inequality with $\alpha_1 = \theta\alpha$, $ \alpha_2 = (1-\theta)\alpha$, $\frac{1}{q_1} = \frac{\theta}{q}$, $\frac{1}{q_2} = \frac{1-\theta}{q}$ to obtain that the last term of the previous inequality is controlled by
\begin{align*}                            
       &  \Big\|  \| 2^{j (i+2h-\alpha_1)/2} T^{(i+2h), *}_{2^{j}}f \|_{\dell^{q_1}} \|2^{j (2m+2k-i -\alpha_2)/2} T^{(2m+2k-i), *}_{2^{j}}g  \|_{\dell^{q_2}}  \Big\|_p\,,  
 \end{align*}      
 which in turn, by H\"older's inequality with  $\frac{1}{r_1} = \frac{\theta}{p_1} + \frac{1-\theta}{p_3}$, $ \frac{1}{r_2} = \frac{\theta}{p_2} + \frac{1-\theta}{p_4}$, is controlled by
       \begin{multline*}
   \| 2^{j (i+2h-\alpha_1)/2} T^{(i+2h), *}_{2^{j}}f \|_{L^{r_1}(\dell^{q_1})}  \| 2^{j (2m+2k-i -\alpha_2)/2} T^{(2m+2k-i), *}_{2^{j}}g  \|_{L^{r_2}(\dell^{q_2})}\\ \lesssim \|[f]\|_{\dF^{r_1,q_1}_{\alpha_1}} \|[g]\|_{\dF^{r_2,q_2}_{\alpha_2}},
\end{multline*}
by Lemma~\ref{teo-equiv3}. Observe now that by H\"older's inequality and the $L^r$-boundedness of  the heat maximal operator for $r\in (1,\infty]$, we have
\[
\|[f]\|_{\dF^{r_1,q_1}_{\alpha_1}}  \lesssim \|f\|_{{p_3}}^{1-\theta } \|[f]\|_{\dF^{p_1,q}_\alpha}^{\theta} , \qquad \|[g]\|_{\dF^{r_2,q_2}_{\alpha_2}}  \lesssim  \|[g]\|_{\dF^{p_4,q}_\alpha}^{1-\theta}  \|g\|_{{p_2}}^{\theta} ,
\]
whence
\begin{align*}
   \|[f]\|_{\dF^{r_1,q_1}_{\alpha_1}} \|[g]\|_{\dF^{r_2,q_2}_{\alpha_2}} &
   \lesssim \|f\|_{{p_3}}^{1-\theta } \|[f]\|_{\dF^{p_1,q}_\alpha}^{\theta} \|[g]\|_{\dF^{p_4,q}_\alpha}^{1-\theta}  \|g\|_{{p_2}}^{\theta} \\
   & \lesssim \|f\|_{{p_3}} \|[g]\|_{\dF^{p_4,q}_\alpha} + \|[f]\|_{\dF^{p_1,q}_\alpha}  \|g\|_{{p_2}}
\end{align*}
which completes the proof of~\eqref{eqalg2TL} and of the theorem.
\end{proof}

\section{Property (S), representations and interpolation}\label{sec:SS}
In this section we prove homogeneous representation formulae, in terms of heat semigroups as well as of Littlewood--Paley decompositions. The former will also lead to interpolation properties of $\dB$- and $\dF$- spaces. The setting will be that of a smooth manifold, as in Section~\ref{sec:AP}.

First of all, we notice that a representation such as~\eqref{HomLnf} or~\eqref{HomLnfBF} does not hold in general, if $n=0$. Indeed, for $f\in \Ss'$, the integral
\[
\int_1^\tau (t\Ls)^m T_t f \, \frac{\dd t}{t}
\]
in general does not converge in $\Ss'$ when $\tau \to \infty$; see also~\cite[p.\ 53]{Peetre}.

This can be overcome if $\Ss$ satisfies the following property, inspired by~\cite[Proposition 2.7]{Bownik} on $\R^\dm$, which for future reference we shall denote by $(S)$: \emph{For all $n\in \N$ there exists $\nu \in \N$ such that, for all sequences $(f_k)_{k\in \N}$ in $\Ss'$ such that $(X_J f_k)_{k\in \N}$ converges in $\Ss'$ for all $J\in \mathcal{I}^n$, there exists a sequence of polynomials $(\rho_k)_{k\in \N}$ in $\Ps_\nu$ such that $(f_k + \rho_k)_{k\in \N}$ converges in $\Ss'$}.

We shall show that under condition $(S)$, Calder\'on-type representation formulae for distributions whose equivalence class belongs to some $\dF$- or $\dB$- space  can be obtained. For more general distributions, we shall also need the following property, similar to that of Lemma~\ref{pointwiseestheat1}~(2),
\begin{equation}\label{byparts}
|T_t X_J \phi| \lesssim t^{-h/2} T_{ct}|\phi|, \qquad  t>0, \, J\in \mathcal{I}^h,\, h\in\N
\end{equation}
for some $c>0$ and all $\phi\in\Ss$. The following lemma can be proved in the exact same way as Lemma~\ref{lemmaseminorm}, and its proof is omitted.

\begin{lemma}\label{lemmaseminorm1}
Let $n\geq 0$ be an integer. If~\eqref{byparts} holds and $\ell \in \N$, $h\in\N$ and $J\in \mathcal{I}^h$, then  $p_n(\Ls^\ell T_t X_J \phi ) \lesssim t^{n/2-h/2-\ell} p_{n}(\phi)$  for all $\phi \in\Ss$ and $t\geq 1$.
\end{lemma}

\begin{proposition}\label{HomLPgen}
Let $m\geq 1$ be an integer and suppose $f\in \Ss'$. If $(S)$ and~\eqref{byparts} hold, then there exist $\nu= \nu(f)\in \N$, $\rho_{f}\in \Ps$, and a sequence $(\rho_k) \subset \Ps_{\nu}$ such that
\begin{equation}\label{LPD2gen}
f - \rho_f= \frac{1}{(m-1)!} \lim_{k \to \infty} \bigg( \int_0^{k} (t\Ls)^mT_t f \, \frac{\dd t}{t} + \rho_k\bigg) \qquad \mbox{ in }\; \Ss'.
\end{equation}
\end{proposition}

\begin{proof}
We claim that there exists $\nu=\nu(f)$ and a sequence $(\rho_k)$ in $\Ps_\nu$ such that
\[
f_0 = \frac{1}{(m-1)!}\lim_{k \to \infty} \bigg( \int_0^{k} (t\Ls)^mT_t f \, \frac{\dd t}{t} + \rho_k\bigg)
\]
exists in $\Ss'$. Assuming the claim for a moment, we complete the proof.

By~\eqref{LPDlocal}, for all $k\in\N$
\[
 f - \frac{1}{(m-1)!} \bigg( \int_0^{k} (t\Ls)^mT_t f \, \frac{\dd t}{t} + \rho_k \bigg) =  \sum_{j=0}^{m-1} \frac{1}{j!} (k\Ls)^j T_k f - \frac{1}{(m-1)!}\rho_k.
\]
Let $h\in\N$ be such that 
\begin{equation}\label{fh}
|\langle f,\phi\rangle|\leq C p_h(\phi) \qquad \forall\, \phi\in \Ss,
\end{equation}
and define $\eta = \max\{ \nu, [h/2]+1\}$. Since for all $k\geq 1$ and $j=0,\dots, m-1$
\[
|\langle (k\Ls)^j \Ls^\eta T_k f, \phi\rangle|  \lesssim p_h( (k\Ls)^j \Ls^\eta T_k \phi) \lesssim k^{h/2-\eta} p_{n}(\phi)
\]
by Lemmas~\ref{pointwiseestheat1} and~\ref{lemmaseminorm}, and since $\Ls^\eta \rho_k=0$ for all $k$,
\[
\Ls^\eta (f - f_0) =  \lim_{k \to \infty}  \sum_{j=0}^{m-1} \frac{1}{j!} (k\Ls)^j \Ls^\eta T_k f=0
\]
in $\Ss'$. Then $f-f_0 \in \Ps_\eta$, and the statement follows. Therefore, it remains to prove the claim.

We first observe that for all $f\in \Ss'$, by~\eqref{LPDlocal},
\begin{equation}\label{01noprobl}
\int_0^1 (t\Ls)^{m} T_t f \, \frac{\dd t}{t} \in \Ss'.
\end{equation}
By $(S)$, the claim will then follow if we show that there exists $n \in \N$ such that if $|J|=n$ the sequence
\[
X_J \int_1^k  (t\Ls)^{m} T_t f  \, \frac{\dd t}{t} = \int_1^k  X_J  (t\Ls)^{m}T_{t} f \, \frac{\dd t}{t} , \qquad k\in \N,
\]
converges in $\Ss'$ for $k\to \infty$. By~\eqref{fh} it is enough to show that  there is $n \in \N$ such that if $|J|=n$ the sequence
\[
 \int_1^k p_h(  (t\Ls)^{m} T_t X_J \phi)\, \frac{\dd t}{t}, \qquad k\in \N
\]
converges for all $\phi\in\Ss$. By Lemma~\ref{lemmaseminorm1}, for $t\geq 1$
\[
p_h(  (t\Ls)^{m}  T_t X_J\phi) \lesssim  t^{h/2-n/2} p_{h}(\phi),
\]
from which the claim follows by taking $n>h$.
\end{proof}

The representation formula~\eqref{LPD2gen} has a stronger form if $f\in \Ss'$ is such that $[f]$ belongs to some $\dB$- or $\dF$- space. In this case, moreover,~\eqref{byparts} is not needed.

\begin{proposition}\label{prop_realdec}
If $(S)$ holds, then for all $\alpha \geq 0$ there exists $\nu = \nu(\alpha)$ such that, for all $f\in \Ss'$ satisfying $[f]\in \dB^{p,q}_\alpha$ for some $p,q\in [1,\infty]$, or $[f]\in \dF^{p,q}_\alpha$ for some $p\in (1,\infty)$ and $q\in [1,\infty]$, and for all integers $m >\alpha/2$,  there are $\rho_{f}\in \Ps$  and  $(\rho_k) \subset \Ps_{\nu}$ such that
\begin{equation}\label{LPD2}
f - \rho_f= \frac{1}{(m-1)!} \lim_{k \to \infty} \bigg( \int_0^{k} (t\Ls)^mT_t (f+\wp_{f}) \, \frac{\dd t}{t} + \rho_k\bigg) \qquad \mbox{ in }\; \Ss',
\end{equation}
where $\wp_{f} = \wp(f,\alpha, [\alpha/2]+1,p,q)$ is that of Proposition~\ref{prop:minattained}.
\end{proposition}
\begin{proof}
Write $\nu_0= [\alpha/2] +1$, and let $\nu$ be the integer associated with $n=2\nu_0$ in $(S)$. Let $f\in \Ss'$ be such that $[f]\in \dF^{p,q}_\alpha$ for $p\in (1,\infty)$ and $q\in [1,\infty]$, and suppose $m > \alpha/2$ is an integer. Pick $\wp_f =\wp(f,\alpha, \nu_0, p,q)\in \Ps$. We shall show that there exists a sequence $(\rho_k) \subset \Ps_{\nu}$ such that
\[
f_0 = \frac{1}{(m-1)!}\lim_{k \to \infty} \bigg( \int_0^{k} (t\Ls)^mT_t (f + \wp_f) \, \frac{\dd t}{t} + \rho_k\bigg)
\]
exists in $\Ss'$. The conclusion will then follow in the same way as in Proposition~\ref{LPD2gen}.

By~\eqref{01noprobl} with $f+\wp_f$ in place of $f$, and by $(S)$, it will be enough to show the convergence in $\Ss'$ of the sequence
\[
X_J \int_1^k (t\Ls)^mT_t (f+ \wp_{f}) \, \frac{\dd t}{t} , \qquad k\in \N,
\]
for $|J|=2\nu_0$. Arguing exactly as in~\eqref{conv1tau} with $X_J$ in place of $\Ls^n$, one actually sees that it converges in $L^p$. By $(S)$ and the choice of $\nu$, there is then a sequence $(\rho_k) \subset \Ps_{\nu}$ such that
\[
\int_1^k (t\Ls)^mT_t (f+ \wp_{f}) \, \frac{\dd t}{t}  + \rho_k, \qquad k\in \N,
\]
converges in $\Ss'$, namely $f_0 \in \Ss'$.
\end{proof}

\subsection{Comparisons and remarks}\label{sub:comparison}
Let us now discuss the equivalence of our $\dB$- and $\dF$- norms with those defined by means of a Littlewood--Paley decomposition. 

To begin with, observe that if $\varphi \in \Ss(\R)$ is even and compactly supported away from $0$, or if it vanishes in $0$ with infinite order, then $\Ls^{-k}\varphi(\sqrt{\Ls}) $ is well defined for all $k\in \N$, and $\varphi(\sqrt{\Ls}) = \Ls^k \Ls^{-k}\varphi(\sqrt{\Ls})$. This implies that $\varphi(\sqrt{\Ls}) \rho =0$ for all $\rho \in \Ps$. As a consequence, $\varphi(\sqrt{\Ls}) (f+\rho) = \varphi(\sqrt{\Ls}) f \in \Ss'$ for all $f\in \Ss'$ and $\rho \in \Ps$. In particular, $\varphi(\sqrt{\Ls}) $ can be defined on $\Ss'/\Ps$ with values in $\Ss'$; actually its image consists of continuous function~\cite[Proposition 3.6]{GKKP1}. If $\varphi$ contains $0$ in its support, instead, this is in general not possible, as $\varphi(\sqrt{\Ls})(f+\rho)$ might depend on $\rho$ and $\varphi(\sqrt{\Ls}) $ can be defined on $\Ss'/\Ps$ only with values in $\Ss'/\Ps$. In particular, this happens for the multipliers $\Ls^k T_t$, $t>0$. Therefore, the claimed characterization in~\cite[Theorem 6.2]{GKKP1} needs to be suitably interpreted. This issue has also been highlighted in~\cite[p.\ 32]{BBD}.

Consider a function $\varphi\in C^\infty((0,\infty))$ such that $\supp \varphi \subset   [1/2, 2]$, $|\varphi(\lambda)| \geq c>0$ for $\lambda\in [2^{-3/4}, 2^{3/4}]$ and
\[
\sum_{j\in \Z} \varphi(2^{-j}\lambda) =1 \qquad \forall \, \lambda> 0.
\] 
For $j\in \Z$, we write $\varphi_j(\lambda)= \varphi(2^{-j}\lambda)$. 

We begin by showing the Littlewood--Paley counterparts of Propositions~\ref{HomLPgen} and~\ref{prop_realdec}, with a preliminary lemma.

\begin{lemma}\label{lemmaseminorm2}
Let $\psi \in C^\infty_c(0,\infty)$ be positive, and $h\in\N$. If $\ell,m\in \N$ and $J\in \mathcal{I}^m$, and if~\eqref{byparts} holds when $m>0$, then for all $t\geq 1$ and $\phi\in\Ss$
\[
p_h( \Ls^\ell \psi(t\sqrt{\Ls})X_J \phi)  \lesssim p_{h+\dm+1}(\phi) t^{h -m-2\ell -\dm^*}.
\]
\end{lemma}
\begin{proof}
Observe that for all $t>0$ and $k\in\N$, one has 
\[
\Ls^k \Ls^\ell \psi(t\sqrt{\Ls})X_J =t^{-2k-2\ell} (t\sqrt{\Ls})^{2k+2\ell}\psi(t\sqrt{\Ls}) \, \e^{t^2\Ls} T_{t^2}  X_J.
\]
Thus, if $\widetilde K_t$ denotes the integral kernel of $(t\sqrt{\Ls})^{2k+2\ell } \psi(t\sqrt{\Ls}) \, \e^{(t\sqrt{\Ls})^2}$, by~\eqref{byparts}
\begin{align*}
|\Ls^k \Ls^\ell \psi(t \Ls) X_J\phi(x)| 
& \leq t^{-2k-2\ell } \int_M|\widetilde K_t(x,z)| \, |T_{ t^2} X_J \phi (z)| \, \dd \mu(z)\\
& \lesssim t^{-2k-2\ell- m}\int_M \bigg[ \int_M |\widetilde K_t(x,z)| H_{ct^2}(z,y) \, \dd \mu(z) \bigg] |\phi(y)|\, \dd \mu(y).
\end{align*}
For notational convenience, given $\delta,\sigma>0$ let us introduce as in~\cite{CKP, KP, GKKP1}
\[
D_{\delta,\sigma}(x,y) = \frac{1}{(\mu(B(x,\delta))\mu(B(y,\delta)))^{1/2}} \left( 1+ \frac{d(x,y)}{\delta}\right)^{-\sigma}.
\]
By the heat kernel estimates~\eqref{Htestimate} and the doubling property of $\mu$, for $\sigma>0$
\[
H_{ct^2}(z,y) \lesssim  D_{t, \sigma}(z,y),
\]
and moreover, by \cite[Theorem 3.1]{KP}, see also~\cite[Theorem 2.5]{GKKP1}, for $\sigma \geq \dm +1$,
\[
 |\widetilde K_t(x,z) |\lesssim  D_{t,\sigma}(x,z), 
\]
the constants involved depending on $\sigma$. For $\sigma\geq 2\dm +1$, we obtain (cf.~\cite[Lemma 2.3 and (2.7)]{CKP})
\begin{align*}
\int_M |\widetilde K_t(x,z)| H_{ct^2}(z,y) \, \dd \mu(z)  &\lesssim \int_M D_{t,\sigma} (x,z) D_{t,\sigma}(z, y) \,\dd \mu(z) \\
& \lesssim  D_{t,\sigma}(x, y)  \\
& \lesssim \mu(B(x,t))^{-1} \left( 1+ \frac{d(x,y)}{t}\right)^{-\sigma+ \dm/2}.
\end{align*}
Observe now that if $\sigma\geq h+ \dm /2$
\begin{align*}
(1+d(x) )^h \left( 1+ \frac{d(x,y)}{t}\right)^{-\sigma+ \dm/2} 
& \leq  (1+d(y))^h  (1+d(x,y) )^h \left( 1+ \frac{d(x,y)}{t}\right)^{-\sigma+\dm/2}\\
& \leq  (1+d(y))^h (1+t)^h.
\end{align*}
By applying the above estimates with $\sigma = h +2\dm+1$, and by Remark~\ref{noncollaps}, we conclude that for all $t\geq 1$, $x\in M$ and $k=0,\dots, h$
\begin{align*}
(1+d(x))^h |\Ls^k \Ls^\ell \psi(t \sqrt{\Ls}) X_J\phi(x)| &
\lesssim t^{-2k-2\ell-m+h-\dm^*} \int_M   (1+d(y))^h |\phi(y)|\, \dd \mu(y)\\
& \lesssim  t^{-2k-2\ell -m+h-\dm^*} p_{n+\dm+1}(\phi),
\end{align*}
which concludes the proof.
\end{proof}

\begin{proposition}\label{convHLP}
If $(S)$ and~\eqref{byparts} hold, then for all $f\in \Ss'$ there exist $n= n(f)$, $\rho_f\in \Ps$, and a sequence $(\rho_k)_k \subset \Ps_{n}$ such that 
\begin{equation}\label{trueLP}
f- \rho_f =  \lim_{k\to \infty} \bigg(\sum_{j\geq -k}  \varphi_j(\sqrt{\Ls}) f + \rho_k\bigg) \qquad \mbox{in } \Ss'.
\end{equation}
\end{proposition}
\begin{proof}
Let $f\in \Ss'$ be given, and let $C>0$ and $h\in \N$ be such that
\[
|\langle f, \phi \rangle| \leq C p_{h}(\phi) \qquad \forall \, \phi\in \Ss.
\]
It is well known, cf.~\cite[Proposition 5.5]{KP}, that for all fixed $k\in\Z$
\begin{equation}\label{LPDk}
f = \varphi_{-k-1}(\sqrt{\Ls}) f + \sum_{j\geq -k} \varphi_j(\sqrt{\Ls}) f  \qquad \mbox{in }\, \Ss'.
\end{equation}
In particular, $\sum_{j\geq 0} \varphi_j(\sqrt{\Ls}) f \in \Ss'$. If we show that there exist polynomials $(\rho_k)$ in $\Ps_n$ such that
\[
f_0 = \lim_{k\to \infty}\bigg(\sum_{j\geq -k}  \varphi_j(\sqrt{\Ls}) f + \rho_k\bigg)
\]
exists in $\Ss'$, then for $\ell \geq \max(n, (h-\dm^*)/2)$ one gets by~\eqref{LPDk} and  Lemma~\ref{lemmaseminorm2}
\[
\Ls^\ell (f-f_0) = \lim_{k\to \infty}\Ls^\ell \varphi_{-k-1}(\sqrt{\Ls}) f =0
\]
in $\Ss'$, and the proof is complete.

By $(S)$ it is then enough to show that there is $m\in \N$ such that
\begin{equation}\label{serieincrim}
\sum_{-k\leq j\leq -1} X_J \varphi_j(\sqrt{\Ls}) f, \qquad k\in\N,
\end{equation}
converges in $\Ss'$ for all $J\in \mathcal{I}^m$. This follows if there exists $m$ such that for  $J\in \mathcal{I}^m$
\[
\sum_{-k \leq j \leq -1} p_h( \varphi_j(\sqrt{\Ls})X_J \phi)
\]
is a convergent sequence for all $\phi \in \Ss$. But this is a consequence of Lemma~\ref{lemmaseminorm2}.
\end{proof}

Let us now introduce the quantities, for $\alpha \geq 0$, $p,q\in [1,\infty]$ and $[f]\in \Ss'/\Ps$, 
\[
\| [f]\|_{\dB^{p,q}_{\alpha},\mathrm{LP}} = \bigg( \sum_{j\in \Z} ( 2^{j\alpha} \| \varphi_j(\sqrt{\Ls}) f\|_p)^q \bigg)^{1/q}
\]
and
\[
\| [f]\|_{\dF^{p,q}_{\alpha},\mathrm{LP}} = \bigg\| \Big( \sum_{j\in \Z} ( 2^{j\alpha} | \varphi_j(\sqrt{\Ls}) f| )^q \Big)^{1/q}\bigg\|_p .
\]
Then, we have the following.
\begin{proposition}\label{convBTLLP}
If $(S)$ holds, then for all $\alpha \geq 0$ there exists $\eta = \eta(\alpha)$ such that, for all $f\in \Ss'$ such that either $\| [f]\|_{\dB^{p,q}_{\alpha},\mathrm{LP}} $ or $\| [f]\|_{\dF^{p,q}_{\alpha},\mathrm{LP}} $ is finite for $q\in [1,\infty]$ and $p\in [1,\infty]$ or $p\in (1,\infty)$ respectively, there exist $\rho_f\in \Ps$ and a sequence $(\rho_k)_k \subset \Ps_{\eta}$ such that~\eqref{trueLP} holds.
\end{proposition}

\begin{proof}
We assume that $f\in \Ss'$ is such that $\| [f]\|_{\dF^{p,q}_{\alpha},\mathrm{LP}} $ is finite, and we show that there exists $m$, depending only on $\alpha$, such that~\eqref{serieincrim} converges in $L^p$ for all $J\in \mathcal{I}^m$. The conclusion follows as in the proof of Proposition~\ref{convHLP}.

We write again $\varphi_j(\sqrt{\Ls}) = T_{2^{-2j}} \widetilde \varphi(2^{-j}\sqrt{\Ls})$, where $\widetilde \varphi (\lambda) = \e^{\lambda^2} \!\varphi(\lambda)$. Then, for $|J|= m= [\alpha]+1$, $h \leq  k \leq 0$ and $q\in(1,\infty)$
\begin{align*}
\bigg| \sum_{h\leq j \leq k} X_J\varphi_j(\sqrt{\Ls}) f \bigg|
& \lesssim   \sum_{h\leq j \leq k}  2^{jm} T_{c2^{-2j}}| \widetilde\varphi_j(\sqrt{\Ls}) f | \\
 &  \lesssim \bigg( \sum_{h\leq j \leq k} (2^{j\alpha} T_{c2^{-2j}}| \widetilde\varphi_j(\sqrt{\Ls}) f |)^q \bigg)^{1/q} \bigg( \sum_{h\leq j \leq k} 2^{jq'(m - \alpha)}\bigg)^{1/q'} 
\end{align*}
from which we deduce, by Proposition~\ref{contCRTN} and~\cite[Remark (1), Sec.\ 5]{GKKP1},
\begin{align*}
\bigg\| \sum_{h\leq j \leq k} X_J\varphi_j(\sqrt{\Ls}) f  \bigg\|_p 
& \lesssim \bigg( \sum_{h\leq j \leq k} 2^{jq'(m - \alpha)}\bigg)^{1/q'}  \bigg\| \Big( \sum_{j\in \Z} ( 2^{j\alpha} | \widetilde \varphi_j(\sqrt{\Ls}) f| )^q \Big)^{1/q}\bigg\|_p\\
& \lesssim  \bigg( \sum_{h\leq j \leq k} 2^{jq'(m - \alpha)}\bigg)^{1/q'}  \| [f]\|_{\dF^{p,q}_\alpha, \mathrm{LP}},
\end{align*}
which completes the proof. The cases $q=1$ or $q=\infty$ are analogous.
\end{proof}

\begin{proposition}\label{PropLP}
Suppose that $(S)$ holds, and that $\alpha \geq 0$ and $q\in [1,\infty]$.
\begin{itemize}
\item[\emph{(i)}]  If $p\in [1,\infty]$, then $\|[f]\|_{\dB_\alpha^{p,q} }  \approx \| [f]\|_{\dB^{p,q}_{\alpha},\mathrm{LP}}$.
\item[\emph{(ii)}]  If $p\in (1,\infty)$ and $q\in (1,\infty]$, then $ \|[f]\|_{\dF_\alpha^{p,q} } \approx \| [f]\|_{\dF^{p,q}_{\alpha},\mathrm{LP}}$.
\end{itemize}
\end{proposition}
\begin{proof}
The proof can be obtained with the same steps as those of \cite[Theorems 6.7 and 7.5]{KP}, so we will not provide all the details.

Let us briefly comment the proof of (ii). Suppose $[f]\in \Ss'/\Ps$ is such that $\| [f]\|_{\dF^{p,q}_{\alpha},\mathrm{LP}}$ is finite. Let $\eta(\alpha)$ be as in Proposition~\ref{convBTLLP}, and suppose $m\geq \max([\alpha/2]+1, \eta(\alpha))$. Then, by Proposition~\ref{convBTLLP} (whose notation we maintain)
\[
T_t \Ls^m(f- \rho_f) = \lim_{k\to \infty} \bigg(\sum_{j\geq -k}  T_t \Ls^m \varphi_j(\sqrt{\Ls}) f + T_t \Ls^m \rho_k\bigg) = \sum_{j\in \Z}T_t \Ls^m \varphi_j(\sqrt{\Ls}) f.
\]
Consider $\psi_j = \sqrt{\varphi_j}$. As in the proof of~\cite[Theorem 7.5]{KP}, then, one shows that
\[
\| t^{m-\alpha/2} T_t \Ls^m(f- \rho_f)\|_{L^p(L^q_+)} \lesssim \Big\| \Big( \sum\nolimits_{j\in \Z} ( 2^{j\alpha} | \psi_j(\sqrt{\Ls}) f| )^q \Big)^{1/q}\Big\|_p.
\]
This gives $ \|[f]\|_{\dF_\alpha^{p,q} } \lesssim \| [f]\|_{\dF^{p,q}_{\alpha},\mathrm{LP}}$, by the equivalences given in~\cite[Remark (1), Sec.\ 5]{GKKP1} and in Theorem~\ref{teo-equiv1}.

To obtain the converse inequality, notice that $\varphi_j(\sqrt{\Ls})f = \varphi_j(\sqrt{\Ls})(f+\rho)$ for all $\rho\in \Ps$; then one gets, as in the proof of~\cite[Theorem 7.5]{KP}, for $m=[\alpha/2]+1$
\[
\Big\| \Big( \sum\nolimits_{j\in \Z} ( 2^{j\alpha} | \varphi_j(\sqrt{\Ls}) f| )^q \Big)^{1/q}\Big\|_p \lesssim \| t^{m-\alpha/2} T_t \Ls^m(f+\rho)\|_{L^p(L^q_+)}, \qquad \rho\in \Ps,
\]
and this gives $ \| [f]\|_{\dF^{p,q}_{\alpha},\mathrm{LP}}\lesssim \|[f]\|_{\dF_\alpha^{p,q} } $.
\end{proof}

\subsection{Interpolation}

In this section, we show the complex interpolation properties of Triebel--Lizorkin and Besov spaces. For $1\leq p,q\leq \infty$ and $\alpha \geq 0$, we introduce the spaces
\[
L^p(\dell_\alpha^q) = \{ u=(u_j)_{j\in \Z} \colon   \|u\|_{L^p(\dell_\alpha^q)}\coloneqq \| 2^{-j\frac{\alpha}{2}} u_j \|_{ L^p(\dell^q)}<\infty\},
\]
and
\[
\dell_\alpha^q(L^p) = \{ u=(u_j)_{j\in \Z} \colon   \|u\|_{\dell_\alpha^q(L^p)}\coloneqq \| 2^{-j\frac{\alpha}{2}} u_j \|_{\dell^q(L^p)}<\infty\},
\]
where it is implicitly assumed that $u_j$ is a measurable function in $\Ss'$. Notice that the presence of $-j$ in the exponents above, rather than $j$ as e.g.\ in Proposition~\ref{PropLP}, is only due to our choice of splitting $(0,\infty)$ into dyadic intervals $[2^j, 2^{j+1})$ rather than $[2^{-j}, 2^{-j+1})$, $j\in \Z$. The following lemma will be instrumental in the following.

\begin{lemma}\label{lemmau0}
Assume that $(S)$ holds. Then for all $\alpha \geq 0$ there exists $\beta= \beta(\alpha) \in\N$ such that, if $u\in L^p(\dell_\alpha^q)$ with $p\in (1,\infty)$ and $q\in [1, \infty]$, or if $u\in \dell_\alpha^q(L^p)$ with $p,q\in [1,\infty]$, and if $m\geq 1$ is an integer, then there exists a sequence of polynomials $(\rho_k)_k \subset \Ps_\beta$ such that
\begin{equation}\label{limitu0}
u_0 = \lim_{k\to\infty} \bigg( \sum_{j\leq \log_2(k) } 2^{-jm} \int_{2^{j}}^{2^{j+1}} t^{2m}\Ls^m T_{t-2^{j-1}} u_j \, \frac{\dd t}{t}  +\rho_k\bigg)
\end{equation}
exists in $\Ss'$. For any other sequence $(\tilde \rho_k)$ for which the corresponding limit above $\tilde u_0$ exists, one has $[u_0] = [\tilde u_0 ]$.
\end{lemma}

\begin{proof}
The second statement is easily proved. If $(\rho_k)$ and $(\tilde \rho_k)$ are two sequences in $\Ps_\beta$ as in the statement, then by~\eqref{limitu0} one has $\Ls^\beta (u_0-\tilde u_0) =0$,  whence $[u_0] = [\tilde u_0]$.

To show the existence of a sequence of polynomials with the required properties, we assume $q\in (1,\infty)$ (the cases $q=1$ or $q=\infty$ being analogous) and we first prove that the sum for negative $j's$ is an element of $L^p$, hence of $\Ss'$. Indeed, since $2^{j-1} \leq t-2^{j-1} \leq 2^{j+1}$ for $t\in [2^j, 2^{j+1}]$, by Lemma~\ref{pointwiseestheat1}
\begin{multline*}
\bigg| \sum_{j\leq 0} 2^{-jm} \!\!\int_{2^{j}}^{2^{j+1}} \!\! \!\!  t^{2m}\Ls^m T_{t-2^{j-1}} u_j \, \frac{\dd t}{t} \bigg|
  \lesssim \sum_{j\leq 0}  T_{c2^j} |u_j|\\
  \leq  \bigg( \sum_{j\leq 0}  (2^{-j\frac{\alpha}{2} }T_{c2^j} |u_j|)^q\bigg)^{1/q} \bigg( \sum_{j\leq 0}  2^{jq'\frac{\alpha}{2} }\bigg)^{1/q'},
\end{multline*}
whence by Proposition~\ref{contCRTN}
\[
\bigg\| \sum_{j\leq 0} 2^{-jm} \!\!\int_{2^{j}}^{2^{j+1}} \! \!\! \!\! t^{2m}\Ls^m T_{t-2^{j-1}} u_j \, \frac{\dd t}{t} \bigg\|_p \lesssim\big\| 2^{-j\frac{\alpha}{2} }T_{c2^j} |u_j|\big\|_{L^p(\dell^q)} \lesssim  \|u\|_{L^p(\dell_\alpha^q)}.
\]
Suppose now $J\in \mathcal{I}^a$ with $a= [\alpha]+1$. For $h,\ell \geq 0$, again by Lemma~\ref{pointwiseestheat1}, by H\"older's inequality and Proposition~\ref{contCRTN} as before,
\begin{multline*}
\bigg\| X_J \sum_{h\leq j \leq \ell } 2^{-jm} \int_{2^{j}}^{2^{j+1}} \!\! \!\!  t^{2m}\Ls^m T_{t-2^{j-1}} u_j \, \frac{\dd t}{t} \bigg\|_p
  \lesssim \bigg\| \sum_{h\leq j \leq \ell }  2^{-aj/2 } T_{c2^j} |u_j|\bigg\|_p \\
  \lesssim \bigg( \sum_{h\leq j \leq \ell }  2^{jq'(\alpha-a)/2 }\bigg)^{1/q'}  \|u\|_{L^p(\dell_\alpha^q)}.
\end{multline*}
The conclusion now follows by~$(S)$.
\end{proof}

\begin{theorem}\label{interpolation}
Assume that $(S)$ holds, and suppose $\alpha_0,\alpha_1\geq 0$,  $\theta \in (0,1)$, $\alpha_\theta = (1-\theta)\alpha_0 + \theta \alpha_1$ and $q_0,q_1\in [1,\infty]$ with $\min(q_0,q_1)<\infty$. Given $ p_0,p_1 \in [1,\infty]$, define $\frac{1}{p_\theta} = \frac{1-\theta}{p_0} + \frac{\theta}{p_1}$ and $\frac{1}{q_\theta} = \frac{1-\theta}{q_0} + \frac{\theta}{q_1}$.
\begin{itemize}
\item[\emph{(i)}] If $p_0,p_1 \in [1,\infty]$, then  $(\dB^{p_0,q_0}_{\alpha_0}, \dB^{p_1,q_1}_{\alpha_1})_{[\theta]} = \dB^{p_\theta,q_\theta}_{\alpha_\theta}$. 
\item[\emph{(ii)}] If $p_0,p_1 \in (1,\infty)$, then  $(\dF^{p_0,q_0}_{\alpha_0}, \dF^{p_1,q_1}_{\alpha_1})_{[\theta]} = \dF^{p_\theta,q_\theta}_{\alpha_\theta}$.
\end{itemize}
\end{theorem}

\begin{proof}

To prove (ii), it is enough to prove that the spaces $\dF^{p,q}_\alpha$ are retracts of $L^p(\dell_\alpha^q)$. The result then follows by~\cite[Theorem 6.4.2]{BerghLofstrom} and the complex interpolation properties of the spaces $L^p(\dell_\alpha^q)$ (see~\cite[Theorem p.\ 128]{TriebelInterp} and~\cite[p.\ 121]{BerghLofstrom}). 

Fix $p\in (1,\infty)$, $q\in [1,\infty]$, $\alpha \geq 0$ and $m= \max([\alpha/2]+1, \beta(\alpha))$ where $\beta=\beta(\alpha)$ is that of Lemma~\ref{lemmau0}, and pick $\wp =\wp(f,\alpha,[\alpha/2]+1, p,q)$. Then define the functional $\mathfrak{I} \colon \dF^{p,q}_\alpha \to L^p(\dell_\alpha^q)$,
\[
(\mathfrak{I}\, [f])_j = 2^{jm} \Ls^m T_{2^{j-1}}  (f + \wp_f), \qquad j\in \Z.
\] 
The functional $\mathfrak{I}$ is well defined by Proposition~\ref{prop:minattained}. 

By Lemma~\ref{lemmau0}, for all $u\in L^p(\dell_\alpha^q)$ there is a sequence $(\rho_k)$ in $\Ps_\beta$ such that~\eqref{limitu0} exists in $\Ss'$, and the functional $\mathfrak{P} \colon L^p(\dell_\alpha^q) \to \dF^{p,q}_\alpha$ given by
\[
\mathfrak{P} u  =\bigg[ \frac{1}{(2m-1)!}\lim_{k\to \infty} \bigg(  \sum_{j\leq \log_2(k) } 2^{-jm} \int_{2^{j}}^{2^{j+1}} t^{2m}\Ls^m T_{t-2^{j-1}} u_j \, \frac{\dd t}{t}  +\rho_k\bigg) \bigg]
\]
is well defined. By Lemma~\ref{lemmau0} and~\eqref{LPD2},  moreover, $\mathfrak{P} \circ \mathfrak{I}\, [f]  = [f]$, that is $\mathfrak{P} \circ \mathfrak{I}  = \mathrm{Id}_{\dF^{p,q}_\alpha}$. Moreover, $\mathfrak{I} $ is bounded from $\dF^{p,q}_\alpha$ to $L^p(\dell_\alpha^q)$ by Theorem~\ref{teo-equiv2}. Thus, it remains to prove that $\mathfrak{P}$ is bounded from $L^p(\dell_\alpha^q)$ to $\dF^{p,q}_\alpha$. 

Observe first of all that 
\begin{multline*}
\Ls^m \lim_{k\to \infty} \bigg(  \sum_{j\leq \log_2(k) } 2^{-jm} \int_{2^{j}}^{2^{j+1}} t^{2m}\Ls^m T_{t-2^{j-1}} u_j \, \frac{\dd t}{t}  +\rho_k\bigg)  \\ = \sum_{j\in\Z } 2^{-jm} \int_{2^{j}}^{2^{j+1}} t^{2m}\Ls^{2m} T_{t-2^{j-1}} u_j \, \frac{\dd t}{t} .
\end{multline*}
By Theorem~\ref{teo-equiv2} then (by  estimating the infimum with the term when $\rho=0$)
\begin{align*}
\| \mathfrak{P} u\|_{\dF^{p,q}_\alpha} 
& \lesssim \bigg\|2^{k(m-\frac{\alpha}{2})}  \sum_{j\in\Z } 2^{-jm} \int_{2^{j}}^{2^{j+1}} t^{2m}\Ls^{2m} T_{t-2^{j-1}+2^k} u_j \, \frac{\dd t}{t}  \bigg\|_{L^p(\dell^q)}.
\end{align*}
By  Lemma~\ref{pointwiseestheat1}, we get
\begin{align*}
\bigg|2^{-jm} \int_{2^{j}}^{2^{j+1}} t^{2m}\Ls^{2m} T_{t-2^{j-1}+2^k} u_j \, \frac{\dd t}{t}\bigg|
&\lesssim 2^{-jm} \int_{2^{j}}^{2^{j+1}}  \frac{t^{2m}}{(2^{j}+2^{k})^{2m}} T_{c(2^{j}+2^{k})}| u_j |\, \frac{\dd t}{t} \\
& \lesssim  T_{c2^k} \left( \frac{2^{jm}}{(2^{j}+2^{k})^{2m}} T_{c2^{j}}| u_j| \right).
\end{align*}
We finally apply~Proposition~\ref{contCRTN},~\eqref{Schurdiscr} and again Proposition~\ref{contCRTN} to get
\begin{align*}
\| \mathfrak{P} u\|_{\dF^{p,q}_\alpha} 
& \lesssim \Big\| 2^{k(m-\frac{\alpha}{2})}  \sum_{j\in \Z}\frac{2^{jm}}{(2^{j}+2^{k})^{2m}} T_{c2^{j}}| u_j|\Big\|_{L^p(\dell^q)} \\
& \lesssim  \| 2^{-j\frac{\alpha}{2}} T_{c2^{j}}|u_j|  \|_{L^p(\dell^q)}\  \lesssim \| u\|_{L^p(\dell_\alpha^q )},
\end{align*}
which concludes the proof.
\end{proof}

\section{A glimpse at inhomogeneous spaces}\label{sec:inho}
In this section we briefly discuss an analogous theory to that developed so far, but for inhomogeneous spaces. These are indeed spaces of functions -- at least for strictly positive regularities, and the analogues of most of the results above can be obtained without substantial modifications in the arguments. We will not provide proofs, as these would go out of the scope of the present paper. The reader might look at~\cite{BPV1} for further insights. Let us first introduce the notation
\[
\|F\|_{L^p(L^q_o)} =\bigg\| \left( \int_0^1\!  |F(t, \cdot )|^q \, \frac{\dd t}{t} \right)^{1/q}\bigg\|_p, \qquad \|F\|_{L^q_o(L^p)} = \left( \int_0^1\!\left\| F(t, \cdot )\right\|_p^q  \, \frac{\dd t}{t} \right)^{\! 1/q},
\]
which is the inhomogeneous counterpart of that introduced in Subsection~\ref{Sub:not}. We shall also use the notation $X^{p,q}_o$ in the same spirit as above. Then, we have the following definition.
\begin{definition}
Suppose $\alpha \geq 0$, $m=[\alpha/2]+1$,  and $p,q\in [1,\infty]$. The Besov space $B_\alpha^{p,q} $ is the subspace of $\Ss' $ of distributions $f$ such that
\begin{equation*}
 \|f\|_{B_\alpha^{p,q} } \coloneqq \|T_{1} f\|_{p} + \| t^{m-\alpha/2} \Ls^m T_t f\|_{L^q_o(L^p)}<\infty,
\end{equation*}
while the Triebel--Lizorkin space $F_\alpha^{p,q} $ is the subspace of $\Ss' $ of distributions $f$ such that
\begin{equation*}
\|f\|_{F_\alpha^{p,q} } \coloneqq  \|T_{1} f\|_{p} + \| t^{m-\alpha/2} \Ls^m T_t f\|_{L^p(L^q_o)}<\infty.
\end{equation*}
\end{definition}
By means of~\eqref{LPDlocal}, and by analogous results to those of Subsection~\ref{sec:consheat} for the spaces $L^p(L^q_o)$, one can show equivalent characterizations as those in Subsection~\ref{sub:normchar}, provided $X^{p,q}_+$ is replaced by $X^{p,q}_o$, $\alpha>0$ and the $L^p$ norm of  $f$ is added. In particular, one can see that if $\alpha>0$, $\|T_{1} f\|_{p} $ can be replaced by $\|f\|_p$ for both $B$- and $F$-spaces. The equivalence of the above norms with those defined by a Littlewood--Paley decomposition was proved in~\cite[Theorems 6.7 and 7.5]{KP}. Also embeddings and interpolation, see~\cite{KP, HH}, and algebra properties have inhomogeneous counterparts when $\alpha>0$. Most of these results might be obtained assuming that the heat kernel estimates~\eqref{dtHtestimate} and~\eqref{XHtestimate} hold only for small times. It is also worth mentioning that some additional result can be proved in the inhomogeneous case, like embeddings in $L^\infty$.

Let us denote by $L^p_\alpha$ the subspace of $L^p$ such that
\[
\|f\|_{L^p_\alpha}\coloneqq \|(I +\Ls)^{\alpha/2}f\|_{p} <\infty, \qquad 1<p<\infty, \; \alpha \geq 0. 
\]
Observe that by the boundedness of the operator $(I+ \Ls)^{-\beta}$ on $L^p$ for every $1<p<\infty$ and $\beta\geq 0$ (see~\cite[Section 5]{Komatsu}), one has the embedding $L^p_{\alpha_1}\subseteq L^p_{\alpha_2}$ for $\alpha_1>\alpha_2\geq 0$. Observe moreover that by~\cite[Theorem II.2.7]{VSCC} (applied to $\e^{-t}T_{t}$, which is $L^1$-$L^\infty$ ultracontractive with norm $\lesssim t^{-\dm/2}$, see the discussion at the beginning of Section~\ref{sec:embed}) one has the Sobolev embeddings $L^p_\alpha \subseteq L^q$ for $1/p-1/q= \alpha / \dm$ for $1< p,q <\infty$ and $\alpha>0$. By means of~\eqref{normHt} with $n=0$, one gets $\|(I+\Ls)^{-\alpha /2} f\|_{\infty} \lesssim \|f\|_p$, hence $L^p_\alpha \subseteq L^\infty$, if $p\in (1,\infty)$ and $\alpha p >\dm$. 

Similar embeddings hold for $B$- and $F$- spaces. First, if $0<\beta<\alpha$ and $p,q,r\in [1,\infty]$, then  $B_\alpha^{p,q} \subseteq B_{\beta}^{p,r}$ and $F_\alpha^{p,q} \subseteq F_{\beta}^{p,r}$, as a consequence of H\"older's inequality. This and Littlewood--Paley--Stein theory~\cite{Meda, Stein} lead to 
\[
F^{p,q}_\alpha \subseteq F^{p,2}_\beta = L^p_\beta \subseteq L^\infty\qquad (\alpha>\beta>\dm/p)
\]
when $p\in (1,\infty)$, $q\in [1,\infty]$ and $\alpha p>\dm$. The analogue of  Theorem~\ref{teo_embeddings}~(3) for inhomogeneous $B$-spaces, instead, gives 
\[
B^{p,q}_\alpha  \subseteq B^{p,1}_{\dm/p} \subseteq B_0^{\infty,1}\subseteq L^\infty
\]
when $p,q\in [1,\infty]$ and $\alpha p>\dm$, the last inclusion by~\eqref{LPDlocal}.

By combining the analogue of Theorem~\ref{teo_algebra} for inhomogeneous spaces with the above embeddings, one obtains the algebra properties under pointwise multiplication of inhomogeneous spaces: if $q,p\in [1,\infty]$ and $\alpha>0$, then $B^{p,q}_\alpha\cap L^\infty$ is an algebra, and in particular, $B^{p,1}_{\dm/p}$ and $B^{p,q}_\alpha$ are algebras for $\alpha>\dm/p$; if $q\in [1,\infty]$, $p\in (1,\infty)$ and $\alpha>0$, then $F^{p,q}_\alpha\cap L^\infty$ is an algebra, and in particular $F^{p,q}_\alpha$ is an algebra for $\alpha>\dm/p$.

\section{Examples}\label{sec:example}
In this final section we discuss the two main examples where our results apply: nilpotent Lie groups and the Grushin setting. While heat kernel estimates on nilpotent Lie groups are well known, to the best of our knowledge those in the Grushin setting are not, and we prove them below.

\subsection{Nilpotent Lie groups} Let $G$ be a connected and simply connected nilpotent Lie group, and consider a system of linearly independent left-invariant vector fields $\mathfrak{X} = \{ X_1,\dots,X_\kappa\}$ satisfying H\"ormander's condition. Let $\Ls=\Delta= -\sum_{j=1}^\kappa X_j^2 $ be the associated sum-of-squares sub-Laplacian. It is well known, see e.g.~\cite{VSCC}, that $G$ endowed with the control distance associated with $\mathfrak{X}$ and the (left- and right-invariant) Haar measure is a doubling metric measure space satisfying all the assumptions of the paper. Its heat kernel is smooth and satisfies all the required bounds~\cite[Theorems IV.4.2, IV.4.3]{VSCC}. Its H\"older continuity can also be seen via~\cite[Theorem 5.11 and Corollary 7.6]{GT}. Condition~\eqref{strongnoncollaps} is proved in~\cite[Proposition IV.5.6]{VSCC}.

We complete the system $\mathfrak{X}$ to a basis $\{X_1, \dots, X_\eta\}$ of its Lie algebra, which we identify with the algebra of all left-invariant vector fields on $G$, and identify $G$ with $\R^\eta$ via the exponential map. Under this identification, the space $\Ss$ defined in~\eqref{defS} (and equivalent to $\tilde \Ss$) is the classical Schwartz space on $G$, that is $\Ss(\Ls) = \Ss(G) = \Ss(\R^\eta)$. Moreover,
\begin{equation}\label{Xtopartial1}
X_j = \sum_{k=1}^\eta p_{k,j }\partial_{k}, \qquad \partial_j = \sum_{k=1}^\eta \tilde p_{k,j }X_{k}, \qquad j=1,\dots ,\eta,
\end{equation}
for certain polynomials, in the Euclidean sense, $p_{k,j }$ and $\tilde p_{k,j }$. For notational convenience, let us write $\mathcal{N}= \{ 1,\dots, \eta\}$, and recall also $\mathcal{I} = \{ 1,\dots, \kappa\}$.  Since the family $\mathfrak{X}$ is a H\"ormander basis, there are $a\in \N$ and polynomials $\tilde p_{J,j }$ such that
\begin{equation}\label{partialtoX1}
\partial_j = \sum_{k=1}^a \sum_{J \in \mathcal{I}^k}  \tilde p_{J,j }X_{J}, \qquad j=1,\dots ,\eta.
\end{equation}

We plan now to show that condition $(S)$ holds, and to do that we shall need the following lemma. As for the $X_I$'s, by $\partial_J$ with $J\in \mathcal{N}^\ell$ we shall mean $\partial_{J_1}\cdots \partial_{J_\ell}$.

\begin{lemma}\label{lemmader}
For all $n\in\N$ there exists $\ell= \ell(n) \in \N$ such that for all $I \in \mathcal{N}^\ell$
\begin{equation}\label{partialtoX}
\partial_I = \sum_{ k\geq n } \sum_{J \in \mathcal{I}^k}  \tilde p_{J,I }X_{J},
\end{equation}
and
\begin{equation}\label{Xtopartial}
X_I = \sum_{ k\geq n } \sum_{J \in \mathcal{N}^k} p_{J,I }\partial _{J}, 
\end{equation}
for certain polynomials, in the Euclidean sense, $\tilde p_{J,I }$ and $p_{J,I }$.
\end{lemma}
\begin{proof}
We prove both~\eqref{partialtoX} and~\eqref{Xtopartial} by induction, and observe that both hold for $n=1$ (and $\ell=1$) by~\eqref{partialtoX1} and~\eqref{Xtopartial1} respectively. 

Assume that~\eqref{partialtoX} holds for $n-1$, i.e.\ there exists $\ell'$ such that for all $I' \in \mathcal{N}^{\ell'}$
\[
\partial_{I'} = \sum_{ k\geq n-1 } \sum_{J \in \mathcal{I}^k}  \tilde p_{J,I' }X_{J}.
\]
Let $m-1$ be the maximum degree of the polynomials $\tilde p_{J,I' }$ above. Then $\partial_{L} \tilde p_{J,I' } =0$ for all $L \in \mathcal{N}^{m}$ and all $J$ and $I'$. Therefore, if $I \in \mathcal{N}^{m+ \ell'}$ is such that $I=(L,I')$ with $L \in \mathcal{N}^{m}$ and $I' \in \mathcal{N}^{\ell'}$,
\[
\partial_I = \partial_{L}\partial_{I'} = \sum_{ k\geq n-1 } \sum_{J \in \mathcal{I}^k}  \partial_L(\tilde p_{J,I' }X_{J}).
\] 
It remains to observe that since $\partial_L \tilde p_{J,I' } =0$ and by~\eqref{partialtoX1}, one has
\[
 \partial_L(\tilde p_{J,I' }X_{J}) = \sum_{u\geq 1} \sum_{U \in \mathcal{I}^u} \tilde p_{J,I', U}X_U X_J
\]
for certain polynomials $\tilde p_{J,I', U}$. This implies~\eqref{partialtoX}.

Assume now that~\eqref{Xtopartial} holds for $n-1$, i.e.\ there exists $\ell'$ such that for all $I' \in \mathcal{N}^{\ell'}$
\[
X_{I'} = \sum_{ k\geq n-1 } \sum_{J \in \mathcal{N}^k}  p_{J,I' }\partial_{J}.
\]
By~\cite[Corollary 1.4]{AntonelliLeDonne}, there is $m\in \N$ such that $X_L p_{J,I' }=0$ for all $L\in \mathcal{N}^m$, $J$ and $I'$ (e.\ g., $m$ can be chosen as the highest degree \emph{à la Leibman} of the polynomials $p_{J,I' }$). Therefore, if $I \in \mathcal{N}^{m+ \ell'}$ is such that $I=(L,I')$ with $L \in \mathcal{N}^{m}$ and $I' \in \mathcal{N}^{\ell'}$,
\[
X_I = X_L X_{I'} = \sum_{ k\geq n -1} \sum_{J \in \mathcal{N}^k} X_L(p_{J,I' }\partial_{J}).
\] 
Since $X_{L} p_{J,I' } =0$ and $X_L$ does not have zero order terms, cf.~\eqref{Xtopartial1}, 
\[
X_L(p_{J,I' }\partial_{J}) =  \sum_{u\geq 1} \sum_{U \in \mathcal{N}^u} p_{J,I', U} \partial_U \partial_J,
\]
for certain polynomials $p_{J,I', U}$, and~\eqref{Xtopartial} follows.
\end{proof}
We are now in a position to show that property $(S)$ holds.

\begin{lemma}\label{Sgroup}
For all $n\in \N$ there exists $\nu \in \N$ such that, for all sequences $(f_k)_{k\in \N}$ in $\Ss'$ such that $(X_J f_k)_{k\in \N}$ converges in $\Ss'$ for all $J\in \mathcal{I}^n$, there exists a sequence of polynomials $(\rho_k)_{k\in \N}$ in $\Ps_\nu$ such that $(f_k + \rho_k)_{k\in \N}$ converges in $\Ss'$.
\end{lemma}
\begin{proof}
Pick $n\in \N$ and assume that $(f_k)$ is a sequence in $\Ss'$ such that $(X_J f_k)$ converges in $\Ss'$ for all $J\in \mathcal{I}^n$. By Lemma~\ref{lemmader},~\eqref{partialtoX}, there is $\ell \in\N$ such that $(\partial_I f_k)$ converges in $\Ss'$ for all $I\in \mathcal{N}^{\ell}$. By~\cite[Proposition 2.7]{Bownik}, we deduce the existence of polynomials $(\rho_k)$ of degree, in the Euclidean sense, at most $\ell-1$ such that $(f_k+\rho_k)_k$ converges in $\Ss'$. In particular, $\partial_I \rho_k =0$ for all $I\in \mathcal{N}^\ell$ and $k\in\N$. To conclude, observe that for all $\nu\in\N$ there exist $(c_J^\nu) \in \C$ such that
\[
\Delta^\nu = \sum_{J \in \mathcal{I}^{2\nu}} c_J^\nu X_J.
\]
By~\eqref{Xtopartial}, there exists $\nu\in\N $ such that $\Delta^\nu \rho_k =0$ for all $k$, whence $(\rho_k)\subset \Ps_\nu$.
\end{proof}

We finally show that $\Ps$ coincides with the space of polynomials on $\mathbb{R}^\eta$. Let us first consider the case when $G$ is a stratified group and $\mathfrak{X}$ is a basis of the first layer of its Lie algebra. Since for every $k\in \N$ the operator $\Delta^k$ is  left-invariant, hypoelliptic and homogeneous with respect to the natural dilations of $G$, by a theorem of Geller~\cite[Theorem 2]{Geller} every element $\rho \in \Ps$ is a polynomial on $\R^d$ in exponential chart. It remains to prove that such result can be extended to all nilpotent $G$'s. We do this in the following proposition, which was kindly shown to us by M.\ Calzi. It might actually work for more general operators, but we limit here to the case under consideration.

\begin{proposition}
If $\rho \in \Ss'(G)$ is such that $\Delta^k \rho =0$ for some $k\in \N$, then $\rho$ is a polynomial on $G$ in exponential chart. 
\end{proposition}

\begin{proof}
Let $G'$ be a simply connected free nilpotent Lie group of the same nilpotency step as $G$, whose Lie algebra $\mathfrak{g}'$ is generated by the vector fields $Y_1,\dots, Y_\kappa$. Then $\Delta'= - \sum_{j=1}^\kappa Y_j^2$ is a Rockland operator on $G'$ with respect to the gradation which assigns degree $1$ to each $Y_j$, and so is $L' = (\Delta')^k$. Consider the continuous, surjective homomorphism $\pi \colon G' \to G$ such that $d\pi(Y_j)=X_j$, where we recall that for $Y\in \mathfrak{g}'$, $\phi\in C^\infty(G)$ and $x\in G$
\[
d\pi(Y)\phi(x) = \lim_{t\to 0} \frac{\phi (x \pi (\exp'(tY))) - \phi(x) }{t},
\]
$\exp'$ being the exponential map $\mathfrak{g}' \to G'$. Observe that $Y(\phi \circ \pi) = (d\pi(Y) \phi)\circ \pi$. Since $d\pi$ is a homomorphism of Lie algebras, it can be extended to the enveloping algebra of $G'$. Moreover, $\pi$ induces a linear mapping on $\Ss(G')$ defined as
	\[
	\langle \pi_*(g), \phi\rangle= \langle g, \phi\circ \pi\rangle, \qquad g\in \Ss(G'), \; \phi\in C_c^\infty(G),
	\]
with the property that for $Y\in \mathfrak{g}'$ and $g\in \Ss(G')$
\begin{equation}\label{pistardiff}
	\pi_*(Y g)=\dd \pi(Y) \pi_*(g),
\end{equation}	
see~\cite[Proposition 1.90]{Calzi}, and hence $\pi_*\colon C_c^\infty(G') \to C_c^\infty(G) $ is continuous. We claim that (a) $\pi_* \colon \Ss(G')\to \Ss(G)$ is continuous, and that (b) $\pi_*\colon C_c^\infty(G') \to C_c^\infty(G) $ is surjective. These imply that its transpose $\pi_*^t$ is a continuous injective linear mapping $\Ss'(G)\to \Ss'(G')$. Assuming the claims, then,
	\[
	L'\pi_*^t(\rho)=  \pi_*^t(\Delta^k\rho)=0,
	\]
	so that $\pi_*^t \rho$ is a polynomial on $G'$ by the aforementioned~\cite[Theorem 2]{Geller}. Then, there is $m\in\N$ such that $Y^m \pi_*^t(\rho)=0$ for every $Y\in \mathfrak{g}'$, so that
	\[
\pi_*^t(\dd \pi(Y)^m \rho)=0,
	\]
	and since $\pi_*^t$ is injective this implies $\dd \pi(Y)^m \rho=0$ for every $Y\in \mathfrak{g}'$. Since $\dd \pi\colon \mathfrak{g}'\to \mathfrak{g}$ is surjective,  $\rho$ is a polynomial on $G$ in exponential chart by~\cite[Theorem 1.3]{AntonelliLeDonne}.
	
It remains to prove the claims. To prove (a), it is enough to notice that if $d'$ is the control distance on $G'$ associated to the vector fields $Y_1, \dots, Y_\kappa$, and $e'$ is the identity of $G'$, then $d(\pi(y), e) \leq  d'(y,e')$ for all $y\in G'$. The continuity  $\pi_* \colon \Ss(G')\to \Ss(G)$ now easily follows from this property and~\eqref{pistardiff}.

To prove (b), we first observe that any compact $K$ in $G$ is the image of a compact $K'$ in $G'$. Indeed, for $x\in K$ choose $y\in \pi^{-1}(x)$ and let $V'_{y}$ be a compact neighbourhood of $y$ in $G'$. Since $\pi$ is open, $\pi(V'_{y})$ is a compact neighbourhood of $x$, thus $K$ can be covered with neighbourhoods $\pi(V'_{y_1}),\dots, \pi(V'_{y_s})$, for certain $y_1,\dots y_s\in G'$. Then $V'\coloneqq \bigcup_j V'_{y_j}$ is a compact in $G$ whose image covers $K$, hence $K= \pi(K')$ with  $K'=V'\cap \pi^{-1}(K)$. 

Let now $ \phi \in C_c^\infty(G)$ be supported in a compact $K$, and let $K'$ be a compact in $G'$ such that $\pi(K')=K$. Take a smooth cut-off function $\chi$ which is $1$ on $K'$, so that $\pi_*(\chi)(\pi(y))>0$ for $y\in K'$ by~\cite[Proposition 1.92]{Calzi}.  Then, the function $g= \frac{\phi'\circ \pi}{\pi_*(\chi)\circ \pi} $ is smooth on an open neighbourhood of $K'$ and it vanishes  on the complement of $K'$. Hence, it can be extended to a smooth function on $G'$. Define $\phi' \coloneqq \chi g$, and observe that it is smooth, with compact support, and $\pi_*( \phi ' )=\phi$.
\end{proof}

\subsection{Grushin operators} For $\x = (x',x'') \in \R^{n+1} $ with $x'\in \R^n$ and $x''\in \R$, we define the Grushin operator
\[
\Ls  = -\Delta_{x'}  - |x'|^2 \partial_{x''}^2, \qquad \Delta_{x'} = \sum\nolimits_{j=1}^n \partial_{x'_j}^2.
\]
The operator $\Ls$ is hypoelliptic and homogeneous of degree 2 with respect to the anisotropic dilations
\[
\delta_r (\x) = (rx', r^2x''), \qquad r>0.
\]
In other words, $\Ls(f \circ \delta_r) = r^2 (\Ls f) \circ \delta_r$. If
\begin{equation}\label{Grushinvectors}
 \X_{j} = \partial_{x'_j}  \qquad \X_{n+j} = x'_j \partial_{x''}, \quad j=1,\dots, n,
\end{equation}
and $\mathfrak X = \{\X_1, \dots, \X_{2n} \}$, then
\[
\Ls = -\sum\nolimits_{j=1}^{2n} \X_j^2,
\]
and $[\X_j, \X_{n+j}] = \partial_{x''}$. It is easily seen that $\Ls$ is symmetric and nonnegative on $L^2(\R^{n+1})$; it is indeed essentially self-adjoint. We endow $\R^{n+1}$ with the control distance associated with $\Ls$, see e.g.~\cite{RobinsonSikora1}, which is homogeneous, i.e.\ $d(\delta_r \x, \delta_r \y) = r d(\x, \y)$, and such that
\[
d(\x, \y) \approx |x'-y'| + \min \left(\frac{|x''-y''|}{|x'|+|y'|} ,  |x''-y''|^{1/2} \right) .
\]
Moreover, if $|\cdot|$ denotes the Lebesgue measure,
\[
|B(\x,r)| \approx r^{n+1} (r+|x'|),
\]
so that
\[
\left(\frac{R}{r}\right)^{n+1} \lesssim \frac{|B(\x,R)|}{|B(\x,r)|} \lesssim \left(\frac{R}{r}\right)^{n+2} \qquad R\geq r>0,
\]
and the measure space $(\R^{n+1}, d, \mathrm{Leb})$ is doubling and satisfies all the assumptions of the paper. We write $\mathcal{N}= \{ 1,\dots, n+1\}$, and observe that  $\mathcal{I}= \{ 1,\dots, 2n\}$. Then, we have the following result.

\begin{lemma}\label{SGrushin}
$\Ss(\Ls) = \Ss(\R^{n+1})$ with equivalence of norms.
\end{lemma}
\begin{proof}
Observe first that
\begin{align}\label{distGrusEucl}
(1+d(\x,0))^k  \lesssim (1+|x'| + |x''|^{1/2})^k \lesssim (1+|x'|+|x''|)^{k} \lesssim (1+d(\x,0))^{2k}.
\end{align}
By the  first inequality in~\eqref{distGrusEucl} and
\[
|\bX_J \phi(\x)| \lesssim (1+|x'|+|x''|)^{|J|} \max_{0\leq |I|\leq |J|, \; 0\leq h\leq |J|}|\partial_{x'}^I \partial_{x''}^h \phi(\x)|,
\]
one gets that $\Ss(\R^{n+1}) \subseteq \Ss(\Ls)$. By the second inequality in~\eqref{distGrusEucl} and the fact that $\partial_{x''}^h = [\X_j, \X_{n+j}]^h$, one gets $\Ss(\Ls) \subseteq \Ss(\R^{n+1}) $ and hence the equality.
\end{proof}
Lemma~\ref{SGrushin}, together with~\cite[Theorem 1]{Xuebo}, implies that $\Ps$ is the space of polynomials on $\R^{n+1}$.  It also allows us to show the validity of $(S)$.

\begin{lemma}
Suppose $\nu\in\N$. For all sequences $(f_k)_{k\in\N}$ in $\Ss'$ such that $(\X_J f_k)_{k\in \N}$ converges in $\Ss'$ for all $J\in \mathcal{I}^\nu$, there exists a sequence of polynomials $(\rho_k)_{k\in \N}$ in $\Ps_\nu$ such that $(f_k + \rho_k)_{k\in \N}$ converges in $\Ss'$.
\end{lemma}
\begin{proof}
Pick $\nu\in \N$ and assume that $(f_k)$ is a sequence in $\Ss'$ such that $(\X_J f_k)$ converges in $\Ss'$ for all $J\in \mathcal{I}^\nu$. By~\eqref{Grushinvectors} and the fact that $\partial_{x''} = [\X_j, \X_{n+j}]$, we see that $(\partial_J f_k)$ converges in $\Ss'$ for all $J\in \mathcal{N}^\nu$.  By~\cite[Proposition 2.7]{Bownik}, there exist polynomials $(\rho_k)$ of degree (in the Euclidean sense) at most $\nu-1$ such that $(f_k+\rho_k)$ converges in $\Ss'$. In particular, $\partial_I \rho_k =0$ for all $I\in \mathcal{N}^\nu$ and $k\in\N$. Since
\[
\Ls^\nu=\sum_{j=\nu}^{2\nu} \sum_{J \in \mathcal{N}^j} p_J^\nu \partial_J
\]
for certain polynomials $p_J^\nu$, one gets $\Ls^{\nu} \rho_k =0$ for all $k\in\N$, whence $(\rho_k)\in \Ps_\nu$.
\end{proof}
We now exploit the relation of the Grushin setting with the Heisenberg group $\mathbb{H}^n \simeq \R^{n} \times \R^n\times \R$, whose Lie algebra we denote with $\mathfrak{h}^n$ and whose homogeneous dimension with $Q=2n+2$. We refer the reader to~\cite{DziubanskiJotsaroop} for more details. Its group law is
\[
(x_1, y_1, t_1) (x_2, y_2, t_2) = (x_1+x_2, y_1+y_2, t_1 + t_2 + \tfrac{1}{2}(y_1x_2 -x_1y_2)).
\]
We also define the left-invariant vector fields on $\mathbb{H}^n$
\[
X_j = \partial_{x_j} + \frac{1}{2}y_j \partial_t, \qquad X_{n+j} = \partial_{y_j} - \frac{1}{2}x_j \partial_t, \qquad j=1,\dots, n,
\]
and the right-invariant ones
\[
Y_j = \partial_{x_j} - \frac{1}{2}y_j \partial_t, \qquad Y_{n+j} = \partial_{y_j} + \frac{1}{2}x_j \partial_t, \qquad j=1,\dots, n.
\]
We denote by $\Delta = -\sum_{j=1}^n (X_j^2 + X_{n+j}^2)$ the ordinary left-invariant sub-Laplacian on $\mathbb{H}^n$.  Let $\pi$ be the unitary representation of  $\mathbb{H}^n$ on $L^2(\R^{n+1})$ given by
\[
\pi(x,y,t) f(\x) = f(x'+y, x'' + t + \tfrac{1}{2} x \cdot y + x\cdot x').
\]
For $Z \in \mathfrak{h}^n$, let $d\pi (Z) f=  \frac{d}{dt}\Big|_{t=0} (\pi(\exp(tZ))f )$, extend $d\pi$ to the enveloping algebra of $\mathbb{H}^n$, and observe that for $j=1,\dots, n$
\begin{equation}\label{dpi}
d\pi (X_j) = \X_{n+j}, \quad d\pi (X_{n+j}) = \X_{j},\qquad d\pi(\Delta)=\Ls.
\end{equation}
For a given $F\in L^1(\mathbb{H}^n)$, let 
\[
\pi(F)(\x,\y) = \int_{\R^n} F(z, y'-x', y''-x'' - \tfrac{1}{2}z \cdot (y'+x'))\, dz.
\]
Then, for every $j=1,\dots, 2n$,
\begin{equation}\label{dpipi}
d\pi (X_j)_{\x} \pi(F)(\x,\y) = - \pi(Y_j F)(\x,\y), \qquad d\pi (X_j)_{\y} \pi(F)(\x,\y) = \pi(X_j F)(\x,\y).
\end{equation}
A key observation moreover is that, if $h_t$ is the convolution kernel of $\e^{-t\Delta}$, and $H_t$ is the integral kernel of $\e^{-t\Ls}=T_t$, then 
\[
\pi(h_t)(\x,\y) = H_t(\x,\y).
\]
The H\"older continuity of $H_t$ is proved in~\cite[Corollary 2.5]{DziubanskiJotsaroop}.

\begin{proposition}
The following properties hold:
\begin{itemize}
\item[\emph{(1)}] $(T_t)_{t>0}$ is a diffusion semigroup on $(\R^{n+1}, \mathrm{Leb})$ and $T_t1 =1$;
\item[\emph{(2)}] there exist two constants $b_0, b_1>0$ such that 
\[
 |B(\x,\sqrt{t})|^{-1}\, \e^{-b_0 d(\x,\y)^2/t} \lesssim H_t(\x,\y) \lesssim  |B(\x,\sqrt{t})|^{-1}\, \e^{-b_1 d(\x,\y)^2/t} 
\]
 for every $t>0$ and $\x,\y \in \R^{n+1}$;
\item[\emph{(3)}]  for every $h,k\in \N$ there exists a positive constant $a=a_{h,k}$ such that 
\[|(\X_{I})_{\x} (\X_{J})_{\y} H_t(\x,\y)| \lesssim t^{-\frac{h+k}{2}} |B(\x,\sqrt{t})|^{-1}\, \e^{-a  d(\x,\y)^2/t} 
\]
for all $t>0$, $\x,\y\in \R^{n+1}$, and  $I\in \mathcal{I}^h$, $J\in \mathcal{I}^k$.
\end{itemize}
\end{proposition}

\begin{proof}
To prove~(1), observe that $T_t$ is positivity preserving, contractive on $L^2$ and symmetric. Thus, by~\cite[Theorem 1.4.1]{Davies} it can be extended to a contraction semigroup on $L^p$, $1\leq p\leq \infty$, strongly continuous when $1\leq p<\infty$. The property $T_t1 =1$ is~\cite[Theorem 6.1]{RobinsonSikora1}.

Statement (2) is nothing but~\cite[Theorem~1.2]{RobinsonSikora2}.

The proof of~(3) owes much to the proof of~\cite[Proposition 2.3]{DziubanskiJotsaroop}.  For $I\in \mathcal{I}^h$,  define an associated $\tilde I \in \mathcal{I}^h$ with the property that $\tilde I_j = I_j -n $ if $I_j \in \{n+1,\dots , 2n\}$ and $\tilde I_j = I_j +n $ if $I_j \in \{1,\dots , n\}$, and similarly for $J$. Then, by~\eqref{dpi} and~\eqref{dpipi}
\begin{align*}
(\X_{I})_{\x} (\X_{J})_{\y} H_t(\x,\y) 
&= d\pi (X_{\tilde I_1})_{\x} \cdots d\pi (X_{\tilde I_h})_{\x} \, d\pi (X_{\tilde J_1})_{\y} \cdots d\pi (X_{\tilde J_k})_{\y}\, \pi(h_t)(\x,\y)\\
& = (-1)^{h} \pi ( Y_{ \tilde I} X_{\tilde J} h_t)(\x, \y).
\end{align*}
Recall now that by~\cite[Theorem~IV.4.2]{VSCC} there exists $a>0$ such that
\[
|Y_{\tilde I}  X_{\tilde J} h_t({\bf x})| \lesssim t^{-(h+k+ Q)/2} \e^{-ad_c({\bf x},e)^2/t}, \qquad {\bf x} \in \mathbb{H}^n, \; t>0.
\]
Here $d_c$ is the Carnot--Carathéodory distance on $\mathbb{H}^n$ and $e$ is its identity. If  ${\bf x} = (x_1,y_1,t_1)$ and $\delta_r({\bf x}) = (rx_1,r y_1,r^2 t_1)$ is the anisotropic dilation on $\mathbb{H}^n$, then there is $\alpha>0$ such that
\begin{align*}
|Y_{\tilde I} X_{\tilde J} h_t({\bf x})|
\lesssim t^{-(h+k+ Q)/2}  \sum_{j\geq 0}\e^{-\alpha j^2} \mathbf{1}_{B_{\mathbb{H}^n}(e, \sqrt{t}j)} ({\bf x}),
\end{align*}
so that
\[
|\pi(Y_{\tilde I} X_{\tilde J} h_t)| \leq \pi(|Y_{\tilde I} X_{\tilde J}  h_t|)   \lesssim t^{-(Q+h+k)/2} \sum_{j\in\Z}\e^{-\alpha j^2} \pi( \mathbf{1}_{B_{\mathbb{H}^n}(e, \sqrt{t}j)}).
\]
By combining~\cite[Lemma~2.1 and Lemma~2.2]{DziubanskiJotsaroop}, there exists $C_1>0$ such that
\[
|\pi( \mathbf{1}_{B_{\mathbb{H}^n}(e, \sqrt{t}j)})(\x,\y)| \lesssim (\sqrt{t}j)^{Q} |B(\x, \sqrt{t}j)|^{-1} {\bf 1}_{B(\x, C_1 \sqrt{t} j)}(\y).
\]
Hence
\begin{align*}
|(\X_{I})_{\x} (\X_{J})_{\y} H_t(\x,\y) | 
& \lesssim  t^{-(h+k)/2} \sum_{j \geq 0}\e^{-\alpha j^2} j^Q |B(\x, \sqrt{t}j)|^{-1} {\bf 1}_{B(\x, C_1 \sqrt{t} j)}(\y)\\
& \lesssim  t^{-(h+k)/2} |B(\x, \sqrt{t})|^{-1} \sum_{j\geq 0}\e^{-\alpha j^2} j^{Q-n-1}  {\bf 1}_{B(\x, C_1 \sqrt{t} j)}(\y)\\
&\lesssim  t^{-(h+k)/2} |B(\x, \sqrt{t})|^{-1} \e^{-\alpha' d(\x,\y)^2/t},
\end{align*}
for some $\alpha'>0$. The statement follows.
\end{proof}

\subsection*{Acknowledgements} It is my pleasure to thank The Anh Bui, Mattia Calzi, Guorong Hu, Marco M.\ Peloso and Maria Vallarino for inspiring conversations.

\end{document}